\documentclass[a4wide,11pt]{amsart}
\usepackage{amsmath, amsthm, amscd, amssymb}
\usepackage{tikz,cite}
\usepackage{caption,subcaption}
\usepackage{adjustbox,multirow,bm,array,makecell}
\usepackage{tabularx, booktabs,ragged2e}
\usepackage{hyperref}

\numberwithin{figure}{section}
\numberwithin{equation}{section}
\theoremstyle{plain}
\newtheorem{theorem}{Theorem}[section]
\newtheorem{proposition}[theorem]{Proposition}
\newtheorem{lemma}[theorem]{Lemma}
\newtheorem{corollary}[theorem]{Corollary}
\newtheorem{claim}[theorem]{Claim}

\theoremstyle{definition}
\newtheorem{definition}[theorem]{Definition}
\newtheorem{example}[theorem]{Example}
\newtheorem{question}[theorem]{Question}
\newtheorem*{remark}{Remark}

\def\Z{\mathbb{Z}}
\def\cC{\mathcal{C}}
\newcommand{\R}{\mathbb{R}}
\newcommand{\C}{\mathbb{C}}
\newcommand{\Q}{\mathbb{Q}}
\def\blue#1{\textcolor{blue}{#1}}
\def\red#1{\textcolor{red}{#1}}

\tikzstyle{v}=[circle, draw, solid, fill=black, inner sep=0pt, minimum width=4pt]
\def\newprec#1{\prec_{\tiny\text{\rm lex}}^{#1}}
\def\newpreceq#1{\preceq_{\tiny\text{\rm lex}}^{#1}}
\def\atomprec#1{\prec_{\tiny\text{\rm atm}}^{#1}}
\def\pP#1{\mathcal{P}_{#1}^{\mathrm{even}}}

\begin{document}
\title{On shellability for a poset of even subgraphs of a graph}

\author[B. Park]{Boram Park}
\address[B. Park]{Department of Mathematics, Ajou University, Suwon, 443-749, Republic of Korea.}
\email{borampark@ajou.ac.kr}

\author[S. Park]{Seonjeong Park}
\address[S. Park]{Department of Mathematics, Osaka City University, Sumiyoshi-ku, Osaka 558-8585, Japan.}
\email{seonjeong1124@gmail.com}


\date{\today}
\maketitle

\begin{abstract}
    Given a simple graph $G$, a poset of its even subgraphs was firstly considered by S.~Choi and H.~Park to study the topology of a real toric manifold associated with~$G$. S.~Choi and the authors extended this to a graph allowing multiple edges, motivated by the work on the pseudograph associahedron of Carr, Devadoss and Forcey. In this paper, we completely characterize the graphs (allowing multiple edges) whose posets of  even subgraphs are always shellable. By the result, we also compute the Betti numbers of a real toric manifold corresponding to a path with two multiple edges.
\end{abstract}

\tableofcontents

\section{Introduction}

Throughout this paper, a graph permits multiple edges but not a loop,
and when a simple graph is considered, we always mention that it is `simple'.
 We only consider a finite poset and a finite graph.

Shellability is a combinatorial property of simplicial complexes with strong topological and algebraic consequences;
a shellable simplicial complex has the
homotopy type of a wedge of spheres (in varying dimensions), and a pure shellable simplicial complex is Cohen-Macaulay.
Hence shellability provides a deep and fundamental link between combinatorics and other branches of mathematics, see~\cite{Wachs}. It has been an important research issue to study {shellable} simplicial complexes and investigate their topological properties.
Among them, many simplicial complexes arising from graphs are beautiful objects with a rich topological structure and may hence be considered as interesting in their own right, see~\cite{Jonsson}. This paper also has a contribution to shellability on some simplicial complexes arising from graphs.
Here, we consider shellability for a nonpure simplicial complex developed  by Bj\"{o}rner and Wachs~\cite{BW1996,BW1997}.

Shellablity also plays an important role for computing the integral cohomology of a real toric manifold.
A \emph{real toric manifold}~$X^\R$ is the real locus of a compact smooth toric variety~$X$, where
a \emph{toric variety} of complex dimension~$n$ is a normal algebraic variety over $\C$ with an effective action of $(\C^\ast)^n$ having an open dense orbit.
Whereas the integral cohomology ring of a compact smooth toric variety~$X$ can be explicitly described by the corresponding fan~$\Sigma_X$, see~\cite{Danilov} and~\cite{Jurkiewicz}, only little is known about the cohomology of~$X^\R$.
In 1985, Jurkiewicz~\cite{Jurkiewicz2} showed that the Betti numbers of $X^\R$ over~$\Z_2$\footnote{Given a topological space $X$, the $i$th \emph{Betti number} of $X$, denoted by $\beta^i(X)$, is the free rank of the singular cohomology group $H^i(X;\Z)$ and the $i$th \emph{Betti number} of $X$ over a field $F$, denoted by $\beta^i_F(X)$, is the dimension of $H^i(X;F)$ as a vector space over~$F$. Note that $\beta^i(X)=\beta^i_\Q(X)$.} form the $h$-vector of the underlying simplicial sphere~$K_X$ of the fan~$\Sigma_X$, that is, $\beta^i_{\Z_2}(X^\R)=h_i(K_X)$.
In~\cite{Suciu-Trevisan} and~\cite{Trevisan}, the rational cohomology groups of~$X^\R$ had been formulated  in terms of those of simplicial subcomplexes $K_S$ of $K_X$ for $S\subseteq [n]=\{1,2,\ldots,n\}$; we give the definition of $K_S$ in Section~\ref{sec:application}. Recently, it was also shown in~\cite{CC2017} that the integral cohomology of a real toric manifold $X^\R$ is either torsion-free or has only two-torsion elements if and only if the integral cohomology of $K_S$ is torsion-free for each $S\subseteq [n]$. Hence if $K_S$ is shellable for every $S\subseteq [n]$, then the integral cohomology of~$X^\R$ is completely determined by the Betti numbers of $K_S$ and the $h$-vector of~$K_X$.  But in general $K_S$ does not have the homotopy type of a wedge of spheres and $H^\ast(K_S;\Z)$ may have arbitrary amount of torsion, see~\cite{CP-torsion}. 

On the other hand, each graph~$G$ defines a compact smooth toric variety.
A graph~$H$ is an \emph{induced} (respectively, \emph{semi-induced}) subgraph of $G$ if $H$ is a subgraph that includes all edges (respectively, at least one edge) between every pair of vertices in $H$,  if such edges exist in $G$. A \emph{pseudograph associahedron}~$P_G$ is a simple convex polytope obtained from a product of simplices by truncating the faces corresponding to proper connected semi-induced subgraphs of each component of~$G$. Then the outward primitive normal vectors of the facets of $P_G$ form a complete non-singular fan and hence defines a compact smooth toric variety $X_G$ and a real toric manifold $X^\R_G$. Note that the simplicial sphere $K_{X_G}$ is the dual simplicial complex of the simple polytope~$P_G$.

For a simple graph $G$, it was shown in~\cite{CP} that for each simplicial subcomplex $K_S$ of $K_{X_G}$, there is a subgraph $H$ of $G$ such that $K_S$ is homotopy equivalent to the order complex of the proper part of the poset of even subgraphs of $H$. The work of~\cite{CP} was generalized to a graph (allowing multiple edges) in~\cite{CPP2015}.
A graph~$H$ is a \emph{partial underlying induced graph} (\emph{PI-graph} for short) of a graph $G$ if $H$ can be obtained from an induced subgraph of $G$ by replacing some bundles with simple edges, where a \emph{bundle} is a maximal set of multiple edges which have the same pair of endpoints.
Let $\mathcal{A}^\ast(G)$ be the set of all pairs $(H,A)$ of a PI-graph $H$ of $G$ and an admissible collection~$A$ of $H$, where an admissible collection $A$ of $H$ is defined to be a set of vertices and multiple edges of $H$ with a certain property, see Definition~\ref{def:admissible-modify}.
For each $(H,A)\in \mathcal{A}^\ast(G)$,  $\pP{H,A}$ is a poset whose elements are all semi-induced subgraphs~$I$ of $H$ such that each component of $I$ has an even number of elements in~$A$, including both $\emptyset$ and $H$, ordered by subgraph containment. Then each simplicial subcomplex $K_S$ of $K_{X_G}$ is homotopy equivalent to the order complex of the proper part of $\pP{H,A}$ for some $(H,A)\in\mathcal{A}^\ast(G)$, see~\cite{CPP2015}. All definitions are presented in Section~\ref{sec:admissible} again, with examples and observations.

In~\cite{CP}, they also showed that for a simple graph~$G$ the poset $\pP{H,A}$ is {pure} shellable for every $(H,A)\in\mathcal{A}^\ast(G)$, and then described the Betti numbers of $X^\R_G$ in terms of a graph invariant called the $a$-number.\footnote{We also refer reader to~\cite{PPP18} for the further relationship between graph invariants and the Betti numbers of real toric manifolds arising from a simple graph.} But for a non-simple graph $H$, $\pP{H,A}$ is not shellable in general, see Section~\ref{sec:List_Graphs}, and hence it is natural to ask the following question.

\begin{question}[\cite{CPP2015}]\label{question1}
Find all graphs $G$ such that $\pP{H,A}$ is shellable for every $(H,A) \in \mathcal{A}^{\ast}(G)$.
\end{question}

As we mentioned before, the question {arose} to understand a real toric manifold $X^\R_G$, see Section~\ref{sec:application} for the detail.
Since multiple edges of a graph $G$ make the poset $\pP{H,A}$ nonpure and so the poset has a complicate structure,
the question above is  nontrivial to handle.  In this paper, we answer the question above completely.

\begin{theorem}[Main result]\label{main}
  Let $G$ be a graph. Then $\pP{H,A}$ is shellable for every $(H,A)\in A^\ast(G)$ if and only if each component of~$G$ is either a simple graph or one of the graphs in the following figure.
\begin{figure}[h]
\begin{subfigure}[t]{.3\textwidth}
    \centering
    \begin{tikzpicture}[scale=.6]
    \draw (-1,0) node{\tiny$\tilde{P}_{n,m}$}     (-1,-.5) node{\tiny$(n\ge 2)$};
    \fill (0,0) circle (3pt) (1,0) circle (3pt) (2,0) circle (3pt) (3,0) circle (3pt) (4,0) circle (3pt) (5,0) circle (3pt) (6,0) circle (3pt);
    \draw (1,0)--(3.2,0);
    \draw (3.8,0)--(6,0);
    \draw plot [smooth] coordinates {(0,0) (0.5,0.4) (1,0)};
    \draw plot [smooth] coordinates {(0,0) (0.5,-0.6) (1,0)};
    \draw plot [smooth] coordinates {(0,0) (0.5,0.6) (1,0)};
    \draw (0.5,0.1) node{$\vdots$};
    \draw[dotted] (3.2,0)--(3.8,0);
    \end{tikzpicture}
\end{subfigure}\quad
\begin{subfigure}[t]{.3\textwidth}
    \centering
    \begin{tikzpicture}[scale=.6]
    \draw (-1,0) node{\tiny$\tilde{S}_{n,m}$} (-1.3,-.5) node{\tiny$(n\ge 5, odd)$};
    \fill (0,0) circle (3pt) (1,0) circle (3pt) (2,0) circle (3pt) (3,0) circle (3pt) (4,0) circle (3pt) (5,0) circle (3pt) (6,0) circle (3pt) (5,-1) circle (3pt);
    \draw (1,0)--(3.2,0);
    \draw (3.8,0)--(6,0);
    \draw (5,0)--(5,-1);
    \draw plot [smooth] coordinates {(0,0) (0.5,0.4) (1,0)};
    \draw plot [smooth] coordinates {(0,0) (0.5,-0.6) (1,0)};
    \draw plot [smooth] coordinates {(0,0) (0.5,0.6) (1,0)};
    \draw (0.5,0.1) node{$\vdots$};
    \draw[dotted] (3.2,0)--(3.8,0);
    \end{tikzpicture}
\end{subfigure}\quad
\begin{subfigure}[t]{.3\textwidth}
    \centering
    \begin{tikzpicture}[scale=.6]
    \draw (-1,0) node{\tiny$\tilde{T}_{n,m}$}   (-1.3,-.5) node{\tiny$(n\ge 5, odd)$};
    \fill (0,0) circle (3pt) (1,0) circle (3pt) (2,0) circle (3pt) (3,0) circle (3pt) (4,0) circle (3pt) (5,0) circle (3pt) (6,0) circle (3pt) (5,-1) circle (3pt);
    \draw (1,0)--(3.2,0);
    \draw (3.8,0)--(6,0);
    \draw (5,0)--(5,-1);
    \draw (5,-1)--(6,0);
    \draw plot [smooth] coordinates {(0,0) (0.5,0.4) (1,0)};
    \draw plot [smooth] coordinates {(0,0) (0.5,-0.6) (1,0)};
    \draw plot [smooth] coordinates {(0,0) (0.5,0.6) (1,0)};
    \draw (0.5,0.1) node{$\vdots$};
    \draw[dotted] (3.2,0)--(3.8,0);
    \end{tikzpicture}
\end{subfigure}

\begin{subfigure}[t]{.3\textwidth}
    \centering
    \begin{tikzpicture}[scale=.6]
    \draw (-1,0) node{\tiny$\tilde{P}'_{n,m}$} (-1,-.5) node{\tiny$(n\ge 3)$};
    \fill (0,0) circle (3pt) (1,0) circle (3pt) (2,0) circle (3pt) (3,0) circle (3pt) (4,0) circle (3pt) (5,0) circle (3pt) (6,0) circle (3pt);
    \draw (1,0)--(3.2,0);
    \draw (3.8,0)--(6,0);
    \draw plot [smooth] coordinates {(0,0) (0.5,0.4) (1,0)};
    \draw plot [smooth] coordinates {(0,0) (0.5,-0.6) (1,0)};
    \draw plot [smooth] coordinates {(0,0) (0.5,0.6) (1,0)};
    \draw plot [smooth] coordinates {(0,0) (0.4,-0.8) (1,-1) (1.6,-0.8) (2,0)};
    \draw (0.5,0.1) node{$\vdots$};
    \draw[dotted] (3.2,0)--(3.8,0);
    \end{tikzpicture}
\end{subfigure}\quad
\begin{subfigure}[t]{.3\textwidth}
    \centering
    \begin{tikzpicture}[scale=.6]
          \draw (-1,0) node{\tiny$\tilde{S}'_{n,m}$} (-1.3,-.5) node{\tiny$(n\ge 5, odd)$};
    \fill (0,0) circle (3pt) (1,0) circle (3pt) (2,0) circle (3pt) (3,0) circle (3pt) (4,0) circle (3pt) (5,0) circle (3pt) (6,0) circle (3pt) (5,-1) circle (3pt);
    \draw (1,0)--(3.2,0);
    \draw (3.8,0)--(6,0);
    \draw (5,0)--(5,-1);
    \draw plot [smooth] coordinates {(0,0) (0.5,0.4) (1,0)};
    \draw plot [smooth] coordinates {(0,0) (0.5,-0.6) (1,0)};
    \draw plot [smooth] coordinates {(0,0) (0.5,0.6) (1,0)};
    \draw plot [smooth] coordinates {(0,0) (0.4,-0.8) (1,-1) (1.6,-0.8) (2,0)};
    \draw (0.5,0.1) node{$\vdots$};
    \draw[dotted] (3.2,0)--(3.8,0);
    \end{tikzpicture}
\end{subfigure}\quad
\begin{subfigure}[t]{.3\textwidth}
    \centering
    \begin{tikzpicture}[scale=.6]
       \draw (-1,0) node{\tiny$\tilde{T}'_{n,m}$}   (-1.3,-.5) node{\tiny$(n\ge 5, odd)$};
    \fill (0,0) circle (3pt) (1,0) circle (3pt) (2,0) circle (3pt) (3,0) circle (3pt) (4,0) circle (3pt) (5,0) circle (3pt) (6,0) circle (3pt) (5,-1) circle (3pt);
    \draw (1,0)--(3.2,0);
    \draw (3.8,0)--(6,0);
    \draw (5,0)--(5,-1);
    \draw (5,-1)--(6,0);
    \draw plot [smooth] coordinates {(0,0) (0.5,0.4) (1,0)};
    \draw plot [smooth] coordinates {(0,0) (0.5,-0.6) (1,0)};
    \draw plot [smooth] coordinates {(0,0) (0.5,0.6) (1,0)};
    \draw plot [smooth] coordinates {(0,0) (0.4,-0.8) (1,-1) (1.6,-0.8) (2,0)};
    \draw (0.5,0.1) node{$\vdots$};
    \draw[dotted] (3.2,0)--(3.8,0);
    \end{tikzpicture}
\end{subfigure}
\caption*{Non-simple connected graphs with $n$ vertices and $m$ multiple edges ($m\ge 2$)}
\end{figure}
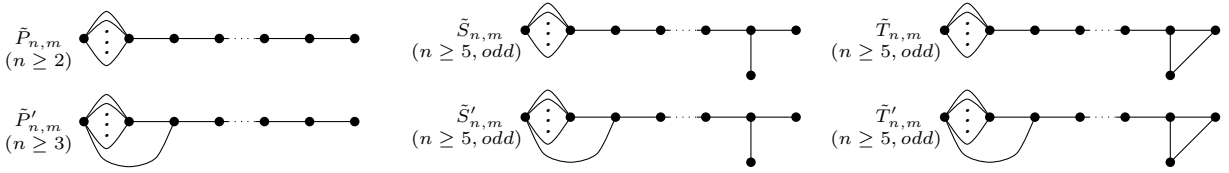
\end{theorem}
\setcounter{figure}{0}

Proving the `if' part of our main result, we find  a recursive atom ordering of a poset $\pP{H,A}$, whereas in general the problem to find a recursive atom ordering of a poset is considered as a difficult problem in the literature.
Since existence of a recursive atom ordering is equivalent to  CL-shellability,
therefore, for a graph $G$ described in the main result, the poset $\pP{H,A}$ is CL-shellable for every $(H,A)\in \mathcal{A}^\ast(G)$.

Theorem~\ref{main} also gives us a family of real toric manifolds whose integral cohomology is either torsion-free or has only two-torsion elements. We compute the Betti numbers of a real toric manifold associated with $\tilde{P}_{n,2}$, a path with two multiple edges, in~\ref{sec:application}.

This paper is organized as follows.
Section~\ref{sec:poset} collects some basic definitions and important facts about a poset and its shellability.
Section~\ref{sec:admissible} gives the definition of the poset $\pP{G,A}$ of the $A$-even subgraphs of a graph $G$, and then explains the main theorem of this paper.
Section~\ref{sec:List_Graphs} proves a necessary condition of the main theorem, which gives a possible list of graphs $G$ such that $\pP{H,A}$ is shellable for every $(H,A)\in \mathcal{A}^{\ast}(G)$.
Section~\ref{sec:CL-shellable} proves a sufficient condition of the main theorem, which shows {the} CL-shellability of each $\pP{H,A}$ for a graph $G$ in the list and $(H,A)\in\mathcal{A}^{\ast}(G)$.
In Section~\ref{sec:falling}, we determine  the homotopy type of $\Delta(\overline{\pP{G,A}})$ by giving an observation on the falling chains of $\pP{G,A}$ for a graph $G$ in Figure~\ref{fig:list of possible graphs} and then we explicitly count the number of falling chains of $\pP{G,A}$ when $G=\tilde{P}_{n,2}$.
In Section~\ref{sec:application}, as an important application of our result, we explain how to compute the  Betti numbers of a  real  toric manifold associated with a graph using our result, and then study it for the graph $\tilde{P}_{n,2}$. Section~\ref{sec:last} gives some further questions.

\section{Preliminaries: A poset and its shellability}\label{sec:poset}

In this section, we prepare some notions and basic facts about a poset and its shellability. See~\cite{Wachs} for more detailed explanation about this section.

We only consider a finite poset in this paper.
Let $\mathcal{P}$ be a poset (partially ordered set).
For two elements $x,y\in\mathcal{P}$, we say $y$ \emph{covers} $x$, denoted by $x\lessdot y$, if $x<y$ and there is no $z$ such that $x<z<y$. We also call it a cover $x\lessdot y$.
One represents $\mathcal{P}$ as a  mathematical diagram, called a \emph{Hasse diagram}, in a way that a point in the plane is drawn for each element of $\mathcal{P}$, and a line segment or curve is drawn  upward from $x$ to $y$ whenever $y$ covers $x$.
A \emph{chain} of $\mathcal{P}$ is
a totally ordered subset
 $\sigma$ of $\mathcal{P}$, and we say the \emph{length} $\ell(\sigma)$ of $\sigma$ is $|\sigma|-1$.
We say $\mathcal{P}$ is \emph{pure} if all maximal chains have the same length.
The \emph{length} $\ell(\mathcal{P})$ of $\mathcal{P}$ is the length of a longest chain of $\mathcal{P}$.
For $x \leq y$ in~$\mathcal{P}$, let $[x,y]$ denote the (closed) interval $\{z\in \mathcal{P}\colon x \le z \le y\}$.
 We say $\mathcal{P}$ is \emph{semimodular} if for all $x,y\in \mathcal{P}$ that cover $a\in \mathcal{P}$, there is an element $b\in \mathcal{P}$ that covers both $x$ and $y$.
If every closed interval of $\mathcal{P}$ is semimodular, then $\mathcal{P}$ is said to be \emph{totally semimodular}.
If $\mathcal{P}$ has a unique minimum element, it is usually denoted by $\hat{0}$ and referred
to as the bottom element. Similarly, the unique maximum element, if it exists,
is denoted by $\hat{1}$ and referred to as the top element.
An element of $\mathcal{P}$ that covers the bottom element, if it exists, is called an \emph{atom}.
We say $\mathcal{P}$ is \emph{bounded} if it has the elements $\hat{0}$ and $\hat{1}$.
The \emph{order complex} of $\mathcal{P}$, denoted by $\Delta(\mathcal{P})$, is an abstract simplicial complex whose  faces are the chains of $\mathcal{P}$.
Note that if $\mathcal{P}$ has either $\hat{0}$ or $\hat{1}$, then $\Delta(\mathcal{P})$ is contractible, hence we usually remove the top and bottom elements, and then study the topology of the remaining part. The \emph{proper part} of a bounded poset $\mathcal{P}$ with  length at least one is defined to be $\overline{\mathcal{P}}:=\mathcal{P}-\{\hat{0},\,\hat{1}\}$.

The notion of shellability {was} firstly appeared in the middle of the nineteenth century in the computation of the Euler characteristic of a convex polytope~\cite{Sch01}, and in this paper {shellability refers to the general notion of nonpure shellability introduced in~\cite{BW1996}.}
A simplicial complex $K$ is \emph{shellable} if its facets can be arranged in linear order $F_1, F_2, \ldots ,F_t$ in such a way that the subcomplex $(\sum_{i=1}^{k-1} \overline{F_i})\cap\overline{F_k}$ is pure and $(\dim {F_k}-1)$-dimensional for all $k = 2, \ldots, t$. Such an ordering of the facets is called a \emph{shelling}.
A poset $\mathcal{P}$ is said to be \emph{shellable} if its order complex $\Delta(\mathcal{P})$ is shellable.

A chain-lexicographic shellability (CL-shellability for short) was
introduced by
Bj\"{o}rner and Wachs to establish {the} shellability of Bruhat order on a Coxeter group \cite{BW1982}.
It is known that CL-shellability is stronger than shellability, that is, if a bounded poset is CL-shellable, then it is shellable, but the converse is not true, see~\cite{VW1985}.
Let  $\mathcal{P}$ be a bounded poset.
We denote by $\mathcal{ME}(\mathcal{P})$ the set of pairs $(\sigma,x \lessdot y)$ consisting of a maximal chain~$\sigma$ and a cover $x\lessdot y$ along that chain.
For $x,y\in\mathcal{P}$ and a maximal chain $r$  of  $[\hat{0},x]$, the closed rooted interval $[x,y]_r$ of $\mathcal{P}$ is a subposet of $\mathcal{P}$ obtained from $[x,y]$ adding the chain $r$.
A \emph{chain-edge labeling} of~$\mathcal{P}$ is a map $\rho\colon \mathcal{ME}(\mathcal{P}) \rightarrow L$, where $L$ is some poset satisfying; if two maximal chains coincide along their bottom $d$ covers, then their labels also coincide along these covers.  A \emph{chain-lexicographic labeling} (\emph{CL-labeling} for short) of a bounded poset $\mathcal{P}$ is a \emph{chain-edge labeling} such that for each closed rooted interval $[x,y]_r$ of~$\mathcal{P}$, there is a unique strictly increasing maximal chain, which lexicographically precedes all other maximal chains of $[x,y]_r$. A poset that admits a CL-labeling is said to be \emph{CL-shellable}. Figure~\ref{fig:CL-example} shows an example of a CL-shellable poset.
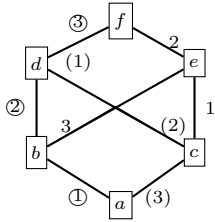
\begin{figure}[h]
\centering
    	\begin{tikzpicture}[scale=.7]
            \node [draw] (a) at (4.6,-1) {\!\tiny$a$\!};
        	\node [draw] (b) at (3,0) {\!\tiny$b$\!};
        	\node [draw] (c) at (6,0) {\!\tiny$c$\!};
        	\node [draw] (d) at (3,1.7) {\!\tiny$d$\!};
        	\node [draw] (e) at (6,1.7) {\!\tiny$e$\!};
            \node [draw] (f) at (4.6, 2.5) {\!\tiny$f$\!};
            \path (b) edge node[left]{\tiny{3}} (4.5,0.875);
            \path (c) edge node[right]{\tiny{(2)}} (4.5,0.875);
            \path[thick] (a) edge node[below]{\tiny\textcircled{1}} (b) edge node[below]{\tiny{(3)}} (c);
            \path[thick] (b) edge node[left]{\tiny{\textcircled{2}}} (d) edge (e);
            \path[thick] (c) edge (d) edge node[right]{\tiny{1}} (e);
            \path[thick] (f) edge node[above]{\tiny\textcircled{3}} node[below]{\tiny{(1)}} (d) edge node[right]{\tiny{2}} (e);
            \path[thick] (7.5,2) node[right]{\!\footnotesize Labeling of the covers in chain $a< b < d <f$ is  1, 2, 3 (marked as \textcircled{1},  \textcircled{2},  \textcircled{3}).    \!};
       \path[thick] (7.5,0.4) node[right]{\!\footnotesize Labeling of the covers in chain $a< c < d <f$ is  3, 2, 1 (marked as (3),  (2), (1)).\!};
       \path[thick] (7.5,1.2) node[right]{\!\footnotesize Labeling of the covers in chain $a< b< e <f$ is  1, 3, 2 (marked as \textcircled{1}, 3, 2).\!};
       \path[thick] (7.5,-0.4) node[right]{\!\footnotesize Labeling of the covers in chain $a< c< e <f$ is  3, 1, 2 (marked as (3), 1, 2).\!};
    	\end{tikzpicture}
        \captionsetup{width=1.0\linewidth}
        \caption{A chain-edge labeling of a poset with four maximal chains  (same example in \cite{Wachs})
        }\label{fig:CL-example}
\end{figure}

We recall well-known properties on shellability and CL-shellability which {we will use.}
The \emph{product} $\mathcal{P}\times \mathcal{Q}$ of two posets $\mathcal{P}$ and $\mathcal{Q}$ is the new poset with partial order given by $(a,b) \le (c,d)$ if and only if $a \le c$ (in $\mathcal{P}$) and $b \le d$ (in $\mathcal{Q}$).

\begin{theorem}[\cite{B1980, BW1996, BW1997}]\label{thm:product_shellable} The following hold:
\begin{enumerate}
  \item \label{(2)} Every (closed) interval of a shellable ({respectively,} CL-shellable) poset is shellable ({respectively,} CL-shellable).
 \item \label{(3)} The product of bounded posets is shellable ({respectively,} CL-shellable) if and only if each of the posets is shellable ({respectively,} CL-shellable).
  \item \label{(4)} A bounded poset is pure and totally semimodular, then it is CL-shellable.
\end{enumerate}
\end{theorem}

It is worthy to note that the homotopy type of $\Delta(\overline{\mathcal{P}})$ is known when
 a bounded poset $\mathcal{P}$ has a CL-labeling $\rho\colon\mathcal{ME}(\mathcal{P}) \rightarrow L$.  A \emph{falling chain} $\sigma:x_0\lessdot x_1\lessdot\cdots\lessdot x_{\ell}$  of $\mathcal{P}$  is a maximal chain such that  $\rho(\sigma,x_{i-1}\lessdot x_i)  \geq_{L} \rho(\sigma,x_{i}\lessdot x_{i+1})$   for every $1\leq i<\ell$.

\begin{theorem}[\cite{BW1996}] \label{thm:falling}
    If a bounded poset $\mathcal{P}$ is CL-shellable, then $\Delta(\overline{\mathcal{P}})$ has the homotopy type of a wedge of spheres. Furthermore, for any fixed CL-labeling, the  $i$th reduced Betti\footnote{For a topological space $X$, the $i$th \emph{reduced Betti number} of $X$, denoted by $\tilde{\beta}^i(X)$, is the free rank of the reduced singular cohomology group $\widetilde{H}^i(X;\Z)$ and the $i$th \emph{reduced Betti number of $X$ over a field $F$}, denoted by $\tilde{\beta}^i_F(X)$, is the dimension of $\widetilde{H}^i(X;F)$ as a vector space over~$F$.} number of $\Delta(\overline{\mathcal{P}})$ is equal to the number of falling chains of length $i+2$.
\end{theorem}
The poset in Figure~\ref{fig:CL-example} has {exactly} one falling chain $a<c<d<f$, and  $\Delta(\overline{P})$ is homotopy equivalent to $S^1$.

A recursive atom ordering is an alternative approach to lexicographic shellability, which is known to be an equivalent concept of CL-shellability.

\begin{definition}\label{def:recursive atom order}
    A bounded poset $\mathcal{P}$ is said to \emph{admit a recursive atom ordering} if its length $\ell(\mathcal{P})$ is $1$, or $\ell(\mathcal{P})>1$ and there is an ordering $\alpha_1,\ldots,\alpha_t$ of the atoms of $\mathcal{P}$ satisfying the following:
    \begin{enumerate}
        \item  For all $j=1,\ldots,t$, the  interval $[\alpha_j,\hat{1}]$ admits a recursive atom ordering in which the atoms of $[\alpha_j,\hat{1}]$ that belong to $[\alpha_i,\hat{1}]$ for some $i<j$ come first.
        \item For all $i,j$ with $1\le i<j \le t$, if $\alpha_i,\alpha_j<y$ then there exist  an integer $k$ and an atom $z$ of $[\alpha_j,\hat{1}]$ such that $1\le k<j$ and $\alpha_k<z\leq y$.
    \end{enumerate}
\end{definition}
For example, for the poset in Figure~\ref{fig:CL-example}, if we order the atoms of each interval by an alphabetical order (for the atoms of $[a,f]$, the ordering is $b\prec c$, for the atoms of $[b,f]$, the ordering is $d\prec e$, and for the atoms of $[c,f]$, the ordering is $d\prec e$), then it is a recursive atom ordering.

We note that any atom ordering of a pure totally semimodular bounded poset is a recursive atom ordering, which implies   (\ref{(4)}) of Theorem~\ref{thm:product_shellable}. We finish the section by introducing a sketch of the proof  shown in \cite{BW1996} that the existence of a recursive atom ordering implies  CL-shellability.

\begin{theorem}[\cite{BW1996}]\label{thm:CL-shellable-recursive-atom}
    A bounded poset  admits a recursive atom ordering if and only if it is CL-shellable.
\end{theorem}
\begin{proof}[Sketch of proof of the `only if' part] Suppose that a bounded poset $\mathcal{P}$  admits a recursive atom ordering, and  let  the atoms of $\mathcal{P}$ be ordered as $\alpha_1$, $\ldots$, $\alpha_t$.
Let us give an integer labeling $\rho$ of the bottom covers of $\mathcal{P}$ such that $\rho(\hat{0},\alpha_i)<\rho(\hat{0},\alpha_j)$ for all $i<j$. For each $j$, let $F(\alpha_j)$ be the set of all atoms of $[\alpha_j,\hat{1}]$ that cover some $\alpha_i$ where $i<j$. We label the bottom covers of $[\alpha_j,\hat{1}]$ consistently with the atom ordering of $[\alpha_j,\hat{1}]$ and satisfying
\[x\in F(\alpha_j) \Rightarrow \rho(\alpha_j,x)<\rho(\hat{0},\alpha_j)\qquad \text{and}\qquad
        x\not\in F(\alpha_j) \Rightarrow \rho(\alpha_j,x)>\rho(\hat{0},\alpha_j),
\] where $\rho$ denotes the labeling of the bottom covers of $[\alpha_j,\hat{1}]$ as well as the original labeling of the bottom covers of $\mathcal{P}$. This labeling inductively extends to an integer CL-labeling of $[\alpha_j,\hat{1}]$. Choosing such an extension at each $\alpha_j$, we obtain a chain-edge labeling $\rho$ of $\mathcal{P}$ which is a CL-labeling of $[\alpha_j, \hat{1}]$ for all $j = 1,\ldots ,t$, and hence for every rooted
interval whose bottom element is not $\hat{0}$, and which extends the original labeling of
the bottom covers of $\mathcal{P}$. Then one can show that the unique lexicographically first maximal chain of
each interval $[0, y]$ is the only increasing maximal chain of that interval. Hence the labeling $\rho$ is an integer CL-labeling on $\mathcal{P}$.
\end{proof}

\section{A poset $\mathcal{P}_{G,A}^{even}$ of $A$-even subgraphs of a graph $G$ and the main result}\label{sec:admissible}

  In this section, we give basic definitions related to graphs and then define the poset $\pP{G,A}$.
  Here, the poset $\pP{G,A}$ arose from the computation of the Betti numbers of a real toric manifold associated with a graph $G$ in \cite{CPP2015}\footnote{In \cite{CDF2011,CPP2015}, a graph allowing multiple edges is called a \emph{pseudograph}.}, which gives a strong  motivation for the main question of this paper. We present the connection of our main result to a real toric manifold in Section~\ref{sec:application}.
  See \cite[Sections 2 and 3]{CPP2015}, for readers to find a much more detailed account of results of pseudograph associahedra.

    For a graph $G=(V,E)$, an element of $V$ and an element of $E$ are called a \emph{vertex} and an \emph{edge} respectively, and we only consider a finite graph not allowing a loop, that is, an edge whose endpoints are the same. An edge $e$ is said to be \emph{multiple} if there exists another edge $e'$ which has the same endpoints as $e$. An edge which is not a multiple edge is said to be  \emph{simple}. A \emph{bundle} is a maximal set of multiple edges which have the same endpoints. A \emph{simple graph} is a graph all of whose edges are simple.

    Let $G$ be a graph.
    A subgraph $H$ of $G$ is an \emph{induced} (respectively, \emph{semi-induced}) subgraph of $G$ if $H$ includes all edges (respectively, at least one edge) between every pair of vertices in $H$ if such edges exist in $G$.
   A graph $H$ is a \emph{partial underlying graph} of $G$ if $H$ can be obtained from $G$ by replacing some bundles with simple edges, that is, the set of all the bundles of $H$ is a subset of that of $G$.
A graph $H$ is a \emph{partial underlying induced} graph (\emph{PI}-graph for short) of $G$ if $H$ is an induced subgraph of some partial underlying graph of $G$. Note that any graph is a PI-graph of itself.
For instance, for the graph~$G$ with two bundles $\{a,b\}$ and $\{c,d,e\}$ in Figure~\ref{fig:example_C_even_poset},
\begin{itemize}
    \item $I_1$, $I_2$, and $I_3$  are semi-induced subgraphs of $G$,
    \item $H_1$, $H_2$, and $H_3$  are partial underlying graphs of $G$, and
    \item all $H_i$'s and the subgraphs $I_1$ and $I_2$ are PI-graphs of $G$.
\end{itemize}

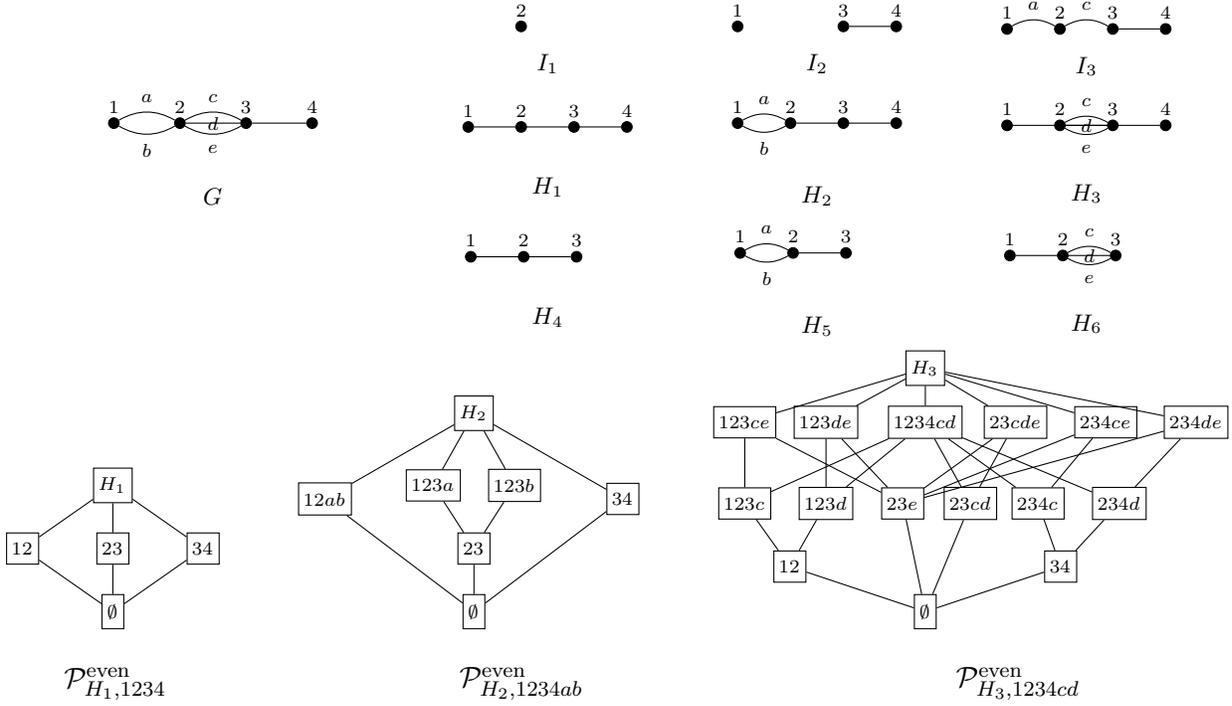
\begin{figure}[t]
    \begin{center}
\hspace{.3\textwidth}
    \begin{subfigure}[c]{.2\textwidth}
    \centering
    \begin{tikzpicture}[scale=0.4]
        \node [v] (v1) [white]{} (v1) node[above]{ };
        \node [v] (v2) [right of=v1, node distance=20pt]{} (v2) node[above]{\tiny2};
        \node [v] (v4) [right of=v2, node distance=20pt, white]{} (v4) node[above]{ };
        \node [v] (v5) [right of=v4, node distance=20pt, white]{} (v5) node[above]{ };
    \end{tikzpicture}
    \caption*{\footnotesize$I_1$}
    \end{subfigure}
    \begin{subfigure}[c]{.2\textwidth}
    \centering
    \begin{tikzpicture}[scale=0.4]
        \node [v] (v1) {} (v1) node[above]{\tiny1};
        \node [v] (v2) [right of=v1, node distance=20pt, white]{} (v2) node[above]{ };
        \node [v] (v4) [right of=v2, node distance=20pt]{} (v4) node[above]{\tiny3};
        \node [v] (v5) [right of=v4, node distance=20pt]{} (v5) node[above]{\tiny4};
        \path (v4) edge (v5);
    \end{tikzpicture}
    \caption*{\footnotesize$I_2$}
    \end{subfigure}
    \begin{subfigure}[c]{.2\textwidth}
    \centering
    \begin{tikzpicture}[scale=0.4]
        \node [v] (v1) {} (v1) node[above]{\tiny1};
        \node [v] (v2) [right of=v1, node distance=20pt]{} (v2) node[above]{\tiny2};
        \node [v] (v4) [right of=v2, node distance=20pt]{} (v4) node[above]{\tiny3};
        \node [v] (v5) [right of=v4, node distance=20pt]{} (v5) node[above]{\tiny4};
        \path
        (v4) edge (v5)
        (v2) edge[bend left] node[above] {\tiny $c$} (v4);
        \path
        (v1) edge[bend left] node[above] {\tiny$a$} (v2);
    \end{tikzpicture}
    \caption*{\footnotesize$I_3$}
    \end{subfigure}

   \begin{subfigure}[c]{0.3\textwidth}
    \centering
    \begin{tikzpicture}[scale=0.4]
        \node [v] (v1) {} (v1) node[above]{\tiny1};
        \node [v] (v2) [right of=v1, node distance=25pt]{} (v2) node[above]{\tiny2};
        \node [v] (v4) [right of=v2, node distance=25pt]{} (v4) node[above]{\tiny3};
        \node [v] (v5) [right of=v4, node distance=25pt]{} (v5) node[above]{\tiny4};
        \path
        (v1) edge[bend left] node[above] {\tiny$a$} (v2)
             edge[bend right] node[below] {\tiny$b$} (v2)
        (v4)
             edge (v5)
        (v2) edge[bend left] node[above] {\tiny$c$} (v4)
             edge node {\tiny$d$} (v4)
             edge[bend right] node[below] {\tiny$e$} (v4);
    \end{tikzpicture}
    \caption*{\footnotesize$G$}
    \end{subfigure}
    \begin{subfigure}[c]{.2\textwidth}
    \centering
    \begin{tikzpicture}[scale=0.35]
        \node [v] (v1) {} (v1) node[above]{\tiny1};
        \node [v] (v2) [right of=v1, node distance=20pt]{} (v2) node[above]{\tiny2};
        \node [v] (v4) [right of=v2, node distance=20pt]{} (v4) node[above]{\tiny3};
        \node [v] (v5) [right of=v4, node distance=20pt]{} (v5) node[above]{\tiny4};
        \draw (v1)--(v5);
        \path[white]  (v1) edge[bend left] node[above] { } (v2)
             edge[bend right] node[below] { } (v2);
    \end{tikzpicture}
    \caption*{\footnotesize$H_1$}
    \end{subfigure}
    \begin{subfigure}[c]{0.2\textwidth}
    \centering
    \begin{tikzpicture}[scale=0.15]
        \node [v] (v1) {} (v1) node[above]{\tiny1};
        \node [v] (v2) [right of=v1, node distance=20pt]{} (v2) node[above]{\tiny2};
        \node [v] (v4) [right of=v2, node distance=20pt]{} (v4) node[above]{\tiny3};
        \node [v] (v5) [right of=v4, node distance=20pt]{} (v5) node[above]{\tiny4};
        \path
        (v1) edge[bend left] node[above] {\tiny$a$} (v2)
             edge[bend right] node[below] {\tiny$b$} (v2);
        \draw (v2)--(v5);
    \end{tikzpicture}
    \caption*{\footnotesize$H_2$}
    \end{subfigure}
      \begin{subfigure}[c]{0.2\textwidth}
    \centering
    \begin{tikzpicture}[scale=0.2]
        \node [v] (v1) {} (v1) node[above]{\tiny1};
        \node [v] (v2) [right of=v1, node distance=20pt]{} (v2) node[above]{\tiny2};
        \node [v] (v4) [right of=v2, node distance=20pt]{} (v4) node[above]{\tiny3};
        \node [v] (v5) [right of=v4, node distance=20pt]{} (v5) node[above]{\tiny4};
        \draw (v1)--(v5);
        \path  (v2)  edge node {\tiny $d$} (v4)
            (v2) edge[bend left] node[above] {\tiny$c$} (v4)
             edge[bend right] node[below] {\tiny$e$} (v4);
    \end{tikzpicture}
    \caption*{\footnotesize$H_3$}
    \end{subfigure}

       \hspace{0.3\textwidth}
\begin{subfigure}[c]{0.2\textwidth}
    \centering
        \begin{tikzpicture}[scale=0.15]
        \node [v] (v1) {} (v1) node[above]{\tiny1};
        \node [v] (v2) [right of=v1, node distance=20pt]{} (v2) node[above]{\tiny2};
        \node [v] (v4) [right of=v2, node distance=20pt]{} (v4) node[above]{\tiny3};
    \node [v] (v5) [right of=v4, node distance=20pt, white]{} (v5) node[above]{ };
 \path[white]
  (v1) edge[bend right] node[below] { } (v2);
        \draw (v1)--(v4);
    \end{tikzpicture}
    \caption*{\footnotesize$H_4$}
    \end{subfigure}
    \begin{subfigure}[c]{0.2\textwidth}
    \centering
    \begin{tikzpicture}[scale=0.15]
        \node [v] (v1) {} (v1) node[above]{\tiny1};
        \node [v] (v2) [right of=v1, node distance=20pt]{} (v2) node[above]{\tiny2};
        \node [v] (v4) [right of=v2, node distance=20pt]{} (v4) node[above]{\tiny3};
    \node [v] (v5) [right of=v4, node distance=20pt, white]{} (v5) node[above]{ };
          \path
        (v1) edge[bend left] node[above] {\tiny$a$} (v2)
             edge[bend right] node[below] {\tiny$b$} (v2);
        \path (v2) edge (v4);
    \end{tikzpicture}
    \caption*{\footnotesize$H_5$}
    \end{subfigure}
    \begin{subfigure}[c]{0.2\textwidth}
    \centering
    \begin{tikzpicture}[scale=0.2]
        \node [v] (v1) {} (v1) node[above]{\tiny1};
        \node [v] (v2) [right of=v1, node distance=20pt]{} (v2) node[above]{\tiny2};
        \node [v] (v4) [right of=v2, node distance=20pt]{} (v4) node[above]{\tiny3};
        \node [v] (v5) [right of=v4, node distance=20pt, white]{} (v5) node[above]{ };
 \draw (v1)--(v2);
          \path   (v2) edge[bend left] node[above] {\tiny$c$} (v4)
             edge node {\tiny$d$} (v4)
             edge[bend right] node[below] {\tiny$e$} (v4);
    \end{tikzpicture}
    \caption*{\footnotesize$H_6$}
    \end{subfigure}
     \addtocounter{subfigure}{-8}
    \captionsetup{width=1.0\linewidth}
    \end{center}

 \begin{center}
    \begin{tikzpicture}[scale=.6]
    	\node [draw] (G) at (0,1.4) {\!\tiny$H_1$\!};
    	\node [draw] (0) at (0,-1.4) {\!\tiny$\emptyset$\!};
    	\node [draw] (23) at (0,0) {\!\tiny$23$\!};
    	\node [draw] (34) at (2,0) {\!\tiny$34$\!};
    	\node [draw] (12ab) at (-2,0) {\!\tiny$12$\!};
    \path (0) edge (34) edge (23)edge (12ab);
    \path (G) edge (12ab) edge (34) edge (23);
  \draw (0,-3) node{$\mathcal{P}_{H_1,1234}^{\mathrm{even}}$\!};
  \draw (9,-3) node{$\mathcal{P}_{H_2,1234ab}^{\mathrm{even}}$\!};
  \draw (20,-3) node{$\mathcal{P}_{H_3,1234cd}^{\mathrm{even}}$\!};
    	\node [draw] (G) at (8,3) {\!\tiny$H_2$\!};
    	\node [draw] (0) at (8,-1.4) {\!\tiny$\emptyset$\!};
    	\node [draw] (23) at (8,0) {\!\tiny$23$\!};
    	\node [draw] (34) at (11.3,1.1) {\!\tiny$34$\!};
    	\node [draw] (12ab) at (4.7,1.1) {\!\tiny$12ab$\!};
    	\node [draw] (123a) at (7.1,1.4) {\!\tiny$123a$\!};
    	\node [draw] (123b) at (8.9,1.4) {\!\tiny$123b$\!};
    \path (0) edge (34) edge (23)edge (12ab);
    \path (G) edge (123a) edge (123b)edge (12ab) edge (34);
    	\path (23) edge (123a);
    	\path (23) edge (123b);
        \node [draw] (G) at (18,4) {\!\tiny$H_3$\!};
    	\node [draw] (0) at (18,-1.4) {\!\tiny$\emptyset$\!};
    	\node [draw] (12) at (15,-0.4) {\!\tiny$12$\!};
    	\node [draw] (34) at (21,-0.4) {\!\tiny$34$\!};
       	\node [draw] (123c) at (14,1) {\!\tiny$123c$\!};
    	\node [draw] (123d) at (15.8,1) {\!\tiny$123d$\!};
    	\node [draw] (23e) at (17.5,1) {\!\tiny$23e$\!};
    	\node [draw] (23cd) at (19,1) {\!\tiny$23cd$\!};
        \node [draw] (234c) at (20.5,1) {\!\tiny$234c$\!};
    	\node [draw] (234d) at (22.3,1) {\!\tiny$234d$\!};
    \node [draw] (23cde) at (20,2.8) {\!\tiny$23cde$\!};
        \node [draw] (234ce) at (22,2.8) {\!\tiny$234ce$\!};
    	\node [draw] (234de) at (24,2.8) {\!\tiny$234de$\!};
    	\node [draw] (123ce) at (14,2.8) {\!\tiny$123ce$\!};
    	\node [draw] (123de) at (15.8,2.8) {\!\tiny$123de$\!};
    	\node [draw] (1234cd) at (18,2.8) {\!\tiny$1234cd$\!};
        \path (0) edge (34) edge (12) edge (23cd) edge (23e);
    \path (G) edge (123ce) edge (123de)edge (1234cd) edge (234ce) edge (234de) edge (23cde);
    	\path (12) edge (123c) edge(123d);
      	\path (34) edge (234c) edge(234d);
      \path (123c) edge (123ce) edge(1234cd);
          	\path (123d) edge (123de) edge(1234cd);
          \path (234c) edge (234ce) edge(1234cd);
          \path (234d) edge (234de) edge(1234cd);
          \path (23e) edge (234ce) edge(234de) edge (123ce) edge (123de) edge (23cde);
          \path (23cd) edge (1234cd) edge (23cde);
    \end{tikzpicture}
        \end{center}
          \captionsetup{width=1.0\linewidth}
    \caption{Examples for PI-graphs of $G$ and the posets $\pP{H,A}$  } \label{fig:example_C_even_poset}
    \end{figure}

    Before stating the definitions of the main notion of the paper, we  need to explain a way to denote a subgraph of a graph by a set. For a graph $G$, we  label  the vertices and the multiple edges {of} a graph $G$ and we set $\cC_G= V(G)\cup B_1\cup \cdots \cup B_k$, where $B_1,\ldots,B_k$ are the bundles of $G$.
    For instance, $\cC_G=\{1,2,3,4,a,b,c,d,e\} $ and $\cC_{H_3}=\{1,2,3,4,c,d,e\}$ for the graphs $G$ and $H_3$ in Figure~\ref{fig:example_C_even_poset}.
    A subgraph~$I$ of $G$ will be  written as the set of the vertices of $I$ and the edges of $I$ in a bundle of $G$. For instance, the three subgraphs $I_1$, $I_2$, and $I_3$ of $G$ in Figure~\ref{fig:example_C_even_poset} are expressed as
   $I_1=\{2\}$, $I_2=\{1,3,4\}$, and  $I_3=\{1,2,3,4,a,c\}$.
    It should be noted that for a semi-induced subgraph $I$, this set expression makes sense because $I$ is distinguishable by the corresponding set. In the same sense, for a semi-induced subgraph $I$, we  say  $\alpha \in I$ if $\alpha$ is a vertex of $I$ or  an edge of $I$ which is  a multiple edge of $G$. For simplicity, we omit the braces and commas to denote a subset of $\cC_G$ and we always denote it in a way that the vertices precede  the multiple edges. For the  semi-induced subgraphs  $I_i$'s in Figure~\ref{fig:example_C_even_poset},
    \[I_1=2, \quad I_2=134, \quad I_3=1234ac.\]
 We  remark that when we consider a subgraph~$I$ of a graph $G$, the labels of $I$ are inherited from the labels of $G$. Thus if a graph~$I$  is considered as a subgraph of a graph $G$, then~$I$ may have a labeled simple edge, which is not in a bundle of~$I$ (actually, it is in a bundle of $G$). Note that for the graph $G$ in Figure~\ref{fig:example_C_even_poset}, $12a$ and $12b$ are different objects if they are considered as semi-induced subgraphs of $G$.

\begin{definition}\label{def:admissible-modify}
For a connected graph $H$, a subset $A$ of $\cC_H$ is \emph{admissible} to $H$ if the following hold:
\begin{itemize}
\item[(1)] $|A\cap V(H)|\equiv 0\pmod{2}$ and each vertex incident to only simple edges of $H$ is contained in $A$,
\item[(2)] $B\cap A\neq \emptyset$ and $|B\cap A|\equiv 0\pmod{2}$ for each bundle $B$ of $H$.
\end{itemize}
For a disconnected graph $H$, $A\subset \cC_H$ is \emph{admissible} to $H$ if  $\mathcal{C}_{H'}\cap A$  is admissible to  each component $H'$ of $H$.
\end{definition}
We denote by $\mathcal{A}(H)$ the set of all the admissible collections of  $H$.
For each $A\in \mathcal{A}(H)$, a semi-induced subgraph $I$ of $H$ is  said to be $A$-\emph{even} if $|I'\cap A|$ is even  for each component $I'$ of $I$.  For the graphs $H_i$'s in Figure~\ref{fig:example_C_even_poset},
\[\begin{array}{lll}
  \mathcal{A}(H_1)=\{1234\}, &\mathcal{A}(H_2)=\{34ab, 1234ab\},&\mathcal{A}(H_3)=\{14cd, 14ce, 14de, 1234cd, 1234ce, 1234de\},\\
  \mathcal{A}(H_4)=\emptyset, &\mathcal{A}(H_5)=\{12ab, 23ab\},
  &\mathcal{A}(H_6)=\{12cd, 12ce, 12de, 13cd, 13ce, 13de\}.
\end{array}\]

For $A\subset \cC_H$, we define the poset $\pP{H,A}$ as follows.
If $A \in \mathcal{A}(H)$,  the poset $\pP{H,A}$ is defined to the poset consisting of all $A$-even semi-induced subgraphs of $H$ ordered by subgraph containment, including  both $\emptyset$ and $H$, and hence $\pP{H,A}$ is a bounded poset.
If $A\not \in \mathcal{A}(H)$, then we define $\pP{H,A}$ by the null poset.
Figure~\ref{fig:example_C_even_poset} shows the posets $\mathcal{P}_{H_1,1234}^{\mathrm{even}}$, $\mathcal{P}_{H_2,1234ab}^{\mathrm{even}}$ and $\mathcal{P}_{H_3,1234cd}^{\mathrm{even}}$.
Note that the first two  posets are shellable but  the last one is not.
For more examples of $\mathcal{P}_{H,A}^\mathrm{even}$, see also Figure~\ref{fig:examples of even posets}.

Let $H$ be a simple graph. Then $\mathcal{A}(H)=\{H\}$ if each component of $H$ has an even number of vertices, and  $\mathcal{A}(H)=\emptyset$ otherwise.
Thus we write $\mathcal{P}_H^{\mathrm{even}}$ instead of $\pP{H,H}$.
In \cite{CP}, it is shown that $\mathcal{P}_H^{\mathrm{even}}$ is always shellable.

\begin{theorem}[\cite{CP}]\label{prop:simple-shellable}
 For each simple graph~$H$,
    $\mathcal{P}_H^{\mathrm{even}}$ is pure and totally semimodular, and so it is shellable.
\end{theorem}

Since a pure  and totally semimodular poset is CL-shellable by (\ref{(4)}) of Theorem~\ref{thm:product_shellable}, $\mathcal{P}_H^{\mathrm{even}}$
is CL-shellable for a simple graph~$H$.
\begin{remark}\label{rmk:mu-simple-graph}
In~\cite{CP}, Theorem~\ref{prop:simple-shellable} is used to determine  the  homotopy type of the order complex $\Delta(\overline{\mathcal{P}_H^{\mathrm{even}}})$. Finally, $\Delta(\overline{\mathcal{P}_H^{\mathrm{even}}})$ is homotopy equivalent to a wedge of the same dimensional spheres, and  the M\"{o}bius  invariant\footnote{The M\"{o}bius function $\mu$, introduced
    by Rota in~\cite{rota}, is inductively defined as follows: for a poset $\mathcal{P}$, for elements $x$ and $y$ in $\mathcal{P}$,
    \[
        \mu_\mathcal{P}(x,y) =
        \begin{cases}
            {}\qquad 1 & \textrm{if}\quad x = y\\[6pt]
            \displaystyle -\sum_{z\, :\,  x\leq z <y} \mu_\mathcal{P}(x,z) & \textrm{if} \quad x<y \\[6pt]
            {}\qquad 0 & \textrm{otherwise}.
        \end{cases}
    \]
 For a bounded poset $\mathcal{P}$, the \emph{M\"{o}bius invariant} is defined as $\mu(\mathcal{P})=\mu_{\mathcal{P}}(\hat{0},\hat{1})$.
See~\cite{Stanley} for various techniques for computing the M\"{o}bius function of a poset.}  $\mu(\mathcal{P}_H^{\mathrm{even}})$
 is equal to the ($\ell-2$)th Betti number of  $\Delta(\overline{\mathcal{P}_H^{\mathrm{even}}})$, where $\ell$ is the length of the poset $\pP{H}$.
For example, when $H$ is a simple path $P_{2n}$ with $2n$ vertices, $\mu(\mathcal{P}_{H}^{\mathrm{even}})=(-1)^n C_n$,  where $C_n$ is  the  $n$th Catalan number $\frac{1}{n+1}{2n \choose n}$,  and hence  $\Delta(\overline{\mathcal{P}_H^{\mathrm{even}}})$ is homotopy equivalent to {\tiny{$\displaystyle\bigvee_{C_n}$}}$S^{n-2}$.
\end{remark}

In \cite{CPP2015}, there was an effort to extend results of~\cite{CP} for a simple graph to a graph  allowing multiple edges. Almost all results of \cite{CP} except for  Theorem~\ref{prop:simple-shellable}  were well-extended by using  $\mathcal{P}_{H,A}^{\mathrm{even}}$ where $H$ is a PI-graph of $G$ and $A\in\mathcal{A}(H)$. As the poset $\pP{H_3,1234cd}$ in Figure~\ref{fig:example_C_even_poset} is not shellable,  Theorem~\ref{prop:simple-shellable} cannot be generalized to $\mathcal{P}_{H,A}^{\mathrm{even}}$. Hence it is natural to ask  which  $\pP{H,A}$ is shellable. From an interest of the topology of a real toric manifold associated with a graph, the following Question~\ref{question1} was asked in~\cite{CPP2015}, instead of asking the conditions on $(H,A)$ {to give a shelling of} $\pP{H,A}$.
For a graph $G$, let
$\mathcal{A}^\ast(G)=\{ (H,A)\mid H \text{ is a PI-graph of $G$ and }A\in \mathcal{A}(H) \}.$

\bigskip

\noindent\textbf{Question~\ref{question1}} (\cite{CPP2015})
Find all graphs $G$ such that $\pP{H,A}$ is shellable for every $(H,A) \in \mathcal{A}^{\ast}(G)$.

\bigskip

For simplicity, throughout the paper, let $\mathcal{G}^{\ast}$ be the family of all graphs $G$ such that $\pP{H,A}$ is shellable for every $(H,A) \in \mathcal{A}^{\ast}(G)$.
{We will give some remark that $\mathcal{G}^\ast$ is distinct from the set of all graphs $G$ such that $\pP{G,A}$ is shellable for each $A\in\mathcal{A}(G)$ in Section~\ref{sec:last}.}
Clearly, the family $\mathcal{G}^\ast$  contains all simple graphs by Theorem~\ref{prop:simple-shellable}.
The  answer to  Question~\ref{question1}  is  the following, which restates Theorem~\ref{main}.

\bigskip

\noindent\textbf{Theorem~\ref{main}} (Main result)
\emph{A graph $G$ is in $\mathcal{G}^*$
if and only if each component of $G$ is either a simple graph or one of the graphs in Figure~\ref{fig:list of possible graphs}.}
\begin{figure}[h]
\begin{subfigure}[t]{.3\textwidth}
    \centering
    \begin{tikzpicture}[scale=.6]
    \draw (-1,0) node{\tiny$\tilde{P}_{n,m}$}     (-1,-.5) node{\tiny$(n\ge 2)$};
    \fill (0,0) circle (3pt) (1,0) circle (3pt) (2,0) circle (3pt) (3,0) circle (3pt) (4,0) circle (3pt) (5,0) circle (3pt) (6,0) circle (3pt);
    \draw (1,0)--(3.2,0);
    \draw (3.8,0)--(6,0);
    \draw plot [smooth] coordinates {(0,0) (0.5,0.4) (1,0)};
    \draw plot [smooth] coordinates {(0,0) (0.5,-0.6) (1,0)};
    \draw plot [smooth] coordinates {(0,0) (0.5,0.6) (1,0)};
    \draw (0.5,0.1) node{$\vdots$};
    \draw[dotted] (3.2,0)--(3.8,0);
    \end{tikzpicture}
\end{subfigure}\quad
\begin{subfigure}[t]{.3\textwidth}
    \centering
    \begin{tikzpicture}[scale=.6]
    \draw (-1,0) node{\tiny$\tilde{S}_{n,m}$} (-1.3,-.5) node{\tiny$(n\ge 5, odd)$};
    \fill (0,0) circle (3pt) (1,0) circle (3pt) (2,0) circle (3pt) (3,0) circle (3pt) (4,0) circle (3pt) (5,0) circle (3pt) (6,0) circle (3pt) (5,-1) circle (3pt);
    \draw (1,0)--(3.2,0);
    \draw (3.8,0)--(6,0);
    \draw (5,0)--(5,-1);
    \draw plot [smooth] coordinates {(0,0) (0.5,0.4) (1,0)};
    \draw plot [smooth] coordinates {(0,0) (0.5,-0.6) (1,0)};
    \draw plot [smooth] coordinates {(0,0) (0.5,0.6) (1,0)};
    \draw (0.5,0.1) node{$\vdots$};
    \draw[dotted] (3.2,0)--(3.8,0);
    \end{tikzpicture}
\end{subfigure}\quad
\begin{subfigure}[t]{.3\textwidth}
    \centering
    \begin{tikzpicture}[scale=.6]
    \draw (-1,0) node{\tiny$\tilde{T}_{n,m}$}   (-1.3,-.5) node{\tiny$(n\ge 5, odd)$};
    \fill (0,0) circle (3pt) (1,0) circle (3pt) (2,0) circle (3pt) (3,0) circle (3pt) (4,0) circle (3pt) (5,0) circle (3pt) (6,0) circle (3pt) (5,-1) circle (3pt);
    \draw (1,0)--(3.2,0);
    \draw (3.8,0)--(6,0);
    \draw (5,0)--(5,-1);
    \draw (5,-1)--(6,0);
    \draw plot [smooth] coordinates {(0,0) (0.5,0.4) (1,0)};
    \draw plot [smooth] coordinates {(0,0) (0.5,-0.6) (1,0)};
    \draw plot [smooth] coordinates {(0,0) (0.5,0.6) (1,0)};
    \draw (0.5,0.1) node{$\vdots$};
    \draw[dotted] (3.2,0)--(3.8,0);
    \end{tikzpicture}
\end{subfigure}

\begin{subfigure}[t]{.3\textwidth}
    \centering
    \begin{tikzpicture}[scale=.6]
    \draw (-1,0) node{\tiny$\tilde{P}'_{n,m}$} (-1,-.5) node{\tiny$(n\ge 3)$};
    \fill (0,0) circle (3pt) (1,0) circle (3pt) (2,0) circle (3pt) (3,0) circle (3pt) (4,0) circle (3pt) (5,0) circle (3pt) (6,0) circle (3pt);
    \draw (1,0)--(3.2,0);
    \draw (3.8,0)--(6,0);
    \draw plot [smooth] coordinates {(0,0) (0.5,0.4) (1,0)};
    \draw plot [smooth] coordinates {(0,0) (0.5,-0.6) (1,0)};
    \draw plot [smooth] coordinates {(0,0) (0.5,0.6) (1,0)};
    \draw plot [smooth] coordinates {(0,0) (0.4,-0.8) (1,-1) (1.6,-0.8) (2,0)};
    \draw (0.5,0.1) node{$\vdots$};
    \draw[dotted] (3.2,0)--(3.8,0);
    \end{tikzpicture}
\end{subfigure}\quad
\begin{subfigure}[t]{.3\textwidth}
    \centering
    \begin{tikzpicture}[scale=.6]
          \draw (-1,0) node{\tiny$\tilde{S}'_{n,m}$} (-1.3,-.5) node{\tiny$(n\ge 5, odd)$};
    \fill (0,0) circle (3pt) (1,0) circle (3pt) (2,0) circle (3pt) (3,0) circle (3pt) (4,0) circle (3pt) (5,0) circle (3pt) (6,0) circle (3pt) (5,-1) circle (3pt);
    \draw (1,0)--(3.2,0);
    \draw (3.8,0)--(6,0);
    \draw (5,0)--(5,-1);
    \draw plot [smooth] coordinates {(0,0) (0.5,0.4) (1,0)};
    \draw plot [smooth] coordinates {(0,0) (0.5,-0.6) (1,0)};
    \draw plot [smooth] coordinates {(0,0) (0.5,0.6) (1,0)};
    \draw plot [smooth] coordinates {(0,0) (0.4,-0.8) (1,-1) (1.6,-0.8) (2,0)};
    \draw (0.5,0.1) node{$\vdots$};
    \draw[dotted] (3.2,0)--(3.8,0);
    \end{tikzpicture}
\end{subfigure}\quad
\begin{subfigure}[t]{.3\textwidth}
    \centering
    \begin{tikzpicture}[scale=.6]
       \draw (-1,0) node{\tiny$\tilde{T}_{n,m}^\prime$ }   (-1.3,-.5) node{\tiny$(n\ge 5, odd)$};
    \fill (0,0) circle (3pt) (1,0) circle (3pt) (2,0) circle (3pt) (3,0) circle (3pt) (4,0) circle (3pt) (5,0) circle (3pt) (6,0) circle (3pt) (5,-1) circle (3pt);
    \draw (1,0)--(3.2,0);
    \draw (3.8,0)--(6,0);
    \draw (5,0)--(5,-1);
    \draw (5,-1)--(6,0);
    \draw plot [smooth] coordinates {(0,0) (0.5,0.4) (1,0)};
    \draw plot [smooth] coordinates {(0,0) (0.5,-0.6) (1,0)};
    \draw plot [smooth] coordinates {(0,0) (0.5,0.6) (1,0)};
    \draw plot [smooth] coordinates {(0,0) (0.4,-0.8) (1,-1) (1.6,-0.8) (2,0)};
    \draw (0.5,0.1) node{$\vdots$};
    \draw[dotted] (3.2,0)--(3.8,0);
    \end{tikzpicture}
\end{subfigure}    \captionsetup{width=1.0\linewidth}
\caption{Non-simple connected graphs in $\mathcal{G}^*$ with $n$ vertices and $m$ multiple edges ($m\ge 2$)}
\label{fig:list of possible graphs}
\end{figure}
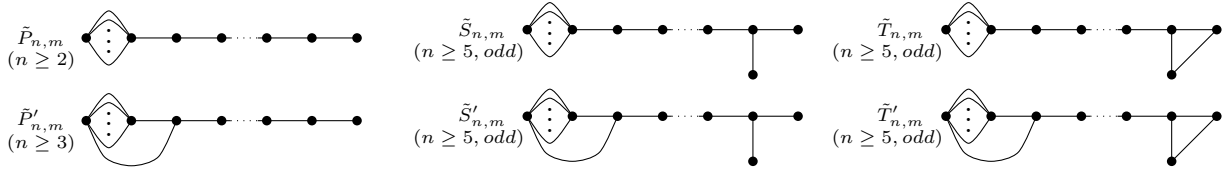

Theorem~\ref{main} is not only a generalization of Theorem~\ref{prop:simple-shellable}, but also provides connections to a real toric manifold associated with a pseudograph associahedron, see
Section~\ref{sec:application}

As an immediate consequence of the proof in
Section~\ref{sec:CL-shellable}, we also get the following:

\begin{theorem}\label{main2}
    For every $G\in \mathcal{G}^\ast$, each $\pP{H,A}$ is $CL$-shellable for every $(H,A)\in\mathcal{A}^{\ast}(G)$.
\end{theorem}
We finish the section by giving a remark that it is sufficient to {consider} a connected graph to prove Theorem~\ref{main} and Theorem~\ref{main2}. To see why, let  $G_1$, $\ldots,$ $G_k$ be the components of a graph $G$.
Note that for a subgraph $H$ of $G$ and $A\in\cC_H$, $(H,A)\in\mathcal{A}^{\ast}(G)$ if and only if
$(H\cap G_i, A\cap \mathcal{C}_{G_i})\in \mathcal{A}^\ast(G_i)$ for each $i$.
Thus for each $(H,A)\in\mathcal{A}^{\ast}(G)$,
$\pP{H,A}$ is isomorphic to the product $\pP{H_1,A_1}\times \cdots \times \pP{H_k,A_k}$, where
$H_i=H\cap G_i$ and $A_i=A\cap \cC_{G_i}$ for each $i$.
By~(\ref{(3)}) of Theorem~\ref{thm:product_shellable}, $\pP{H,A}$ is shellable if and only if $\pP{H_i,A_i}$ is shellable for each $i$. Thus $G\in \mathcal{G}^\ast$ if and only if  $G_i\in \mathcal{G}^\ast$ for each $i$.

\section{Graphs which admit a non-shellable poset $\mathcal{P}_{H,A}^{\mathrm{even}}$}\label{sec:List_Graphs}

In this section, we give the `only if'  part of Theorem~\ref{main}. We will see that almost all graphs do not belong to the family $\mathcal{G}^{\ast}$. The result  of this section is  started  from the following basic observation.

\begin{lemma}\label{lem:non-shellable}
Let $\mathcal{P}_0$ be a poset in Figure~\ref{fig:base_poset} and
$\mathcal{Q}$ be any subposet which has at least two chains of length 3, with one containing $a$ or $b$ and another containing
$a'$ or $b'$.
Then $\mathcal{Q}$ is not shellable.
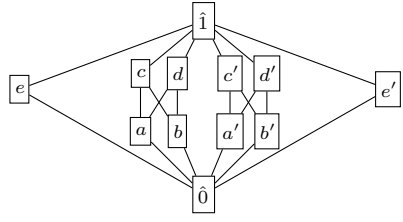
\begin{figure}[h]
    \centering\vspace{-0.2cm}
    \begin{tikzpicture}[scale=.7]
    \node [draw] (0) at (0,-0.7) {\!\tiny$\hat{0}\!$};
    \node [draw] (1) at (0,2.6) {\!\tiny$\hat{1}\!$};
    \node [draw] (a) at (-1.2,.5) {\!\tiny$a^{}\!\!$};
    \node [draw] (b) at (-0.5,.5) {\!\tiny$b^{}\!\!$};
    \node [draw] (c) at (-1.2,1.6) {\!\tiny$c^{}\!\!$};
    \node [draw] (d) at (-0.5,1.6) {\!\tiny$d^{}\!\!$};
    \node [draw] (b') at (1.2,.5) {\!\tiny$b'\!\!$};
    \node [draw] (a') at (0.5,.5) {\!\tiny$a'\!\!$};
    \node [draw] (d') at (1.2,1.6) {\!\tiny$d'\!\!$};
    \node [draw] (c') at (0.5,1.6) {\!\tiny$c'\!\!$};
    \node [draw] (e) at (-3.5,1.3) {\!\tiny$e^{}\!\!$};
    \node [draw] (e') at (3.5,1.3) {\!\tiny$e'\!\!$};
       \path (0) edge (a) edge (b) edge (a') edge (b') edge (e) edge (e');
    \path (1) edge (c) edge (d) edge (c') edge (d') edge (e) edge (e');
    \path (a) edge (c) edge (d);
    \path (b) edge (c) edge (d);
    \path (a') edge (c') edge (d');
    \path (b') edge (c') edge (d');
    \end{tikzpicture}\vspace{-0.2cm}
    \caption{The poset $\mathcal{P}_0$}
    \label{fig:base_poset}
\end{figure}
\end{lemma}

\begin{theorem}\label{thm:list_graphs}
Let $G$ be a connected non-simple graph in $\mathcal{G}^\ast$.
Then $G$ is one of the graphs in Figure~\ref{fig:list of possible graphs}.
\end{theorem}

Before starting the proof, recall that we often drop the braces and commas to denote a subset of~$\cC_G$.

\begin{proof}
Suppose that $G$ is a connected non-simple graph in $\mathcal{G}^\ast$.
If $|V(G)|=2$, then $G=\tilde{P}_{2,m}$ in Figure~\ref{fig:list of possible graphs} for some $m$.
Assume that $|V(G)|\ge 3$ and $G$ has a bundle~$B$ whose endpoints are
$1$ and $2$.
\begin{claim}\label{claim:bundle}
 The graph  $G$ has exactly one bundle $B$.
\end{claim}

\begin{proof}[Proof of Claim~\ref{claim:bundle}]
Suppose that $G$ has a bundle $B'$ other than $B$.
Take a shortest path $Q$ in $G$  whose starting vertex is an endpoint of $B$ and  whose  terminal vertex is an endpoint of $B'$.
Note that $Q$ does not contain a multiple edge.
Let $Q:=(v_1,\ldots, v_k)$, where $k\ge 1$, and let $v_1=2$ without loss of generality.
Let $H$ be a  PI-graph of $G$ such that $V(H)=V(Q)\cup\{1,2\}\cup\{\text{endpoints of }B'\}$ and $H$ has exactly two bundles $B$ and $B'$.
Let $a,b\in B$ and $a',b'\in B'$.
\smallskip

(Case 1) Suppose that $k=1$.
Then $|V(H)|=3$, so we set $V(H)=\{1,2,3\}$.
Then $A:=23aba'b'$  belongs to $\mathcal{A}(H)$.
  Setting $I=123aba'b'$ (the dotted edge  in Figure~\ref{fig:unique_bundle_k=1} is a simple edge or does not exist), we see  $I\cap A=A$ and hence  $I$ is an element of  $\pP{H,A}$.
Let $I'=1$ and consider  the  interval $\mathcal{I}=[I',I]$ of $\pP{H,A}$. Then $\mathcal{I}$ is a subposet of $\mathcal{P}_0$ in Figure~\ref{fig:base_poset} {as in} Figure~\ref{fig:unique_bundle_k=1}.
By Lemma~\ref{lem:non-shellable},  $\mathcal{I}$ is not shellable, a contradiction to (\ref{(2)}) of Theorem~\ref{thm:product_shellable}.

    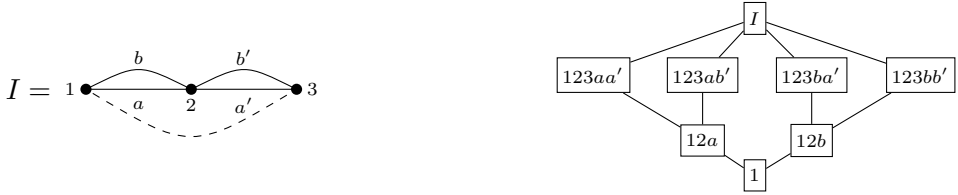
\begin{figure}[h]\centering
    \begin{subfigure}{.3\textwidth}\centering
    \begin{tikzpicture}[scale=.7]
            \draw (-1,0) node{$I = $ \quad};
        	\fill (0,0) circle(3pt);
        	\fill (2,0) circle(3pt);
        	\fill (4,0) circle(3pt);
        	\draw (0,0)--(4,0);
        	\draw (0,0)..controls (1,0.5)..(2,0);
        	\draw (2,0)..controls (3,0.5)..(4,0);
        	\draw[dashed,dash phase=3pt] (0,0)..controls (2,-1.2)..(4,0);
        	\draw (-0.3,0) node{\tiny$1$};
        	\draw (2,-0.3) node{\tiny$2$};
        	\draw (4.3,0) node{\tiny$3$};
        	\draw (1,-0.3) node{\tiny$a$};
        	\draw (1,0.6) node{\tiny$b$};
        	\draw (3,-0.3) node{\tiny$a'$};
        	\draw (3,0.6) node{\tiny$b'$};
        \end{tikzpicture}
        \end{subfigure}
        \begin{subfigure}[h]{.6\textwidth}\centering
    	\begin{tikzpicture}[scale=.58]
            \node [draw] (0) at (3.7,-0.8) {\!\tiny$1$\!};
        	\node [draw] (23) at (2.5,0) {\!\tiny$12a$\!};
        	\node [draw] (14) at (5,0) {\!\tiny$12b$\!};
        	\node [draw] (123a) at (0,1.5) {\!\tiny$123aa'$\!};
        	\node [draw] (123b) at (2.5,1.5) {\!\tiny$123ab'$\!};
        	\node [draw] (124a) at (5,1.5) {\!\tiny$123ba'$\!};
        	\node [draw] (124b) at (7.5,1.5) {\!\tiny$123bb'$\!};
            \node [draw] (G) at (3.7, 2.8) {\!\tiny$I$\!};
            \path (0)  edge (23) edge (14);
          \path (G) edge (123a) edge (123b) edge (124a) edge (124b);	
        	\path (23) edge (123a) edge (123b);
        	\path (14) edge (124a) edge (124b);
    	\end{tikzpicture}
    \end{subfigure}
    \caption{A  graph  $I$ and the interval $\mathcal{I}$}\label{fig:unique_bundle_k=1}
    \end{figure}

\smallskip

(Case 2) Suppose that $k\ge 2$. Let the endpoints of $B'$ be labeled by $3$ and $4$, and $v_k=3$.
Let \[A=\begin{cases}
(V(H)\setminus\{1\}) \cup  a  b a' b'  &\text{ if }k\text{ is odd}; \\
(V(H)\setminus\{1,2\}) \cup  aba'b'   &\text{ if }k\text{ is even}.\end{cases}\]
Note that $A\in\mathcal{A}(H)$.
Let $I'=V(Q)\setminus\{v_k\}$, and $I=I'\cup 134aba'b'$. Then $I'\cap A=\{v_1,\ldots,v_{k-1}\}$ (if $k$ is odd) or $I'\cap A=\{v_2,\ldots,v_{k-1}\}$ (if $k$ is even). Then they have the form in Figure~\ref{fig:unique_bundle_kge2} (the dotted edges  are simple edges or do not exist), and both $I'$ and $I$ are elements of $\pP{H,A}$.
Consider the interval $\mathcal{I}=[I',I]$ in $\pP{H,A}$.
Thus $\mathcal{I}$ is a subposet of~$\mathcal{P}_0$  in Figure~\ref{fig:base_poset} {as in} Figure~\ref{fig:unique_bundle_kge2}. Note that   $I'\cup 134aa'$,  $I'\cup 134ab'$, $I'\cup 134ba'$, $I'\cup 134bb'$ are elements in $\mathcal{I}$, and both $I'\cup 13a$ and $I'\cup 13b$ are also elements in $\mathcal{I}$. The elements $I'\cup14a$ and $I'\cup 14b$ in the dotted boxes of Figure~\ref{fig:unique_bundle_kge2} are in $\mathcal{I}$ if there is an edge between the vertex 4 and a vertex in $I'$.
By Lemma~\ref{lem:non-shellable}, $\mathcal{I}$ is not shellable, a contradiction to  (\ref{(2)}) of Theorem~\ref{thm:product_shellable}.
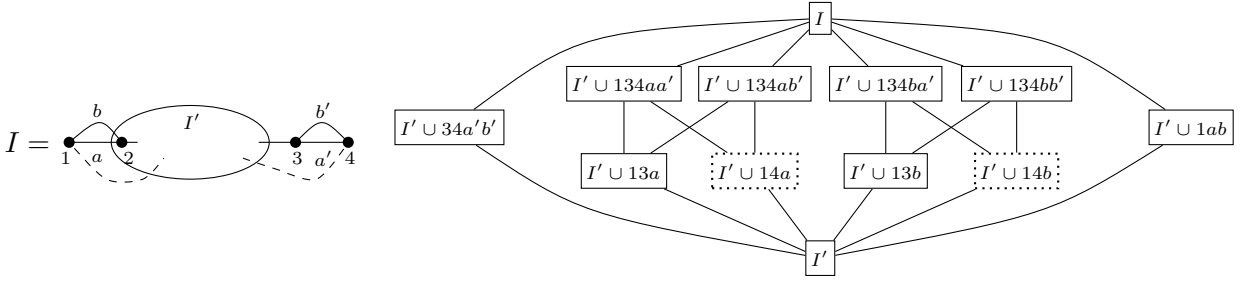
\begin{figure}[h]
\hspace{-1cm}
        \begin{subfigure}[h]{.3\textwidth}
    	\begin{tikzpicture}[scale=.7]
            \draw (-3,0) node{$I = $ \quad};
        	\fill (-2.3,0) circle(3pt);
        	\fill (-1.3,0) circle(3pt);
        	\fill (2,0) circle(3pt);
        	\fill (3,0) circle(3pt);
        \draw  (0,0) ellipse (1.5 and 0.7);
        	\draw (-2.3,0)--(-1,0);
        	\draw (-2.3,0)..controls (-1.7,0.5)..(-1.3,0);
        \draw[dashed, dash phase=3pt] plot [smooth] coordinates {(-2.3,0) (-1.6,-0.7) (-0.9,-0.7)(-0.5,-0.3)};
        \draw[dashed, dash phase=3pt] plot [smooth] coordinates {(3,0) (2.4,-0.7) (1.6,-0.5)(1.0,-0.3)};
        	\draw (2,0)..controls (2.5,0.5)..(3,0);
        	\draw (1.3,0) -- (3,0);
        	\draw (-2.4,-0.3) node{\tiny$1$\!};
        	\draw (-1.2,-0.3) node{\tiny$2$\!};
        	\draw (2,-0.3) node{\tiny$3$\!};
        	\draw (3,-0.3) node{\tiny$4$\!};
        	\draw (-1.8,0.6) node{\tiny$b$\!};
        	\draw (-1.8,-0.3) node{\tiny$a$\!};
        	\draw (2.5,0.6) node{\tiny$b'$\!};
        	\draw (2.5,-0.3) node{\tiny$a'$\!};
        	\draw (0,0.4) node{\tiny$I'$\!};
 \end{tikzpicture}
        \end{subfigure}
        \begin{subfigure}[h]{.6\textwidth}
    	\begin{tikzpicture}[scale=.58]
            \node [draw] (0) at (4.5,-2) {\!\tiny $I'$\!};
        	\node [draw] (13) at (0,0) {\!\tiny$I'\cup 13a$\!};
        	\node [draw,dotted,thick] (23) at (3,0) {\!\tiny$I'\cup 14a$\!};
        	\node [draw] (14) at (6,0) {\!\tiny$I'\cup 13b$\!};
        	\node [draw,dotted,thick] (24) at (9,0) {\!\tiny$I'\cup 14b$\!};
            \node [draw] (34) at (13,1) {\!\tiny$I'\cup 1ab$\!};
        	\node [draw] (123a) at (0,2) {\!\tiny$I'\cup 134aa'$\!};
        	\node [draw] (123b) at (3,2) {\!\tiny$I'\cup 134ab'$\!};
        	\node [draw] (124a) at (6,2) {\!\tiny$I'\cup 134ba'$\!};
        	\node [draw] (124b) at (9,2) {\!\tiny$I'\cup 134bb'$\!};
            \node [draw] (12ab) at (-4,1) {\!\tiny$I'\cup 34a'b'$\!};
            \node [draw] (G) at (4.5, 3.5) {\!\tiny$I$\!};
            \path (0) edge (13) edge (23) edge (14) edge (24);
          \path (G) edge (123a) edge (123b) edge (124a) edge (124b);	
       	                \draw (G)..controls (9.8,3.2)..(34);
                    	\draw (G)..controls (-1,3.2)..(12ab);
                       	\draw (0)..controls (9.8,-1)..(34);
                    	\draw (0)..controls (-1.,-1)..(12ab);
        \path (13) edge (123a) edge (123b);
        	\path (23) edge (123a) edge (123b);
        	\path (14) edge (124a) edge (124b);
        	\path (24) edge (124a) edge (124b);
          \end{tikzpicture}
        \end{subfigure}
            \captionsetup{width=1.0\linewidth}
    \caption{A  graph  $I$ and the interval $\mathcal{I}$  where the dotted boxes may be in $\mathcal{I}$}\label{fig:unique_bundle_kge2}
    \end{figure}
\end{proof}
Hence $G$ has the only one bundle $B$.
{If $|V(G)|=3$, then  clearly $G$ is one of the graphs in Figure~\ref{fig:list of possible graphs}. Now assume that $|V(G)|\ge 4$.}
For each vertex $i$, we let $N^{\ast}(i)=N_G(i)\setminus\{1,2\}$, where $N_G(i)$ is the set of vertices which are adjacent to $i$ in $G$.

\begin{claim}\label{claim:neighbor}
    $|N^{\ast}(1)\cup N^\ast(2)|= 1$.
\end{claim}

\begin{proof}[Proof of Claim~\ref{claim:neighbor}] Since $|V(G)|\ge 3$ and $G$ is connected, $|N^{\ast}(1)\cup N^\ast(2)| \ge 1$.
Suppose that  $|N^{\ast}(1)\cup N^\ast(2)| \ge 2$, and $ 3,4\in N^{\ast}(1)\cup N^\ast(2)$.
Let $H$ be a PI-graph  of $G$ such that $V(H)=\{1,2,3,4\}$ and $H$ has the bundle $B$.   Let $A=1234ab$ for some $a,b\in B$. Note that $A\in\mathcal{A}(H)$. Let $I=1234ab$, and consider the interval $\mathcal{I}=[\emptyset, I]$ in $\pP{H,A}$.
Then $I$ is a subgraph of a complete graph of four vertices with exactly one bundle of size two, and $\mathcal{I}$ is a subposet of  $\mathcal{P}_0$
{as in} Figure~\ref{fig:claim_last}. Note $123a$, $123b$, $124a$, $124b$ are elements of $\mathcal{I}$.
Since the vertex $3$ is a neighbor of  $1$ or  $2$, at least one of $13$ and $23$ is an element of $\mathcal{I}$ (the elements $13$ and $23$ are drawn in a dotted box in Figure~\ref{fig:claim_last}).
Similarly, since the vertex $4$ is also a neighbor of  $1$ or $2$, at least one of $14$ and $24$ is an element of  $\mathcal{I}$ (the elements $14$ and $24$ are drawn in a dotted box in Figure~\ref{fig:claim_last}).
By Lemma~\ref{lem:non-shellable}, $\mathcal{I}$ is not shellable, a contradiction to (\ref{(2)}) of Theorem~\ref{thm:product_shellable}.
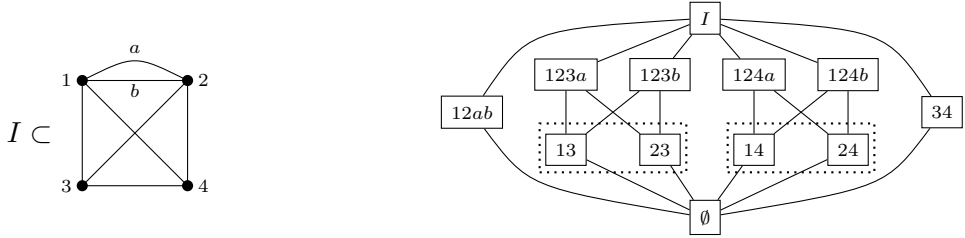
\begin{figure}[h]
    \centering
    \begin{subfigure}{.3\textwidth}
    \centering
    \begin{tikzpicture}[scale=.7]
            \draw (-1,1) node{$I \subset$};
        	\fill (0,0) circle(3pt);
        	\fill (2,0) circle(3pt);
        	\fill (0,2) circle(3pt);
        	\fill (2,2) circle(3pt);
        	\draw (0,0)--(0,2)--(2,2)--(2,0)--cycle;
            \draw (0,0)--(2,2);
            \draw (0,2)--(2,0);
        	\draw (0,2)..controls (1,2.5)..(2,2);
        	\draw (-0.3,0) node{\tiny$3$};
        	\draw (2.3,0) node{\tiny$4$};
        	\draw (-0.3,2) node{\tiny$1$};
        	\draw (2.3,2) node{\tiny$2$};
        	\draw (1,2.6) node{\tiny$a$};
        	\draw (1,1.8) node{\tiny$b$};
        \end{tikzpicture}
        \end{subfigure}
        \begin{subfigure}[h]{.6\textwidth}
        \centering
    	\begin{tikzpicture}[scale=.5]
            \node [draw] (0) at (3.7,-1.8) {\tiny $\emptyset$};
        	\node [draw] (13) at (0,0) {\tiny$13$};
        	\node [draw] (23) at (2.5,0) {\tiny$23$};
        	\node [draw] (14) at (5,0) {\tiny$14$};
        	\node [draw] (24) at (7.5,0) {\tiny$24$};
            \node [draw] (34) at (10,1) {\tiny$34$};
        	\node [draw] (123a) at (0,2) {\tiny$123a$};
        	\node [draw] (123b) at (2.5,2) {\tiny$123b$};
        	\node [draw] (124a) at (5,2) {\tiny$124a$};
        	\node [draw] (124b) at (7.5,2) {\tiny$124b$};
            \node [draw] (12ab) at (-2.5,1) {\tiny$12ab$};
            \node [draw] (G) at (3.7, 3.5) {\tiny $I$};
            \draw[dotted,thick]    (-0.7,-0.6) rectangle  (3.2,0.7);
            \draw[dotted,thick]    (4.3,-0.6) rectangle  (8.2,0.7);
            \path (0) edge (13) edge (23) edge (14) edge (24) ;
            \path (G) edge (123a) edge (123b) edge (124a) edge (124b) ;	
            \path (13) edge (123a) edge (123b);
        	\path (23) edge (123a) edge (123b);
        	\path (14) edge (124a) edge (124b);
        	\path (24) edge (124a) edge (124b);
            \draw (G)..controls (8.5,3.2)..(34);
        	\draw (G)..controls (-1,3.2)..(12ab);
           	\draw (0)..controls (8.5,-1)..(34);
        	\draw (0)..controls (-1.,-1)..(12ab);
    	\end{tikzpicture}
    \end{subfigure}
    \captionsetup{width=1.0\linewidth}
    \caption{A graph containing $I$ and the interval $\mathcal{I}$  where at least one of the elements in each dotted box is in $\mathcal{I}$}\label{fig:claim_last}
\end{figure}
\end{proof}
From now on, we set $N^{\ast}(1)\cup N^\ast(2)=N^\ast(2)=\{3\}$.
\begin{claim}\label{claim:tail} For each vertex $i$ other than  $1$ or $2$, let $Q_i$ be a shortest path of $G$ from $3$ to $i$. Then
$$|N^\ast(i)\setminus V(Q_i)|\le 2,$$ where the equality holds if and only if $|V(Q_i)|$ is odd  and $V(G)=V(Q_i)\cup\{1,2\}\cup N^\ast(i)$.
\end{claim}

\begin{proof}[Proof of Claim~\ref{claim:tail}] Suppose that there is a vertex $i\in V(G)\setminus\{1,2\}$ satisfying one of the following:
\begin{itemize}
  \item[(1)] $|N^\ast(i)\setminus V(Q_i)  |\ge 3$;
  \item[(2)] $|N^\ast(i)\setminus V(Q_i)|=2$  and  $|V(Q_i)|$  is even;
  \item[(3)] $|N^\ast(i)\setminus V(Q_i)|=2$, $|V(Q_i)|$ is odd, and $V(G)\neq V(Q_i)\cup\{1,2\}\cup N^\ast(i)$.
\end{itemize}
{If $|V(Q_i)|$ is even then we set $I'=Q_i$, and if $|V(Q_i)|$ is odd then we set $I'= Q_i\cup\{w\}$ by taking some vertex $w\in N^\ast(i)\setminus V(Q_i)$.}
Then  $3\in I'$, $I'\cap\{1,2\}=\emptyset$,  $|I'|$ is even, and $I'$ is {a} connected subgraph of $G$. Furthermore, there are two vertices $x$ and $y$ in  $V(G)\setminus(I'\cup\{1,2\})$  such that both $I'\cup x$ and
$I'\cup  y$  are connected. More precisely, for  the cases of (1) and (2), $x$ and $y$ are selected from $N^\ast(i)\setminus  V(I') $. For  the case of  (3), $x$ is selected from  $N^\ast(i)\setminus V(I')$
and $y$ is a vertex in $V(G)\setminus \left( V(Q_i)\cup\{1,2\}\cup N^\ast(i)\right)$ which is closest to the vertex~$1$ or~$2$.
Let $H$ be a PI-graph  such that $V(H)=I' \cup  12xy$ and $B$ is the bundle of $H$.
Let $A=V(H)\cup ab$ and $I=A$ {for some} $a,b\in B$.
Note that $A\in\mathcal{A}(H)$ and $I$ is the graph in the left of Figure~\ref{fig:claim3} (the dotted edges  are simple edges or do not exist). Consider the interval $\mathcal{I}=[I',I]$ in $\pP{H,A}$, and then $\mathcal{I}$ is a subposet of $\mathcal{P}_0$ {as in}  Figure~\ref{fig:claim3}.
Note that $I'\cup 12xa$,  $I'\cup 12xb$, $I'\cup 12ya$, and $I'\cup 12yb$ are elements in $\mathcal{I}$.
Moreover, both $I'\cup 2x$ and $I'\cup 2y$ are in $\mathcal{I}$.
The elements $I'\cup1x$ and $I'\cup 1y$ in the dotted boxes of Figure~\ref{fig:claim3} are in $\mathcal{I}$ if there is an edge between the vertex 1 and a vertex in $I'$.
By Lemma~\ref{lem:non-shellable}, $\mathcal{I}$ is not shellable,  a contradiction to (\ref{(2)}) of Theorem~\ref{thm:product_shellable}.
\end{proof}

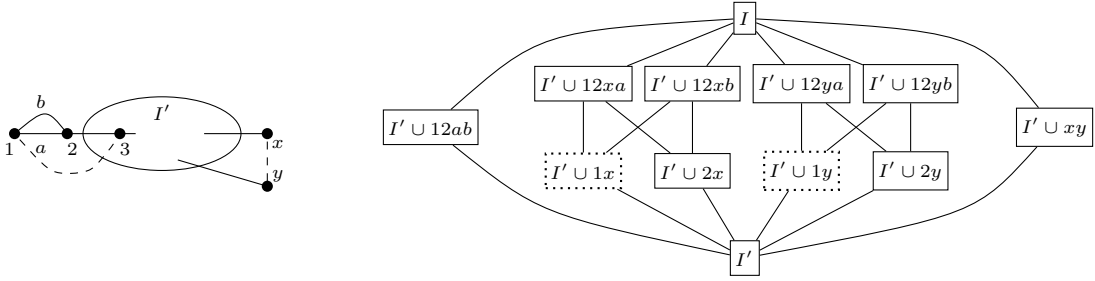
\begin{figure}[h]
     \hspace{-0.5cm}\centering
     \begin{subfigure}[h]{.3\textwidth}\centering
\begin{tikzpicture}[scale=.7]
            \fill (-2.3,0) circle(3pt);
        	\fill (-1.3,0) circle(3pt);
        	\fill (2.5,0) circle(3pt);
        	\fill (2.5,-1) circle(3pt);
        	\fill (-0.3,0) circle(3pt);
        \draw  (0.5,0) ellipse (1.5 and 0.7);
        	\draw (-2.3,0)--(0,0);
        	\draw (0.8,-0.5)--(2.5,-1);
        	\draw (-2.3,0)..controls (-1.7,0.5)..(-1.3,0);
        \draw[dashed, dash phase=3pt] plot [smooth] coordinates {(-2.3,0) (-1.6,-0.7) (-0.9,-0.7)(-0.5,-0.3)(-0.3,0)};
        \draw[dashed, dash phase=3pt] plot [smooth] coordinates {(2.5,0) (2.5,-1)};
        	\draw (1.3,0) -- (2.5,0);
        	\draw (-2.4,-0.3) node{\tiny$1$};
        	\draw (-1.2,-0.3) node{\tiny$2$};
        	\draw (-0.2,-0.3) node{\tiny$3$};
        	\draw (2.7,-0.2) node{\tiny$x$};
        	\draw (2.7,-0.8) node{\tiny$y$};
        	\draw (-1.8,0.6) node{\tiny$b$};
        	\draw (-1.8,-0.3) node{\tiny$a$};
        	\draw (0.5,0.4) node{\tiny$I'$};
        \end{tikzpicture}
        \end{subfigure}
        \begin{subfigure}[h]{.6\textwidth}\centering
    	\begin{tikzpicture}[scale=.58]
            \node [draw] (0) at (3.7,-2) {\!\tiny $I'$\!};
        	\node [draw,dotted,thick] (13) at (0,0) {\!\tiny$I'\cup 1x$\!};
        	\node [draw] (23) at (2.5,0) {\!\tiny$I'\cup 2x$\!};
        	\node [draw,dotted,thick] (14) at (5,0) {\!\tiny$I'\cup 1y$\!};
        	\node [draw] (24) at (7.5,0) {\!\tiny$I'\cup 2y$\!};
            \node [draw] (34) at (10.8,1) {\!\tiny$I'\cup xy$\!};
        	\node [draw] (123a) at (0,2) {\!\tiny$I'\cup 12xa$\!};
        	\node [draw] (123b) at (2.5,2) {\!\tiny$I'\cup 12xb$\!};
        	\node [draw] (124a) at (5,2) {\!\tiny$I'\cup 12ya$\!};
        	\node [draw] (124b) at (7.5,2) {\!\tiny$I'\cup 12yb$\!};
            \node [draw] (12ab) at (-3.5,1) {\!\tiny$I'\cup 12ab$\!};
            \node [draw] (G) at (3.7, 3.5) {\!\tiny $I$\!};
            \path (0) edge (13) edge (23) edge (14) edge (24);
          \path (G) edge (123a) edge (123b) edge (124a) edge (124b);	
               \draw (G)..controls (9,3.2)..(34);
                    	\draw (G)..controls (-1,3.2)..(12ab);
                       	\draw (0)..controls (9,-1)..(34);
                    	\draw (0)..controls (-1.,-1)..(12ab);
        \path (13) edge (123a) edge (123b);
        	\path (23) edge (123a) edge (123b);
        	\path (14) edge (124a) edge (124b);
        	\path (24) edge (124a) edge (124b);
    	\end{tikzpicture}
    \end{subfigure}
           \captionsetup{width=1.0\linewidth}
    \caption{A  graph containing $I$ and the poset  containing $\mathcal{I}$ where the dotted boxes may be in $\mathcal{I}$
    }\label{fig:claim3}
    \end{figure}

Since $|V(G)|\ge 4$, we have $|N^\ast(3)|\ge 1$.
Since $N^\ast(3)\setminus V(Q_3)=N^\ast(3)$, {we see $|N^\ast(3)|\le 2$ by Claim~\ref{claim:tail}.}
If $|N^\ast(3)|=2$, then  {the equality part of Claim~\ref{claim:tail} says that} $G$ is one of $\tilde{S}_{5,m}$, $\tilde{S}'_{5,m}$, $\tilde{T}_{5,m}$, and $\tilde{T}'_{5,m}$ in Figure~\ref{fig:list of possible graphs} for some $m$.
Suppose that $|N^\ast(3)|=1$, and let $N^\ast(3)=\{4\}$.
Since $N^\ast(4)\setminus V(Q_4)=N^\ast(4)\setminus\{3\}$, {we see $|N^\ast(4)\setminus\{3\}|\le 1$ by Claim~\ref{claim:tail}.}
If $|N^\ast(4)\setminus\{3\}|=0$, then $G$ is one of $\tilde{P}_{4,m}$, and $\tilde{P}'_{4,m}$ in  Figure~\ref{fig:list of possible graphs} for some $m$. Suppose that $|N^\ast(4)\setminus\{3\}|=1$, and
 let $N^\ast(4)\setminus\{3\}=\{5\}$. Then consider $N^\ast(5)\setminus V(Q_5)$.
Repeating the argument through the vertices one by one completes the proof.
\end{proof}

\section{Shellability of $\mathcal{P}_{G,A}^{\mathrm{even}}$}\label{sec:CL-shellable}

In this section, we show that  the poset $\mathcal{P}_{H,A}^{\mathrm{even}}$ is shellable for every $(H,A)\in\mathcal{A}^\ast(G)$ if $G$ is a graph in Figure~\ref{fig:list of possible graphs}. Note that a connected PI-graph of $G$ in Figure~\ref{fig:list of possible graphs} is a simple graph or a graph in Figure~\ref{fig:list of possible graphs}.  Thus it is sufficient to show that when $G$ is a graph in Figure~\ref{fig:list of possible graphs},
$\mathcal{P}_{G,A}^{\mathrm{even}}$ is shellable for every $A\in\mathcal{A}(G)$.
From now on, throughout this section, we fix a graph $G$ with $n$ vertices and $m$ multiple edges in Figure~\ref{fig:list of possible graphs}, and an admissible collections $A\in\mathcal{A}(G)$.
\subsection{Definition of an ordering $\atomprec{I}$ for the atoms of $[I,G]$}\label{subsec:def:ordering}
We let $V=\{1,2,\ldots,n\}$ $(n\ge 2)$ be the set of vertices of $G$, and $1$ and $2$ be the endpoints of the bundle $B$. By the definition of an admissible collection, note that $A\cap B\neq\emptyset$ and $|A\cap B|$ is even, and so we let $B\cap A=\{a_1,\ldots,a_{2m}\}$  ($m\ge 1$), and $B\setminus A=\{b_1,\ldots,b_\ell\}$. Here, $B\setminus A$  may be the empty set.
In addition, there are three cases:
\begin{itemize}
\item $|V|$ is odd and $ V\cap A = V \setminus \{w\} $ for some $w\in\{1,2\}$;
\item $|V|$ is even and $ V\cap A = V \setminus \{1,2\}$;
\item$|V|$ is even and $ V\cap A = V$.
\end{itemize}
We label the vertices which are {not the endpoints of $B$}  so that for each $i\in \{3,\ldots,n\}$, the vertex $i$ is closest to the vertex $i-1$.
We relabel the endpoints of $B$ so that $1\not\in A$
if $|V|$ is odd, and so that $13$ is an edge if $|V|$ is even.
See (i) of Figure~\ref{fig:labeling of vertices-1x2x} for all the possible labelings when $|V|$ is odd.
{We illustrate all the possible labelings when  $|V|$ is even} in (ii) of Figure~\ref{fig:labeling of vertices-1x2x}.
See Figure~\ref{fig:examples of even posets} for examples of $\mathcal{P}_{G,A}^{\mathrm{even}}$ under this labeling. We also assume that there is a total ordering between the vertices:~$1\prec 2 \prec \cdots \prec n$. Thus for $I\subset V$, the minimum  of $I$, denoted by $\min(I)$, means the frontmost one in the ordering.

\begin{figure}[t]
\begin{subfigure}[b]{.24\textwidth}
    \centering
 \begin{tikzpicture}[scale=0.55]
   \fill (0,0.6) circle (0pt);
   \fill (0,-1.3) circle (0pt);
    \fill (0,0) circle (3pt) (1,0) circle (3pt) (2,0) circle (3pt) (3,0) circle (3pt) (4,0) circle (3pt) (5,0) circle (3pt) (6,0) circle (3pt);
    \draw (1,0)--(3.2,0);
    \draw (3.8,0)--(6,0);
    \draw plot [smooth] coordinates {(0,0) (0.5,0.4) (1,0)};
    \draw plot [smooth] coordinates {(0,0) (0.5,-0.6) (1,0)};
    \draw plot [smooth] coordinates {(0,0) (0.5,0.6) (1,0)};
    \draw (0.5,0.1) node{$\vdots$};
    \draw[dotted] (3.2,0)--(3.8,0);
    \filldraw[fill=white] (1,0) circle (3pt);
    \draw (0,-0.3) node{\tiny$2$} (1,-0.3) node{\tiny$1$} (2,-0.3) node{\tiny$3$} (3,-0.3) node{\tiny$4$} (4,-0.3) node{\tiny$n\!-\!2$} (5,0.3) node{\tiny$n\!-\!1$} (6,-0.3) node{\tiny$n$};
    \end{tikzpicture}
\end{subfigure}
\begin{subfigure}[t]{.24\textwidth}
    \centering
    \begin{tikzpicture}[scale=0.55]
       \fill (0,0.6) circle (0pt);
   \fill (0,-1) circle (0pt);
    \fill (0,0) circle (3pt) (1,0) circle (3pt) (2,0) circle (3pt) (3,0) circle (3pt) (4,0) circle (3pt) (5,0) circle (3pt) (6,0) circle (3pt) (5,-1) circle (3pt);
    \draw (1,0)--(3.2,0);
    \draw (3.8,0)--(6,0);
    \draw (5,0)--(5,-1);
    \draw plot [smooth] coordinates {(0,0) (0.5,0.4) (1,0)};
    \draw plot [smooth] coordinates {(0,0) (0.5,-0.6) (1,0)};
    \draw plot [smooth] coordinates {(0,0) (0.5,0.6) (1,0)};
    \draw (0.5,0.1) node{$\vdots$};
    \draw[dotted] (3.2,0)--(3.8,0);
    \filldraw[fill=white] (1,0) circle (3pt);
    \draw (0,-0.3) node{\tiny$2$} (1,-0.3) node{\tiny$1$} (2,-0.3) node{\tiny$3$} (3,-0.3) node{\tiny$4$} (4,-0.3) node{\tiny$n\!-\!3$} (5,+0.3) node{\tiny$n\!-\!2$} (6,-0.3) node{\tiny$n$} (5.6,-1) node{\tiny$n\!-\!1$};
    \end{tikzpicture}
\end{subfigure}
\begin{subfigure}[t]{.24\textwidth}
    \centering
    \begin{tikzpicture}[scale=0.55]
       \fill (0,0.6) circle (0pt);
   \fill (0,-1) circle (0pt);
    \fill (0,0) circle (3pt) (1,0) circle (3pt) (2,0) circle (3pt) (3,0) circle (3pt) (4,0) circle (3pt) (5,0) circle (3pt) (6,0) circle (3pt) (5,-1) circle (3pt);
    \draw (1,0)--(3.2,0);
    \draw (3.8,0)--(6,0);
    \draw (5,0)--(5,-1);
    \draw (5,-1)--(6,0);
    \draw plot [smooth] coordinates {(0,0) (0.5,0.4) (1,0)};
    \draw plot [smooth] coordinates {(0,0) (0.5,-0.6) (1,0)};
    \draw plot [smooth] coordinates {(0,0) (0.5,0.6) (1,0)};
    \draw (0.5,0.1) node{$\vdots$};
    \draw[dotted] (3.2,0)--(3.8,0);
    \filldraw[fill=white] (1,0) circle (3pt);
    \draw (0,-0.3) node{\tiny$2$} (1,-0.3) node{\tiny$1$} (2,-0.3) node{\tiny$3$} (3,-0.3) node{\tiny$4$} (4,-0.3) node{\tiny$n\!-\!3$} (5,+0.3) node{\tiny$n\!-\!2$} (6,-0.3) node{\tiny$n$} (5.6,-1) node{\tiny$n\!-\!1$};
    \end{tikzpicture}
\end{subfigure}

\begin{subfigure}[t]{.24\textwidth}
    \centering
    \begin{tikzpicture}[scale=0.55]
       \fill (0,-1.3) circle (0pt);
      \fill (0,0.6) circle (0pt);
    \fill (0,0) circle (3pt) (1,0) circle (3pt) (2,0) circle (3pt) (3,0) circle (3pt) (4,0) circle (3pt) (5,0) circle (3pt) (6,0) circle (3pt);
    \draw (1,0)--(3.2,0);
    \draw (3.8,0)--(6,0);
    \draw plot [smooth] coordinates {(0,0) (0.5,0.4) (1,0)};
    \draw plot [smooth] coordinates {(0,0) (0.5,-0.6) (1,0)};
    \draw plot [smooth] coordinates {(0,0) (0.5,0.6) (1,0)};
    \draw (0.5,0.1) node{$\vdots$};
    \draw[dotted] (3.2,0)--(3.8,0);
    \filldraw[fill=white] (0,0) circle (3pt);
    \draw (0,-0.3) node{\tiny$1$} (1,-0.3) node{\tiny$2$} (2,-0.3) node{\tiny$3$} (3,-0.3) node{\tiny$4$} (4,-0.3) node{\tiny$n\!-\!2$} (5,0.3) node{\tiny$n\!-\!1$} (6,-0.3) node{\tiny$n$};
    \end{tikzpicture}
\end{subfigure}
\begin{subfigure}[t]{.24\textwidth}
    \centering
    \begin{tikzpicture}[scale=0.55]
    \fill (0,0) circle (3pt) (1,0) circle (3pt) (2,0) circle (3pt) (3,0) circle (3pt) (4,0) circle (3pt) (5,0) circle (3pt) (6,0) circle (3pt) (5,-1) circle (3pt);
    \draw (1,0)--(3.2,0);
    \draw (3.8,0)--(6,0);
    \draw (5,0)--(5,-1);
    \draw plot [smooth] coordinates {(0,0) (0.5,0.4) (1,0)};
    \draw plot [smooth] coordinates {(0,0) (0.5,-0.6) (1,0)};
    \draw plot [smooth] coordinates {(0,0) (0.5,0.6) (1,0)};
    \draw (0.5,0.1) node{$\vdots$};
    \draw[dotted] (3.2,0)--(3.8,0);
    \filldraw[fill=white] (0,0) circle (3pt);
    \draw (0,-0.3) node{\tiny$1$} (1,-0.3) node{\tiny$2$} (2,-0.3) node{\tiny$3$} (3,-0.3) node{\tiny$4$} (4,-0.3) node{\tiny$n\!-\!3$} (5,+0.3) node{\tiny$n\!-\!2$} (6,-0.3) node{\tiny$n$} (5.6,-1) node{\tiny$n\!-\!1$};
    \end{tikzpicture}
\end{subfigure}
\begin{subfigure}[t]{.24\textwidth}
    \centering
    \begin{tikzpicture}[scale=0.55]
    \fill (0,0) circle (3pt) (1,0) circle (3pt) (2,0) circle (3pt) (3,0) circle (3pt) (4,0) circle (3pt) (5,0) circle (3pt) (6,0) circle (3pt) (5,-1) circle (3pt);
    \draw (1,0)--(3.2,0);
    \draw (3.8,0)--(6,0);
    \draw (5,0)--(5,-1);
    \draw (5,-1)--(6,0);
    \draw plot [smooth] coordinates {(0,0) (0.5,0.4) (1,0)};
    \draw plot [smooth] coordinates {(0,0) (0.5,-0.6) (1,0)};
    \draw plot [smooth] coordinates {(0,0) (0.5,0.6) (1,0)};
    \draw (0.5,0.1) node{$\vdots$};
    \draw[dotted] (3.2,0)--(3.8,0);
    \filldraw[fill=white] (0,0) circle (3pt);
    \draw (0,-0.3) node{\tiny$1$} (1,-0.3) node{\tiny$2$} (2,-0.3) node{\tiny$3$} (3,-0.3) node{\tiny$4$} (4,-0.3) node{\tiny$n\!-\!3$} (5,+0.3) node{\tiny$n\!-\!2$} (6,-0.3) node{\tiny$n$} (5.6,-1) node{\tiny$n\!-\!1$};
    \end{tikzpicture}
\end{subfigure}

\begin{subfigure}[t]{.24\textwidth}
    \centering
    \begin{tikzpicture}[scale=0.55]
       \fill (0,0.6) circle (0pt);
   \fill (0,-1.3) circle (0pt);
    \fill (0,0) circle (3pt) (1,0) circle (3pt) (2,0) circle (3pt) (3,0) circle (3pt) (4,0) circle (3pt) (5,0) circle (3pt) (6,0) circle (3pt);
    \draw (1,0)--(3.2,0);
    \draw (3.8,0)--(6,0);
    \draw plot [smooth] coordinates {(0,0) (0.5,0.4) (1,0)};
    \draw plot [smooth] coordinates {(0,0) (0.5,-0.6) (1,0)};
    \draw plot [smooth] coordinates {(0,0) (0.5,0.6) (1,0)};
    \draw plot [smooth] coordinates {(0,0) (0.4,-0.8) (1,-1) (1.6,-0.8) (2,0)};
    \draw (0.5,0.1) node{$\vdots$};
    \draw[dotted] (3.2,0)--(3.8,0);
    \filldraw[fill=white] (1,0) circle (3pt);
    \draw (0,-0.3) node{\tiny$2$} (1,-0.3) node{\tiny$1$} (2,-0.3) node{\tiny$3$} (3,-0.3) node{\tiny$4$} (4,-0.3) node{\tiny$n\!-\!2$} (5,0.3) node{\tiny$n\!-\!1$} (6,-0.3) node{\tiny$n$};
    \end{tikzpicture}
\end{subfigure}
\begin{subfigure}[t]{.24\textwidth}
    \centering
    \begin{tikzpicture}[scale=0.55]
       \fill (0,0.6) circle (0pt);
    \fill (0,0) circle (3pt) (1,0) circle (3pt) (2,0) circle (3pt) (3,0) circle (3pt) (4,0) circle (3pt) (5,0) circle (3pt) (6,0) circle (3pt) (5,-1) circle (3pt);
    \draw (1,0)--(3.2,0);
    \draw (3.8,0)--(6,0);
    \draw (5,0)--(5,-1);
    \draw plot [smooth] coordinates {(0,0) (0.5,0.4) (1,0)};
    \draw plot [smooth] coordinates {(0,0) (0.5,-0.6) (1,0)};
    \draw plot [smooth] coordinates {(0,0) (0.5,0.6) (1,0)};
    \draw plot [smooth] coordinates {(0,0) (0.4,-0.8) (1,-1) (1.6,-0.8) (2,0)};
    \draw (0.5,0.1) node{$\vdots$};
    \draw[dotted] (3.2,0)--(3.8,0);
    \filldraw[fill=white] (1,0) circle (3pt);
    \draw (0,-0.3) node{\tiny$2$} (1,-0.3) node{\tiny$1$} (2,-0.3) node{\tiny$3$} (3,-0.3) node{\tiny$4$} (4,-0.3) node{\tiny$n\!-\!3$} (5,+0.3) node{\tiny$n\!-\!2$} (6,-0.3) node{\tiny$n$} (5.6,-1) node{\tiny$n\!-\!1$};
    \end{tikzpicture}
\end{subfigure}
\begin{subfigure}[t]{.24\textwidth}
    \centering
    \begin{tikzpicture}[scale=0.55]
       \fill (0,0.6) circle (0pt);
    \fill (0,0) circle (3pt) (1,0) circle (3pt) (2,0) circle (3pt) (3,0) circle (3pt) (4,0) circle (3pt) (5,0) circle (3pt) (6,0) circle (3pt) (5,-1) circle (3pt);
    \draw (1,0)--(3.2,0);
    \draw (3.8,0)--(6,0);
    \draw (5,0)--(5,-1);
    \draw (5,-1)--(6,0);
    \draw plot [smooth] coordinates {(0,0) (0.5,0.4) (1,0)};
    \draw plot [smooth] coordinates {(0,0) (0.5,-0.6) (1,0)};
    \draw plot [smooth] coordinates {(0,0) (0.5,0.6) (1,0)};
    \draw plot [smooth] coordinates {(0,0) (0.4,-0.8) (1,-1) (1.6,-0.8) (2,0)};
    \draw (0.5,0.1) node{$\vdots$};
    \draw[dotted] (3.2,0)--(3.8,0);
    \filldraw[fill=white] (1,0) circle (3pt);
    \draw (0,-0.3) node{\tiny$2$} (1,-0.3) node{\tiny$1$} (2,-0.3) node{\tiny$3$} (3,-0.3) node{\tiny$4$} (4,-0.3) node{\tiny$n\!-\!3$} (5,+0.3) node{\tiny$n\!-\!2$} (6,-0.3) node{\tiny$n$} (5.6,-1) node{\tiny$n\!-\!1$};
    \end{tikzpicture}
\end{subfigure}


\begin{subfigure}[t]{.24\textwidth}
    \centering
    \begin{tikzpicture}[scale=0.55]
           \fill (0,-1.3) circle (0pt);
      \fill (0,0.6) circle (0pt);
    \fill (0,0) circle (3pt) (1,0) circle (3pt) (2,0) circle (3pt) (3,0) circle (3pt) (4,0) circle (3pt) (5,0) circle (3pt) (6,0) circle (3pt);
    \draw (1,0)--(3.2,0);
    \draw (3.8,0)--(6,0);
    \draw plot [smooth] coordinates {(0,0) (0.5,0.4) (1,0)};
    \draw plot [smooth] coordinates {(0,0) (0.5,-0.6) (1,0)};
    \draw plot [smooth] coordinates {(0,0) (0.5,0.6) (1,0)};
    \draw plot [smooth] coordinates {(0,0) (0.4,-0.8) (1,-1) (1.6,-0.8) (2,0)};
    \draw (0.5,0.1) node{$\vdots$};
    \draw[dotted] (3.2,0)--(3.8,0);
    \filldraw[fill=white] (0,0) circle (3pt);
    \draw (0,-0.3) node{\tiny$1$} (1,-0.3) node{\tiny$2$} (2,-0.3) node{\tiny$3$} (3,-0.3) node{\tiny$4$} (4,-0.3) node{\tiny$n\!-\!2$} (5,0.3) node{\tiny$n\!-\!1$} (6,-0.3) node{\tiny$n$};
    \end{tikzpicture}
\end{subfigure}
\begin{subfigure}[t]{.24\textwidth}
    \centering
    \begin{tikzpicture}[scale=0.55]
    \fill (0,0) circle (3pt) (1,0) circle (3pt) (2,0) circle (3pt) (3,0) circle (3pt) (4,0) circle (3pt) (5,0) circle (3pt) (6,0) circle (3pt) (5,-1) circle (3pt);
    \draw (1,0)--(3.2,0);
    \draw (3.8,0)--(6,0);
    \draw (5,0)--(5,-1);
    \draw plot [smooth] coordinates {(0,0) (0.5,0.4) (1,0)};
    \draw plot [smooth] coordinates {(0,0) (0.5,-0.6) (1,0)};
    \draw plot [smooth] coordinates {(0,0) (0.5,0.6) (1,0)};
    \draw plot [smooth] coordinates {(0,0) (0.4,-0.8) (1,-1) (1.6,-0.8) (2,0)};
    \draw (0.5,0.1) node{$\vdots$};
    \draw[dotted] (3.2,0)--(3.8,0);
    \filldraw[fill=white] (0,0) circle (3pt);
    \draw (0,-0.3) node{\tiny$1$} (1,-0.3) node{\tiny$2$} (2,-0.3) node{\tiny$3$} (3,-0.3) node{\tiny$4$} (4,-0.3) node{\tiny$n\!-\!3$} (5,+0.3) node{\tiny$n\!-\!2$} (6,-0.3) node{\tiny$n$} (5.6,-1) node{\tiny$n\!-\!1$};
    \end{tikzpicture}
\end{subfigure}
\begin{subfigure}[t]{.24\textwidth}
    \centering
    \begin{tikzpicture}[scale=0.55]
    \fill (0,0) circle (3pt) (1,0) circle (3pt) (2,0) circle (3pt) (3,0) circle (3pt) (4,0) circle (3pt) (5,0) circle (3pt) (6,0) circle (3pt) (5,-1) circle (3pt);
    \draw (1,0)--(3.2,0);
    \draw (3.8,0)--(6,0);
    \draw (5,0)--(5,-1);
    \draw (5,-1)--(6,0);
    \draw plot [smooth] coordinates {(0,0) (0.5,0.4) (1,0)};
    \draw plot [smooth] coordinates {(0,0) (0.5,-0.6) (1,0)};
    \draw plot [smooth] coordinates {(0,0) (0.5,0.6) (1,0)};
    \draw plot [smooth] coordinates {(0,0) (0.4,-0.8) (1,-1) (1.6,-0.8) (2,0)};
    \draw (0.5,0.1) node{$\vdots$};
    \draw[dotted] (3.2,0)--(3.8,0);
    \filldraw[fill=white] (0,0) circle (3pt);
    \draw (0,-0.3) node{\tiny$1$} (1,-0.3) node{\tiny$2$} (2,-0.3) node{\tiny$3$} (3,-0.3) node{\tiny$4$} (4,-0.3) node{\tiny$n\!-\!3$} (5,+0.3) node{\tiny$n\!-\!2$} (6,-0.3) node{\tiny$n$} (5.6,-1) node{\tiny$n\!-\!1$};
    \end{tikzpicture}
\end{subfigure}
\captionsetup{width=1.0\linewidth}
\caption*{(i) Labeling of the vertices, where the hollow vertex does not belong to $A$, when $n$ is odd. }
 \vspace{0.3cm}

\begin{subfigure}[t]{.24\textwidth}
    \centering
    \begin{tikzpicture}[scale=0.55]
    \fill (0,-1) circle (0pt);
    \fill (0,0) circle (3pt) (1,0) circle (3pt) (2,0) circle (3pt) (3,0) circle (3pt) (4,0) circle (3pt) (5,0) circle (3pt) (6,0) circle (3pt);
    \draw (1,0)--(3.2,0);
    \draw (3.8,0)--(6,0);
    \draw plot [smooth] coordinates {(0,0) (0.5,0.4) (1,0)};
    \draw plot [smooth] coordinates {(0,0) (0.5,-0.6) (1,0)};
    \draw plot [smooth] coordinates {(0,0) (0.5,0.6) (1,0)};
    \draw (0.5,0.1) node{$\vdots$};
    \draw[dotted] (3.2,0)--(3.8,0);
        \filldraw[fill=white] (1,0) circle (3pt);
            \filldraw[fill=white] (0,0) circle (3pt);
    \draw (0,-0.3) node{\tiny$2$} (1,-0.3) node{\tiny$1$} (2,-0.3) node{\tiny$3$} (3,-0.3) node{\tiny$4$} (4,-0.3) node{\tiny\!$n\!-\!2$\!} (5,0.3) node{\tiny$n\!-\!1$} (6,-0.3) node{\tiny$n$};
    \end{tikzpicture}
\end{subfigure}
\begin{subfigure}[t]{.24\textwidth}
    \centering
    \begin{tikzpicture}[scale=0.55]
    \fill (0,-1) circle (0pt);
    \fill (0,0) circle (3pt) (1,0) circle (3pt) (2,0) circle (3pt) (3,0) circle (3pt) (4,0) circle (3pt) (5,0) circle (3pt) (6,0) circle (3pt);
    \draw (1,0)--(3.2,0);
    \draw (3.8,0)--(6,0);
    \draw plot [smooth] coordinates {(0,0) (0.5,0.4) (1,0)};
    \draw plot [smooth] coordinates {(0,0) (0.5,-0.6) (1,0)};
    \draw plot [smooth] coordinates {(0,0) (0.5,0.6) (1,0)};
    \draw (0.5,0.1) node{$\vdots$};
    \draw[dotted] (3.2,0)--(3.8,0);
    \draw (0,-0.3) node{\tiny$2$} (1,-0.3) node{\tiny$1$} (2,-0.3) node{\tiny$3$} (3,-0.3) node{\tiny$4$} (4,-0.3) node{\tiny$n\!-\!2$} (5,0.3) node{\tiny$n\!-\!1$} (6,-0.3) node{\tiny$n$};
    \end{tikzpicture}
\end{subfigure}
\begin{subfigure}[t]{.24\textwidth}
    \centering
    \begin{tikzpicture}[scale=0.55]
    \fill (0,0) circle (3pt) (1,0) circle (3pt) (2,0) circle (3pt) (3,0) circle (3pt) (4,0) circle (3pt) (5,0) circle (3pt) (6,0) circle (3pt);
    \draw (1,0)--(3.2,0);
    \draw (3.8,0)--(6,0);
    \draw plot [smooth] coordinates {(0,0) (0.5,0.4) (1,0)};
    \draw plot [smooth] coordinates {(0,0) (0.5,-0.6) (1,0)};
    \draw plot [smooth] coordinates {(0,0) (0.5,0.6) (1,0)};
    \draw plot [smooth] coordinates {(0,0) (0.4,-0.8) (1,-1) (1.6,-0.8) (2,0)};
    \draw (0.5,0.1) node{$\vdots$};
           \filldraw[fill=white] (1,0) circle (3pt);
            \filldraw[fill=white] (0,0) circle (3pt);
    \draw[dotted] (3.2,0)--(3.8,0);
    \draw (0,-0.3) node{\tiny$2$} (1,-0.3) node{\tiny$1$} (2,-0.3) node{\tiny$3$} (3,-0.3) node{\tiny$4$} (4,-0.3) node{\tiny$n\!-\!2$} (5,0.3) node{\tiny$n\!-\!1$} (6,-0.3) node{\tiny$n$};
    \end{tikzpicture}
\end{subfigure}
\begin{subfigure}[t]{.24\textwidth}
    \centering
    \begin{tikzpicture}[scale=0.55]
    \fill (0,0) circle (3pt) (1,0) circle (3pt) (2,0) circle (3pt) (3,0) circle (3pt) (4,0) circle (3pt) (5,0) circle (3pt) (6,0) circle (3pt);
    \draw (1,0)--(3.2,0);
    \draw (3.8,0)--(6,0);
    \draw plot [smooth] coordinates {(0,0) (0.5,0.4) (1,0)};
    \draw plot [smooth] coordinates {(0,0) (0.5,-0.6) (1,0)};
    \draw plot [smooth] coordinates {(0,0) (0.5,0.6) (1,0)};
    \draw plot [smooth] coordinates {(0,0) (0.4,-0.8) (1,-1) (1.6,-0.8) (2,0)};
    \draw (0.5,0.1) node{$\vdots$};
    \draw[dotted] (3.2,0)--(3.8,0);
    \draw (0,-0.3) node{\tiny$2$} (1,-0.3) node{\tiny$1$} (2,-0.3) node{\tiny$3$} (3,-0.3) node{\tiny$4$} (4,-0.3) node{\tiny$n\!-\!2$} (5,0.3) node{\tiny$n\!-\!1$} (6,-0.3) node{\tiny$n$};
    \end{tikzpicture}
\end{subfigure}
\captionsetup{width=1.0\linewidth}
\caption*{(ii) Labeling of the vertices,  where the hollow vertices do not belong to $A$, when $n$ is even.}
\caption{Labeling of the vertices}\label{fig:labeling of vertices-1x2x}
\end{figure}

\begin{figure}
    \begin{subfigure}{0.25\textwidth}
    \begin{center}
    \begin{tikzpicture}[scale=.6]
    \fill (1.5,0) circle(3pt);
    \fill (3,0) circle(3pt);
    \fill (4.5,0) circle (3pt);
    \fill (6,0) circle (3pt);
    \draw (1.5,0)--(6,0);
    \draw (0,0)..controls (0.75,1)..(1.5,0);
    \draw (0,0)..controls (0.75,-0.7)..(1.5,0);
    \draw (0,0)..controls (0.75,0.3)..(1.5,0);
    \filldraw[fill=white] (0,0) circle(3pt);
	\draw (0,-0.3) node{\tiny$1$};
	\draw (1.5,-0.3) node{\tiny$2$};
	\draw (3,-0.3) node{\tiny$3$};
	\draw (4.5,-0.3) node{\tiny$4$};
	\draw (6,-0.3) node{\tiny$5$};
	\draw (0.75,0.8) node{\tiny$a_1$};
	\draw (0.75,0.2) node{\tiny$a_2$};
		\draw (0.75,-0.5) node{\tiny$b_1$};
    \end{tikzpicture}
    \caption*{$G$}
    \end{center}
    \end{subfigure}
    \begin{subfigure}{.72\textwidth}
    \begin{center}
    \begin{tikzpicture}[scale=.47]
    \draw[thick] (11,5.4)..controls (8,10)..(-5,11.5);
    \draw[thick] (0,2.9)..controls (-1,5.5)..(-7.5,7.5);
    \draw[thick] (0,2.9)..controls (-.5,5.5)..(-4.6,7.5);
    \node [draw] (0) at (2.5,1) {\!\tiny $\emptyset$\!};
    \node [draw] (1) at (-4,2.5) {\!\!\tiny$1$\!\!};
    \node [draw] (23) at (0,2.5) {\!\!\tiny$23$\!\!};
    \node [draw] (34) at (5,2.5) {\!\!\tiny$34$\!\!};
    \node [draw] (45) at (8,2.5) {\!\!\tiny$45$\!\!};
    \node [draw] (12a) at (-8,5) {\!\!\tiny$12a_1$\!\!};
    \node [draw] (12b) at (-3,5) {\!\!\tiny$12a_2$\!\!};
    \node [draw] (134) at (6.5,5) {\!\!\tiny$134$\!\!};
    \node [draw] (145) at (9,5) {\!\!\tiny$145$\!\!};
    \node [draw] (2345) at (11,5) {\!\!\tiny$2345$\!\!};
    \node [draw] (123ab) at (-10,8) {\!\!\tiny$123a_1a_2$\!\!};
    \node [draw] (1234a) at (-7.5,8) {\!\!\tiny$1234a_1$\!\!};
    \node [draw] (1234b) at (-5,8) {\!\!\tiny$1234a_2$\!\!};
    \node [draw] (1245a) at (-2.5,8) {\!\!\tiny$1245a_1$\!\!};
    \node [draw] (1245b) at (0,8) {\!\!\tiny$1245a_2$\!\!};
    \node [draw] (12ac) at (2.5,8) {\!\!\tiny$12a_1b_1$\!\!};
    \node [draw] (12bc) at (5,8) {\!\!\tiny$12a_2b_1$\!\!};
    \node [draw] (123c) at (7.5,8) {\!\!\tiny$123b_1$\!\!};
    \node [draw] (12345c) at (11.5,12) {\!\!\tiny$12345b_1$\!\!};
    \node [draw] (12345ab) at (-6,12) {\!\!\tiny$12345a_1a_2$\!\!};
    \node [draw] (123abc) at (-3,12) {\!\!\tiny$123a_1a_2b_1$\!\!};
    \node [draw] (1234ac) at (0,12) {\!\!\tiny$1234a_1b_1$\!\!};
    \node [draw] (1234bc) at (3,12) {\!\!\tiny$1234a_2b_1$\!\!};
    \node [draw] (1245ac) at (5.8,12) {\!\!\tiny$1245a_1b_1$\!\!};
    \node [draw] (1245bc) at (9,12) {\!\!\tiny$1245a_2b_1$\!\!};
    \node [draw] (G) at (3, 14.5) {\!\tiny$G$\!};
    \path (0) edge (1) edge (23) edge (34) edge (45);
    \path (G) edge (12345ab) edge (123abc) edge (1234ac) edge (1234bc) edge (1245ac) edge (1245bc);
    \path (G) edge (12345c);
    \path (1) edge (12a) edge (12b) edge (134) edge (145);\path[thick] (1) edge (123c);
    \path (23) edge (2345);\path[thick] (23) edge (123c);
    \path[thick] (23) edge (123ab);
    \path (34) edge (134) edge (2345);
    \path (45) edge (145) edge (2345);
    \path (12a) edge (123ab) edge (1234a) edge (1245a) edge (12ac);
    \path (12b) edge (123ab) edge (1234b) edge (1245b) edge (12bc);
    \path (134) edge (1234a) edge (1234b);\path[thick] (134) edge (12345c);
    \path (145) edge (1245a) edge (1245b);\path[thick] (145) edge (12345c);
    \path[thick] (2345) edge (12345c);
    \path (123c) edge (12345c) edge (123abc) edge (1234ac) edge (1234bc);
    \path (123ab) edge (12345ab) edge (123abc);
    \path (1234a) edge (12345ab) edge (1234ac);
    \path (1234b) edge (12345ab) edge (1234bc);
    \path (1245a) edge (12345ab) edge (1245ac);
    \path (1245b) edge (12345ab) edge (1245bc);
    \path (12ac) edge (123abc) edge (1234ac) edge (1245ac);
    \path (12bc) edge (123abc) edge (1234bc) edge (1245bc);
    \end{tikzpicture}
    \end{center}
    \subcaption*{(i) A poset $\pP{G,A}$ when $A=2345a_1a_2$ ($|V|$ is odd and 1 is not in $A$.)}\label{fig:odd_ordering}
    \end{subfigure}

    \begin{subfigure}{.25\textwidth}
    \begin{center}
    \begin{tikzpicture}[scale=0.6]
      \fill (1.5,0) circle(3pt);
    \fill (3,0) circle(3pt);
    \fill (4.5,0) circle (3pt);
    \draw (1.5,0)--(4.5,0);
    \draw (0,0)..controls (0.75,1)..(1.5,0);
    \draw (0,0)..controls (0.75,-0.6)..(1.5,0);
    \draw (0,0)..controls (0.75,-1)..(1.5,0);
    \draw (0,0)..controls (0.75,0.6)..(1.5,0);
    \filldraw[fill=white] (0,0) circle(3pt);
    \filldraw[fill=white] (1.5,0) circle(3pt);
	\draw (0,-0.3) node{\tiny$2$};
	\draw (1.5,-0.3) node{\tiny$1$};
	\draw (3,-0.3) node{\tiny$3$};
	\draw (4.5,-0.3) node{\tiny$4$};
	\draw (0.75,0.8) node{\tiny$a_1$};
	\draw (0.75,0.3) node{\tiny$a_2$};
	\draw (0.75,-0.3) node{\tiny$b_1$};
	\draw (0.75,-0.8) node{\tiny$b_2$};
\end{tikzpicture}
\caption*{$G$}
\end{center}
\end{subfigure}
\begin{subfigure}{.7\textwidth}
\begin{center}
\begin{tikzpicture}[scale=0.6]
    \node [draw] (0) at (2.5,0.7) {\!\tiny $\emptyset$\!};
    \node [draw] (1) at (0,2) {\!\!\tiny$1$\!\!};
    \node [draw] (2) at (2.5,2) {\!\!\tiny$2$\!\!};
    \node [draw] (34) at (5,2) {\!\!\tiny$34$\!\!};
    \node [draw] (12ab) at (-3,4) {\!\!\tiny$12a_1a_2$\!\!};
    \node [draw] (123a) at (-1,4) {\!\!\tiny$123a_1$\!\!};
    \node [draw] (123b) at (1,4) {\!\!\tiny$123a_2$\!\!};
    \node [draw] (12c) at (3,4) {\!\!\tiny$12b_1$\!\!};
    \node [draw] (12d) at (5,4) {\!\!\tiny$12b_2$\!\!};
    \node [draw] (134) at (7,4) {\!\!\tiny$134$\!\!};
    \node [draw] (234) at (9,4) {\!\!\tiny$234$\!\!};
    \node [draw] (1234ab) at (-4.4,6.5) {\!\!\tiny$1234a_1a_2$\!\!};
    \node [draw] (12abc) at (-2.2,6.5) {\!\!\tiny$12a_1a_2b_1$\!\!};
    \node [draw] (12abd) at (0,6.5) {\!\!\tiny$12a_1a_2b_2$\!\!};
    \node [draw] (123ac) at (2,6.5) {\!\!\tiny$123a_1b_1$\!\!};
    \node [draw] (123ad) at (4,6.5) {\!\!\tiny$123a_1b_2$\!\!};
    \node [draw] (123bc) at (6,6.5) {\!\!\tiny$123a_2b_1$\!\!};
    \node [draw] (123bd) at (8,6.5) {\!\!\tiny$123a_2b_2$\!\!};
    \node [draw] (12cd) at (10,6.5) {\!\!\tiny$12b_1b_2$\!\!};
    \node [draw] (1234c) at (12,6.5) {\!\!\tiny$1234b_1$\!\!};
    \node [draw] (1234d) at (14,6.5) {\!\!\tiny$1234b_2$\!\!};
    \node [draw] (1234abc) at (-3,9) {\!\!\tiny$1234a_1a_2b_1$\!\!};
    \node [draw] (1234abd) at (0,9) {\!\!\tiny$1234a_1a_2b_2$\!\!};
    \node [draw] (123acd) at (3,9) {\!\!\tiny$123a_1b_1b_2$\!\!};
    \node [draw] (123bcd) at (6,9) {\!\!\tiny$123a_2b_1b_2$\!\!};
    \node [draw] (12abcd) at (9,9) {\!\!\tiny$12a_1a_2b_1b_2$\!\!};
    \node [draw] (1234cd) at (12,9) {\!\!\tiny$1234b_1b_2$\!\!};
    \node [draw] (G) at (2.5, 10.4) {\!\tiny$G$\!};
    \path (0) edge (1) edge (2) edge (34);
    \path (G) edge (1234abc) edge (1234abd) edge (123acd) edge (123bcd) edge (12abcd) edge (1234cd);
    \path (1) edge (12ab) edge (123a) edge (123b) edge (12c) edge (12d) edge (134);
    \path (2) edge (12ab) edge (123a) edge (123b) edge (12c) edge (12d) edge (234);
    \path (34) edge (134) edge (234);
    \path (12ab) edge (1234ab) edge (12abc) edge (12abd);
    \path (123a) edge (1234ab) edge (123ac) edge (123ad);
    \path (123b) edge (1234ab) edge (123bc) edge (123bd);
    \path (12c) edge (12abc) edge (123ac) edge (123bc) edge (12cd) edge (1234c);
    \path (12d) edge (12abd) edge (123ad) edge (123bd) edge (12cd) edge (1234d);
    \path (134) edge (1234ab) edge (1234c) edge (1234d);
    \path (234) edge (1234ab) edge (1234c) edge (1234d);
    \path (1234abc) edge (1234ab) edge (12abc) edge (123ac) edge (123bc) edge (1234c);
    \path (1234abd) edge (1234ab) edge (12abd) edge (123ad) edge (123bd) edge (1234d);
    \path (123acd) edge (123ac) edge (123ad) edge (12cd);
    \path (123bcd) edge (123bc) edge (123bd) edge (12cd);
    \path (12abcd) edge (12abc) edge (12abd) edge (12cd);
    \path (1234cd) edge (12cd) edge (1234c) edge (1234d);
\end{tikzpicture}
\end{center}
\subcaption*{(ii) A poset $\pP{G,A}$ when $A=34a_1a_2$ ($|V|$ is even, and $1$ and $2$ are not in $A$.)}\label{fig:even_poset_1x2x}
\end{subfigure}

    \begin{subfigure}{.2\textwidth}
    \begin{center}
    \begin{tikzpicture}[scale=0.65]
          \fill (0,0) circle(3pt);
      \fill (1.5,0) circle(3pt);
    \fill (3,0) circle(3pt);
    \fill (4.5,0) circle (3pt);
    \draw (1.5,0)--(4.5,0);
    \draw (0,0)..controls (0.75,1)..(1.5,0);
    \draw (0,0)..controls (0.75,-0.6)..(1.5,0);
    \draw (0,0)..controls (0.75,-1)..(1.5,0);
    \draw (0,0)..controls (0.75,0.6)..(1.5,0);
	\draw (0,-0.3) node{\tiny$2$};
	\draw (1.5,-0.3) node{\tiny$1$};
	\draw (3,-0.3) node{\tiny$3$};
	\draw (4.5,-0.3) node{\tiny$4$};
	\draw (0.75,0.8) node{\tiny$a_1$};
	\draw (0.75,0.3) node{\tiny$a_2$};
	\draw (0.75,-0.3) node{\footnotesize$b_1$};
	\draw (0.75,-0.8) node{\tiny$b_2$};
\end{tikzpicture}
\caption*{$G$}
\end{center}
\end{subfigure}
\begin{subfigure}{.7\textwidth}
\begin{center}
\begin{tikzpicture}[scale=0.7]
    \node [draw] (0) at (4,0.5) {\!\tiny $\emptyset$\!};
    \node [draw] (13) at (1,1.5) {\!\tiny$13$\!};
    \node [draw] (34) at (9,1.5) {\!\tiny$34$\!};
    \node [draw] (123a) at (-3,3) {\!\tiny$123a_1$\!};
    \node [draw] (123b) at (-1,3) {\!\tiny$123a_2$\!};
    \node [draw] (12ab) at (1,3) {\!\tiny$12a_1a_2$\!};
    \node [draw] (12c) at (7,3) {\!\tiny$12b_1$\!};
    \node [draw] (12d) at (8.4,3) {\!\tiny$12b_2$\!};
    \node [draw] (1234ab) at (-4.5,5.5) {\!\tiny$1234a_1a_2$\!};
    \node [draw] (12abc) at (-2.5,5.5) {\!\tiny$12a_1a_2b_1$\!};
    \node [draw] (12abd) at (-0.5,5.5) {\!\tiny$12a_1a_2b_2$\!};
    \node [draw] (123ac) at (1.5,5.5) {\!\tiny$123a_1b_1$\!};
    \node [draw] (123ad) at (3.2,5.5) {\!\tiny$123a_1b_2$\!};
    \node [draw] (123bc) at (4.9,5.5) {\!\tiny$123a_2b_1$\!};
    \node [draw] (123bd) at (6.6,5.5) {\!\tiny$123a_2b_2$\!};
    \node [draw] (1234c) at (10,5.5) {\!\tiny$1234b_1$\!};
    \node [draw] (1234d) at (11.7,5.5) {\!\tiny$1234b_2$\!};
    \node [draw] (12cd) at (8.3,5.5) {\!\tiny$12b_1b_2$\!};
    \node [draw] (1234abc) at (1,7.5) {\!\tiny$1234a_1a_2b_1$\!};
    \node [draw] (1234abd) at (3.5,7.5) {\!\tiny$1234a_1a_2b_2$\!};
    \node [draw] (12abcd) at (6,7.5) {\!\tiny$12a_1a_2b_1b_2$\!};
    \node [draw] (123acd) at (8.2,7.5) {\!\tiny$123a_1b_1b_2$\!};
    \node [draw] (123bcd) at (10.3,7.5) {\!\tiny$123a_2b_1b_2$\!};
    \node [draw] (1234cd) at (12.5,7.5) {\!\tiny$1234b_1b_2$\!};
    \node [draw] (G) at (6, 8.5) {\!\tiny$G$\!};
    \path (0) edge (13) edge (34);
    \path[thick] (0) edge (12c);
    \path[thick] (0) edge (12d);
    \path[thick] (0) edge (12ab);
    \path (G) edge (1234abc) edge (1234abd) edge (12abcd) edge (123acd) edge (123bcd) edge (1234cd);
    \path (1234abc) edge (1234ab) edge (12abc) edge (123ac) edge (123bc) edge (1234c);
    \path (1234abd) edge (1234ab) edge (12abd) edge (123ad) edge (123bd) edge (1234d);
    \path (12abcd) edge (12abc) edge (12abd) edge (12cd);
    \path (123acd) edge (123ac) edge (123ad) edge (12cd);
    \path (123bcd) edge (123bc) edge (123bd) edge (12cd);
    \path (1234cd) edge (1234c) edge (1234d) edge (12cd);
    \path (12ab) edge (1234ab) edge (12abc) edge (12abd);
    \path (12c) edge (1234c) edge (12cd);
    \path (12d) edge (1234d) edge (12cd);
    \path (12c) edge (12abc) edge (123ac) edge (123bc);
    \path (12d) edge (12abd) edge (123ad) edge (123bd);
    \path (13) edge (123a) edge (123b);
    \path[thick] (13) edge (1234c) edge (1234d);
    \path[thick] (34)  edge (1234c) edge (1234ab) edge (1234d);
    \path (123a) edge (1234ab) edge (123ac) edge (123ad);
    \path (123b) edge (1234ab) edge (123bc) edge (123bd);
\end{tikzpicture}
\end{center}
\subcaption*{(iii) A poset $\pP{G,A}$ when $A=1234a_1a_2$ ($|V|$ is even)}\label{fig:even_poset_1o2o}
\end{subfigure}
\caption{Examples of posets $\pP{G,A}$}\label{fig:examples of even posets}
\end{figure}

Note that given a cover $I\lessdot J$ in $\mathcal{P}_{G,A}^{\mathrm{even}}$, if $(J\setminus I)\cap B\neq\emptyset$, then either  $(J\setminus I)\cap (B\cap A)=\emptyset$ or $(J\setminus I)\cap (B\setminus A)=\emptyset$. Suppose not, that is, $(J\setminus I)\cap (B\cap A)\neq\emptyset$ and $(J\setminus I)\cap (B\setminus A)\neq\emptyset$. Then $K:=J\setminus \{(J\setminus I)\cap (B\setminus A)\}$ satisfies that $I<K<J$ and $|(K\setminus I)\cap A|\equiv|(J\setminus I)\cap A|$, a contradiction to $I\lessdot J$.
Hence we can define the type of a cover $I\lessdot J$ in $\mathcal{P}_{G,A}^{\mathrm{even}}$ according to the size of $J\setminus I$ and the intersection with $B\setminus A$.
A cover $I\lessdot J$  has type~(E$i$)  if $|J\setminus I|=i$ and $J\setminus I$ has no element  of  $B\setminus A$; and $I\lessdot J$  has type~(E${i}^{\prime}$)  if $|J\setminus I|=i$ and $J\setminus I$ contains some elements  of $B\setminus A$. Hence (E$i{}^\prime$) can occur {only when} $B\setminus A\neq \emptyset$. 

We observe that for any cover $I\lessdot J$ in $\mathcal{P}_{G,A}^{\mathrm{even}}$,
each of $|I\cap A|$, $|J\cap A|$, and $|(J\setminus I)\cap A|$ is even, and $J\setminus I$ satisfies the following condition, which we will call ($\dagger$).
\begin{quote}
($\dagger$)\label{dagger condition} The elements in $J\setminus I$ belong to a same component of $J$.
\end{quote}
From {these} observations, we obtain the following lemma. See Appendix for its proof.

\begin{lemma}\label{lem:types}
Let $I\lessdot J$ be a cover in $\mathcal{P}_{G,A}^{\mathrm{even}}$. Then $J\setminus I$ is one of forms in Table~\ref{table:all_type}.
\end{lemma}
\renewcommand{\arraystretch}{1.15}
\begin{table}[h]
\centering
   \footnotesize{ \begin{tabular}{c|cc||c|c|c|c|c|c|c|c}
    \toprule
   \multirow{2}{*}{$I\lessdot J$ }&  \multirow{2}{*}{Type}&\multirow{2}{*}{}  &\multirow{2}{*}{(E1)}
    &\multirow{2}{*}{(E2)}
     &\multirow{2}{*}{(E3)}
      &\multirow{2}{*}{(E4)}
       &\multirow{2}{*}{(E1${}^\prime$)} &\multirow{2}{*}{(E2${}^\prime$)}
       &  \multicolumn{1}{r}{(E3${}^\prime$)} \\ \cline{10-11}
       &&&&&&&&& \multicolumn{1}{c}{(E3${}^\prime$-1)}&  \multicolumn{1}{|c}{(E3${}^\prime$-2)} \\
         \hline \hline
        \multirow{3}{*}{$J\setminus I$}& \multicolumn{1}{c}{$|V|$:odd} & \multicolumn{1}{|c||}{$A\cap\{1,2\}=\{2\}$}
        & 1& $cc'$ & $1ac$ &-&$b$ &$1b$& $2vb$ &- \\ \cline{2-11}
        \multicolumn{1}{c|}{}    &  \multirow{2}{*}{$|V|$:even} & \multicolumn{1}{|c||}{$A\cap\{1,2\}=\emptyset$}
        & 1 or 2 & $cc'$ & $1ac$ or $2ac$ &-& $b$ &$1b$ or $2b$ & - &- \\ \cline{3-11}
         \multicolumn{1}{c|}{}    &   & \multicolumn{1}{|c||}{$A\subset \{1,2\}$}
        & - & $cc'$ & - & $12aa'$& $b$ &- & $1vb$ or $2vb$ & $12b$ \\
      \bottomrule
    \end{tabular}}
    \captionsetup{width=1.0\linewidth}\\[1ex]
    \caption{Types of $I\lessdot J$ in $\pP{G,A}$, where $a,\,a'\in B\cap A$, $b\in (B\setminus A)$, $c,\,c'\in A$, $v=\min(V\setminus (I\cup\{1,2\}))$. For the case of (E3), $c\in B\cap A$ or $c=\min(V\setminus (I\cup\{1,2\}))$.}\label{table:all_type} \vspace{-0.5cm}
\end{table}

Especially when $I\lessdot J$ is of (E3${}^\prime$), as in Table~\ref{table:all_type},
we divide {the type (E3${}^\prime$)} into two {subtypes} according to the size of
$(J\setminus I)\cap \{1,2\}$:
\begin{itemize}
\item  $I\lessdot J$ {has type} (E3${}^{\prime}$-1) if  $I\lessdot J$ {has type} (E3${}^{\prime}$) and $|(J\setminus I)\cap \{1,2\}|=1$;
\item  $I\lessdot J$ {has type} (E3${}^{\prime}$-2) if   $I\lessdot J$ {has type} (E3${}^{\prime}$) and $|(J\setminus I)\cap \{1,2\}|=2$.
\end{itemize}
We can also show that when $I\lessdot  J$ is of~(E3${}^{\prime}$-1),  $J\setminus I$ contains  the vertex $\min(V\setminus (I\cup\{1,2\}))$.

\begin{proposition}\label{prop:length:1x}
The lengths of maximal chains of $\mathcal{P}_{G,A}^{\mathrm{even}}$  are
\[\begin{cases}
 \frac{|A|}{2}+|B\setminus A|+1 \ \text{ or }\ \frac{|A|}{2}+ |B\setminus A|  & \text{if }|V|\text{ is odd,}\\
  \frac{|A|}{2}+1   & \text{if }|V|\text{ is odd, $2$ and $3$ are not adjacent in $G$ and $B\subset A$,}\\
 \frac{|A|}{2}+|B\setminus A|+1 & \text{if }|V|\text{ is even and }A\cap V\neq V,  {and}\\
\frac{|A|}{2}+|B\setminus A| \ \text{ or }\ \frac{|A|}{2}+|B\setminus A|-1& \text{if }|V|\text{ is even and }A\cap V = V.
\end{cases} \]
\end{proposition}
\begin{proof} Recall that $|V|=n$, $|B\cap A|=2m$, and $|B\setminus A|=\ell$. Note that $2m+n\ge 4$.
 Let $\sigma\colon I_0\lessdot I_1\lessdot\cdots\lessdot I_p$  be a maximal chain  of $\mathcal{P}_{G,A}^{\mathrm{even}}$.
Note that $\{I_i\setminus I_{i-1} \mid i=1,\ldots,p \}$ is a partition of $V\cup B$.

Let $k$ be the smallest index such that $I_{k}\setminus I_{k-1}$ contains an element in $B$, that is,
$I_k$ is the first element of $\sigma$ containing a multiple edge. Then $\{1,2\}\subset I_k$ and $\{1,2\}\not\subset I_{k-1}$. Together with Table~\ref{table:all_type}, we see that for each cover $I_{i-1}\lessdot I_i$ of $\sigma$, except the cover $I_{k-1}\lessdot I_k$, it holds that $|I_{i}\setminus I_{i-1}|=1$ or 2.
For each $j\in\{1,2\}$, let $t_j$ be the number of covers $I_{i-1}\lessdot I_i$ of $\sigma$, except the cover $I_{k-1}\lessdot I_k$, such that $|I_{i}\setminus I_{i-1}|=j$.
Then the number of covers of $\sigma$, which is equal to $\ell(\sigma)$, is $1+t_1+t_2$.
Since $\{I_i\setminus I_{i-1} \mid i=1,\ldots,p \}$ is a partition of $V\cup B$, we have
\[ t_1 +2t_2 +|I_k\setminus I_{k-1}|= |I_p\setminus I_{0}|=n+2m+\ell \]
or, $t_2=\frac{(n+2m+\ell)-t_1-|I_k\setminus I_{k-1}|}{2}$. Therefore
\begin{eqnarray}
&&\ell(\sigma)=1+t_1+t_2 = \frac{(n+2m+\ell)+2+t_1-|I_k\setminus I_{k-1}|}{2}. \label{eq:length}
\end{eqnarray}

Note that $\sigma$ has exactly $ |B\setminus (A\cup I_k)|$ covers of (E1${}^{\prime}$) and
$\sigma$ has at most one cover of (E1). In addition, $\sigma$ has one cover of (E1) if and only if $I_{k-1}$ contains a vertex in $\{1,2\}\setminus A$.
Since $t_1$ is the sum of the number of covers of (E1) and the number of covers of (E1${}^{\prime}$), and so
\begin{eqnarray}
&&t_1=\begin{cases}
1+|B\setminus (A\cup I_k)| &\text{if }I_{k-1} \text{ contains a vertex in }\{1,2\}\setminus A,\\
|B\setminus (A\cup I_k)| &\text{otherwise}.
\end{cases}\label{eq:length:t1}
\end{eqnarray}

Suppose that $|V|$ is odd.
Then $|A|=2m+n-1$. By Table~\ref{table:all_type} again, $I_{k-1}\lessdot I_{k}$ has one of types (E2), (E3), (E2${}^{\prime}$), and (E3${}^{\prime}$-1).
By \eqref{eq:length} and \eqref{eq:length:t1},
\[\ell(\sigma)=\begin{cases}
\frac{(n+2m+\ell)+2+(1+\ell)-2}{2} = \ell+1+\frac{n+2m-1}{2}=\ell+1+\frac{|A|}{2} &\text{if }I_{k-1}\lessdot I_{k} \text{ is of  (E2)}\\
\frac{(n+2m+\ell)+2+\ell-3}{2}=\ell+\frac{n+2m-1}{2} =\ell +\frac{|A|}{2}  &\text{if }I_{k-1}\lessdot I_{k} \text{ is of  (E3)}\\
\frac{(n+2m+\ell)+2+(\ell-1)-2}{2}=\ell+\frac{n+2m-1}{2}=\ell +\frac{|A|}{2}    &\text{if }I_{k-1}\lessdot I_{k} \text{ is of  (E2${}^{\prime}$)}\\
\frac{(n+2m+\ell)+2+(1+(\ell-1))-3}{2}=\ell+\frac{n+2m-1}{2}=\ell +\frac{|A|}{2}   & \text{if }I_{k-1}\lessdot I_{k} \text{ is of  (E3${}^{\prime}$-1)}.
\end{cases}\]
Hence, every maximal chain has the length either  ${\frac{|A|}{2}+\ell+1}$ or ${\frac{|A|}{2}+\ell}$, and hence the poset $\mathcal{P}_{G,A}^{\mathrm{even}}$ is nonpure.
Note that if $2$ and $3$ are not adjacent in $G$ then there is no cover of (E3), and if  $B\subset A$ then $B\setminus A=\emptyset$  and so there is no cover of   (E2${}^\prime$) or (E3${}^\prime$). Hence if $2$ and $3$ are not adjacent in $G$ and $B\subset A$, then $\mathcal{P}_{G,A}^{\mathrm{even}}$ is  a pure poset  of length~${\frac{|A|}{2}+1}$.
We summarize  in the following table:
\begin{center}\small{
 \begin{tabular}{c||l | c}
\toprule
\!\!\!\!$I_{k-1}\lessdot I_{k}$ \!\!\!\!&  The types of the covers in $\sigma$  & Length of $\sigma$ \\ \hline
  {(E2)}
  & one (E1), $\ell$ (E1${}^{\prime}$)s, $\frac{2m+n-1}{2}$  (E2)s&   {$\frac{2m+n+1}{2}+\ell$} \\    \hline
  {(E3)} &   {one  (E3), $\ell$  (E1${}^{\prime}$)s, $\frac{2m+n-3}{2}$  (E2)s}
&   \multirow{3}{*}{$\frac{2m+n-1}{2}+\ell$}  \\   \cline{1-2}
 {(E2${}^\prime$)}
&{one   (E2${}^\prime$), $(\ell\!-\!1)$  (E1${}^{\prime}$)s, ${\frac{2m+n-1}{2}}$  (E2)s}
&     \\    \cline{1-2}
{(E3${}^\prime$-1)}
&{one  (E1),  one  (E3${}^\prime$-1), {$(\ell\!-\!1)$}  (E1${}^\prime$)s,  ${\frac{2m+n-3}{2}}$  (E2)s} & \\
 \bottomrule
    \end{tabular}}\end{center}

Suppose that $|V|$ is even.
When $A\cap \{1,2\}=\emptyset$,
it holds that $|A|=2m+n-2$ and
$I_{k-1}\lessdot I_{k}$ is of (E3) or (E2${}^{\prime}$).
By \eqref{eq:length} and \eqref{eq:length:t1},
\[\ell(\sigma)=\begin{cases}
\frac{(n+2m+\ell)+2+(1+\ell)-3}{2} = \ell+1+\frac{n+2m-2}{2}=\ell+1+\frac{|A|}{2} &\text{if }I_{k-1}\lessdot I_{k} \text{ is of  (E3)}\\
\frac{(n+2m+\ell)+2+(1+(\ell-1))-2}{2}=\ell+1+\frac{n+2m-2}{2} =\ell +1+\frac{|A|}{2}  &\text{if }I_{k-1}\lessdot I_{k} \text{ is of   (E2${}^{\prime}$)}.
\end{cases}\]
Hence every maximal chain has the length $\frac{|A|}{2}+\ell+1$.

When $A$ contains $\{1,2\}$,
it holds that $|A|=2m+n$ and
$I_{k-1}\lessdot I_{k}$ is one of (E2), (E4), and (E3${}^{\prime}$).
By \eqref{eq:length} and \eqref{eq:length:t1},
\[\ell(\sigma)=\begin{cases}
\frac{(n+2m+\ell)+2+\ell-2}{2} = \ell+\frac{n+2m}{2}=\ell+\frac{|A|}{2} &\text{if }I_{k-1}\lessdot I_{k} \text{ is of  (E2)}\\
\frac{(n+2m+\ell)+2+\ell-4}{2} = \ell-1+\frac{n+2m}{2}=\ell-1+\frac{|A|}{2} &\text{if }I_{k-1}\lessdot I_{k} \text{ is of  (E4)}\\
\frac{(n+2m+\ell)+2+(\ell-1)-3}{2}=\ell-1+\frac{n+2m}{2} =\ell-1+\frac{|A|}{2}  &\text{if }I_{k-1}\lessdot I_{k} \text{ is of   (E3${}^{\prime}$)}.
\end{cases}\]
Hence every maximal chain has the length $\ell+\frac{|A|}{2} $ or $\ell-1+\frac{|A|}{2}$.
We summarize  in the following table:
\begin{center}\small{
\begin{tabular}{c|c||l | c}
\toprule
  \!\!{{$A\cap\{1,2\}$}} \!\!& {$\!I_{k-1}\!\lessdot\!I_k\!$} & The types of the covers in $\sigma$ &  \!\! Length of $\sigma$ \\ \hline
  \multirow{2}{*}{$ \emptyset$}&
 {(E3)} &  {\!one  (E1), one  (E3), $\ell $   (E1${}^{\prime}$)s, ${\frac{2m+n-4}{2}}$   (E2)s}
 &   \multirow{2}{*}{$\frac{2m+n}{2}+\ell$} \\
  \cline{2-3}
&  {(E2${}^{\prime}$)}   &  {\!one  (E1), one   (E2${}^{\prime}$), $(\ell\!-\!1)$ (E1${}^{\prime}$)s, $\frac{2m+n-2}{2}$ (E2)s}  &  \\
\hline
 \multirow{3}{*}{$\{1,2\}$}   &  {(E2)}
  & {\!$\frac{2m+n}{2}$   (E2)s, $\ell$   (E1${}^{\prime}$)s}&  {$\frac{2m+n}{2}+\ell$} \\    \cline{2-4}
  &{(E4)} &   {\!one   (E4), $\ell$    (E1${}^{\prime}$),  $\frac{2m+n-4}{2}$   (E2)s}
&   \multirow{2}{*}{$\frac{2m+n-2}{2}+\ell$} \\      \cline{2-3}
  & {(E3${}^{\prime}$)}
&{one  (E3${}^{\prime}$), $(\ell\!-\!1)$  (E1${}^{\prime}$)s, $\frac{2m+n-2}{2}$  (E2)s}
&
         \\ \bottomrule
    \end{tabular}} \end{center}
\end{proof}

We shall show that $\mathcal{P}_{G,A}^{\mathrm{even}}$ admits a recursive atom ordering. We first define the lexicographic order $\newprec{I}$ on $V\cup B$ for each $I\in \mathcal{P}_{G,A}^{\mathrm{even}}$ and then define the atom ordering $\atomprec{I}$ for $[I,G]$.

\begin{definition}\label{def:atomprec}
    Let $I\in \mathcal{P}_{G,A}^{\mathrm{even}}$.
    We define the lexicographic order~$\newprec{I}$ on $V\cup B$ as follows:
    \begin{itemize}
        \item   If $B\cap I=\emptyset$, then $$\newprec{I}: \  1,2,3,\ldots,{n},a_1,\ldots,a_{2m},b_1,\ldots,b_\ell.$$
        \item    If $B\cap I\neq\emptyset$  and $(B\setminus A)\cap I= \emptyset$,  then let $k:=\max\{i\mid a_i\in B\cap A\cap I\}$ and
          $$\newprec{I}:  \  1,2,a_1,\ldots,a_{k},3,\ldots,n,a_{k+1},\ldots,a_{2m}, b_1,\ldots,b_{\ell}.$$
        \item  If $(B\setminus A)\cap I \neq \emptyset$,  then let $k:=\max\{i\mid b_i\in (B\setminus A)\cap I\}$ and
           $$\newprec{I}: \ 1,2,a_1,\ldots,a_{2m},b_1,\ldots,b_{k},3,\ldots,n,b_{k+1},\ldots,b_{\ell}.$$
    \end{itemize}
    Then for two atoms $J$ and $J'$ of $[I,G]$, we define $J\atomprec{I} J'$ if one of the following holds:
    \begin{enumerate}
        \item[(O1)] $|(J\setminus I)\cap\{1,2\}|=1$ and $|(J'\setminus I)\cap \{1,2\}|=2$; or
        \item[(O2)] $J\setminus I \newprec{I} J'\setminus I$, where we compare lexicographic order induced by $\newprec{I}$, that is,
 we  order the elements in each of $J\setminus I$ and $J'\setminus I$ by $\newprec{I}$ and then compare pairwisely each position by $\newprec{I}$.
    \end{enumerate}Note that  (O1) is considered only when $\mathcal{P}_{G,A}^{\mathrm{even}}$ admits a cover of (E4) or (E3${}^\prime$-2), that is, $|V|$ is even and $A$ contains $\{1,2\}$.
\end{definition}

 Here is an example.
Let $G$ be the graph $\tilde{P}_{6,5}$ in Figure~\ref{fig:list of possible graphs}. Suppose that $A=V \cup \{a_1,a_2,a_3,a_4\}$.
Then the atoms of $\mathcal{P}_{G,A}^{\mathrm{even}}$ are ordered as follows:
\[\atomprec{\emptyset}: \ 13, \  12a_1a_2, \  12a_1a_3 , \ 12a_1a_4, \  12a_2a_3, \  12a_2a_4
, \  12a_3a_4, \  12b_1, \  34, \  45, \  56.\]
For $I=12a_1a_3$,  $\newprec{I}: 1, 2,a_1,{\boldsymbol a_2},a_3,{\boldsymbol 3},{\boldsymbol 4},{\boldsymbol 5},{\boldsymbol 6},{\boldsymbol a_4},{\boldsymbol b_1}$, and the atoms of $[I,G]$ are ordered as follows:
\[\atomprec{I}: \ 12{\boldsymbol 3}a_1{\boldsymbol a_2}a_3, \  12a_1{\boldsymbol a_2}a_3{\boldsymbol a_4}, \  12\boldsymbol{34}a_1a_3, \  12{\boldsymbol 3}a_1a_3{\boldsymbol a_4}, \  12\boldsymbol{45}a_1a_3, \   12\boldsymbol{56}a_1a_3, \  12a_1a_3{\boldsymbol b_1} \]
{because $a_23\newprec{I}a_2a_4\newprec{I}34\newprec{I}3a_4\newprec{I}45\newprec{I}56\newprec{I}b_1$, where the bold letters indicate the elements not in $I$.}

The following is the main theorem of this section, whose proof is given in Subsection~\ref{subsec:proof}.
 \begin{theorem}\label{thm:cl shellable}
Let $G$ be a connected graph in Figure~\ref{fig:list of possible graphs} and $A$ be an admissible collection of $G$. Then
the poset $\mathcal{P}_{G,A}^{\mathrm{even}}$ admits a recursive atom ordering, and hence $\mathcal{P}_{G,A}^{\mathrm{even}}$ is CL-shellable.
 \end{theorem}
\begin{remark}
We insist that the ordering $\atomprec{I}$ is essential.
Suppose that we consider the lexicographic order $\prec^\ast$ given by $1,2,a_1,a_2,\ldots,a_{2m},3,4,\ldots,n,b_1,\ldots, b_\ell$, and define $\atomprec{\ast}$ by an ordering obtained by replacing $\atomprec{I}$  in (O2) of Definition~\ref{def:atomprec} with the fixed ordering $\prec^*$. For the posets in Figure~\ref{fig:examples of even posets}, $\atomprec{\ast}$ gives a recursive atom ordering.
However, it fails to be a recursive atom ordering in general.
For example, let $G$ be a graph {in Figure~\ref{fig:list of possible graphs} with $|V|= 4$ and $|B|= 6$, and let $A=V\cup B$. Then $A\in\mathcal{A}(G)$.}
Let $I=12a_1a_3$, and {consider} the atoms $J_1= 12 a_1 a_3{\boldsymbol a_5}{\boldsymbol a_6}$ and $J_2= 12{\boldsymbol 3}a_1 a_3{\boldsymbol a_5}$ of $[I,G]$, where the bold letters indicate the elements not in $I$.
Then the atoms of $[\emptyset,G]$  {preceding} $I$ in $\atomprec{\ast}$ are $13$ and $ 12a_1a_2$.
However,  $J_1  \atomprec{\ast} J_2$, $J_2$ contains the atom $13$, and $J_1$ does not contain any atom of $[\emptyset,G]$  {preceding} $I$, and so (2) of Definition~\ref{def:recursive atom order} fails.
\end{remark}

\subsection{Proof of Theorem~\ref{thm:cl shellable}}\label{subsec:proof}

{For a subset $X\subset V\cup B$, $\min^{I}(X)$ and $\max^{I}(X)$ denote the minimum and the maximum of $X$ with respect to $\newprec{I}$, respectively.}
The proof of the following {lemma} will be given later.
\begin{lemma} \label{lem:falling}
Let {$I_j$ be an atom of $[I,G]$, not the first in $\atomprec{I}$.}
There is an element {$x(I\lessdot I_j)\in V\cup B$} such that an atom $J$ of $[I_j,G]$ belongs to $[I_k,G]$ for some $I_k\atomprec{I} I_j$ if and only if $\min^{I_j}(J\setminus I_j) \newprec{I_j} {x(I\lessdot I_j)}$.
\end{lemma}

We first prove Theorem~\ref{thm:cl shellable} by using Lemma~\ref{lem:falling}.

\begin{proof}[Proof of Theorem~\ref{thm:cl shellable}]
    We will show that the ordering $\atomprec{I}$ ($I\in \mathcal{P}_{G,A}^{\mathrm{even}}$) is a recursive atom ordering. Lemma~\ref{lem:falling} implies that $\atomprec{I}$ satisfies (1) of Definition~\ref{def:recursive atom order}.
    Let us check (2) of Definition~\ref{def:recursive atom order}. Let $I_i$ and $I_j$ be atoms of $[I,G]$ such that $I_i\atomprec{I} I_j$. Suppose that there is  an element $K$ of $[I,G]$ such that $I_i, I_j< K$. We need to find an atom  $K_\ast$  of $[I_j,G]$ and an atom $I_\ast$  of $[I,G]$ such that  $$K_\ast\le K, \quad K_\ast\in [I_\ast,G],\quad\text{ and }\quad I_\ast\atomprec{I} I_j.$$ Let $K_0=I_i\cup I_j$ for simplicity. Note that $K_0\subset K$ and one can check that  from Lemma~\ref{lem:types}, there is no element $L\in  \mathcal{P}_{G,A}^{\mathrm{even}}$ such that $I_j \subsetneq  L \subsetneq K_0$.

    {We first consider when} $K_0$ is not a semi-induced subgraph of $G$.  Note that
  a subset of $V\cup B$ is not a semi-induced subgraph if and only if
 it contains $\{1,2\}$ and has no element in  {$B$}.
    Then $I_i\setminus I$ and $I_j\setminus I$ contain exactly one of the endpoints of $B$, not the same. More precisely, letting  $v=\min(V\setminus (I\cup\{1,2\}))$, the following hold:
    \[    \begin{array}{lll}
        I_i\setminus I=1 &\text{and }\quad I_j\setminus I=2v &\text{if }A\cap\{1,2\}=\{2\},\\
        I_i\setminus I=1 &\text{and }\quad I_j\setminus I=2 &\text{if }A\cap\{1,2\}=\emptyset,\text{ or}\\
        I_i\setminus I=1v &\text{and }\quad I_j\setminus I=2v &\text{if }A\cap\{1,2\}=\{1,2\}.\\
    \end{array}\]
    Note that $I_j\setminus I=2v$ occurs only when $2$ and $3$ are adjacent in $G$.
    {In addition, $1\not\in A$ if and only if $|K_0\cap A|\equiv |(K_0\setminus I)\cap A|\equiv 0\pmod{2}$.}
    {Since $K$ have both $1$ and $2$, it should have} a multiple edge~$e$. Let $H$ be the component of $K$ containing $e$ and so $e\in H\setminus K_0$.

    If $|A\cap\{1,e\}|\equiv 0\pmod{2}$, then $K_\ast=K_0\cup e=I_j\cup 1e$ and $I_\ast=I_i$.

    If $|A\cap\{1,e\}|=1$, then {$|K_0\cap A|\equiv|H\cap K_0\cap A|\equiv|(H\setminus K_0)\cap A|$,} where the first equivalence is from the definition of $H$ and the second equivalence is from $|H\cap A|\equiv 0\pmod{2}$.
    Hence $1\not\in A$ if and only if $|(H\setminus K_0)\cap A|$ is even.
    Let $X=(H\setminus K_0)\cap A$ for simplicity.
    If $|X|$ is even, then $1\not\in A$ and  {$e\in A$} and therefore, $|X\setminus\{e\}|\geq 1$. If $|X|$ is odd, then  $1\in A$ and  {$e\not\in A$}  and therefore, $|X\setminus \{e\}|=|X|\geq 1$. Hence in any case we can take an element $c\in X\setminus\{e\}$ so that $K_\ast=K_0\cup ce=I_j\cup 1ce$ and $I_\ast=I_i$. More precisely, either $c$ is the vertex $\min(V\setminus (I_j\cup\{1,2\}))$  or belongs to $(B\cap A)\setminus{e}$.

    Now we consider when $K_0$ is a semi-induced subgraph of $G$.
    Note that $(I_i\setminus I) \cap (I_j\setminus I)\cap A$   is nonempty and has at most three elements.

    If $|(I_i\setminus I) \cap (I_j\setminus I)\cap A|$ is even, then \[|K_0\cap A|=|I_i\cap A|+|I_j\cap A|-|(I_i\cap I_j)\cap A|=|I_i\cap A|+|I_j\cap A|-|I\cap A|-|(I_i\setminus I) \cap (I_j\setminus I)\cap A|,\]
and therefore $|K_0\cap A|$ is even.
By ($\dagger$), $K_0=I_j\cup (I_i\setminus I_j)$ is an atom of $[I_j,G]$ and so $K_{\ast}=K_0$ and $I_{\ast}=I_i$.

    If $|(I_i\setminus I) \cap (I_j\setminus I)\cap A|$ is odd, then $|(I_i\setminus I) \cap (I_j\setminus I)\cap A|$ is one or three.
Since $(I_i\setminus I) \cap (I_j\setminus I)\cap A \neq \emptyset$,
the elements in $K_0\setminus I$ lie on a same component of $K_0$ by ($\dagger$).
Thus $K_0$ has exactly one component $H_0$ such that $|H_0\cap A|$ is odd.
Let $H$ be the component of $K$ containing $H_0$.
Note that  $|(H\setminus H_0)\cap A| \ge 1$ since $|H\cap A|$ is even.

If $H_0$ contains a multiple edge, then there exists an element $c\in (H\setminus H_0)\cap A$ such that $K_\ast=K_0  {\cup c}$ and $I_\ast=I_i$; more precisely, if $(H\setminus H_0)\cap B\cap A\neq\emptyset$, then $c\in B\cap A$; otherwise, $c$ is the vertex $\min(V\setminus H_0)$.

Suppose that $H_0$ has no multiple edge.
Then both $I_i\setminus I$ and $I_j\setminus I$ consist of two vertices  in $A$, and $|(I_i\setminus I)\cap (I_j\setminus I)\cap A|=1$.
Since $H_0$ is a semi-induced subgraph {of $G$ and} $|H_0\cap A|$ is odd, $(H\setminus H_0)\cap V\neq \emptyset$.
If $(H\setminus H_0)\cap V$ has a vertex in $\{3,...,n\}$, then by the structure of $G$, it is easy to see that there is a vertex $v$ in $(H\setminus H_0)\cap \{3,...,n\}$ such that $K_{\ast}=K_0\cup v$ and $I_{\ast}=I_i$. Hence we only need to consider the case in which ${(H\setminus H_0)\cap V}\subset\{1,2\}$.
If ${(H\setminus H_0)\cap V}=\{1,2\}$, then
\[\begin{cases}
    K_\ast=I_j\cup 1 \text{ and } I_\ast=I\cup 1&\text{ if }1\not\in A\\
    K_\ast=K_0\cup 1 \text{ and } I_\ast=I_i &\text{ if }1\in A.
    \end{cases}\]
It remains to consider the case where $(H\setminus H_0)\cap V=\{1\}$ or $\{2\}$.
Let $(H\setminus H_0)\cap V=\{w_1\}$, and let $w_2$ be the other vertex in $\{1,2\}$.

If $H_0$ does not contain $w_2$, then $H=H_0\cup w_1$ and $w_1\in A$ (and therefore, $w_1$ must be a neighbor of~$3$ since $H$ is an element of $\mathcal{P}_{G,A}^{\mathrm{even}}$) and hence $K_\ast=K_0\cup w_1$ and $I_\ast=I_i$.

Suppose that $H_0$ contains $w_2$. Then $H$ contains a multiple edge, that is, $(H\setminus H_0) \cap B=H\cap B\neq \emptyset$. Hence
\[   1\le|H\cap B|=|H\cap B\cap A|+|H\cap (B\setminus A)|.\]
Together with \[|(H\setminus H_0)\cap A| = |(H\setminus H_0)\cap V\cap A| + | (H\setminus H_0) \cap B\cap A|=|\{w_1\}\cap A| + | (H\setminus H_0) \cap B\cap A|
,\]
since $|(H\setminus H_0)\cap A|$ is odd, we see
 \[\begin{cases}
       |H\cap B\cap A|\ge 1  &\text{ if } w_1\not\in A,\\
      |H\cap B\cap A|\ge 2  &\text{ if } w_1\in A \text{ and } H\cap (B\setminus A)=\emptyset\\
      |H\cap (B\setminus A)|\ge 1 &\text{ otherwise.}
\end{cases}\]
For the case where $w_1\not\in A$, it easily follows that $K_\ast=K_0\cup w_1a$ and  $I_\ast=I_i$ for some $a\in H\cap B\cap A$. Now assume that $w_1\in A$, and we prove the remaining part by dividing two subcases whether $w_2\in I$ or not.
If $w_2\in I$, then
\[\begin{cases}
         K_\ast=I_j\cup w_1aa' \text{ and } I_\ast=I\cup w_1a  \text{ for some } a,a'\in H\cap B\cap A&\text{ if } {H}\cap (B\setminus A)=\emptyset,\\
        K_\ast=K_0\cup w_1b \text{ and } I_\ast=I_i \text{ for some }b\in H\cap(B\setminus A)&\text{ if } {H}\cap(B\setminus A)\neq\emptyset.
\end{cases}\]
If $w_2\not\in I$, then $w_2\in (I_i\setminus I) \cup (I_j\setminus I)$.
Since $I_i\atomprec{I} I_j$, $w_2\in I_i$. Moreover, by the structure of $G$,
$I_i\setminus I=w_2v$ and $I_j\setminus I=vv'$ for $v=\min(V\setminus (I\cup \{1,2\} ))$ and $v'=\min(V\setminus (I\cup\{1,2,v\}))$. Hence
\[\begin{cases}
        K_\ast=I_j\cup 12aa' \text{ and }I_\ast=I \cup 12aa' \text{ for some } a,a'\in H\cap B\cap A& \text{ if } H\cap (B\setminus A)=\emptyset,\\
        K_\ast=I_j\cup 12b \text{ and }I_\ast=I \cup 12b \text{ for some }b\in H\cap (B\setminus A)&  \text{ if }H\cap (B\setminus A)\neq\emptyset.
\end{cases}\]
This completes the proof.
\end{proof}

For an element $I$ of $\mathcal{P}_{G,A}^{\mathrm{even}}$, a multiple edge {$e$} is {called} a \emph{big} (respectively, \emph{small}) edge of $I$ if $n \newprec{I} e$ (respectively, $e\newpreceq{I} n$). Note that if $e$ is a small edge of $I$, then  $e\newprec{I} 3$.
Now we prove Lemma~\ref{lem:falling}.

\begin{proof}[Proof of Lemma~\ref{lem:falling}]
Let $I_j$ be an atom of $[I,G]$, not the first in $\atomprec{I}$.
We will show that an atom $J$ of $[I_j,G]$ belongs to $[I_\ast,G]$ for some atom $I_\ast$ of $[I,G]$ with $I_\ast\atomprec{I} I_j$ if and only if the following hold:
\begin{itemize}
        \item[\textbf{(1)}] $\min^{I_j}(J\setminus I_j)\newprec{I_j} 2$, if  $I_j\setminus I\subset V$ and $(I_j\setminus I)\cap\{1,2\}\neq\emptyset$;
        \item[\textbf{(2)}] $\min^{I_j}(J\setminus I_j)\newprec {I_j} \min^{I_j}\{v,b_1\}$, if $I_j\setminus I=va$ for $v\in V$  {and a small edge $a$ of $I$ in $B\cap A$};
        \item[\textbf{(3)}] $\min^{I_j}(J\setminus I_j)\newpreceq{I_j} n$, if $I_j\setminus I$ consists of only big edges of $I$;
        \item[\textbf{(4)}] $\min^{I_j}(J\setminus I_j)\newpreceq{I_j} {\min(V\setminus I_j)}$, if $V\setminus I_j\neq\emptyset$, $I_j\setminus I$ has {an element $c\newpreceq{I}n$ and} a big edge of~$I$, and  $|(I_j\setminus I) \cap \{1,2\}|\equiv 0\pmod{2}$,
        \item[\textbf{(5)}] $\min^{I_j}(J\setminus I_j)\newprec{I_j}\max^{I_j}(I_j\setminus I)$, otherwise.
    \end{itemize}
Whenever we show the `if' part of  each case,  we finish the proof when we find a proper atom $I_\ast$ of $[I,G]$, that is, $I_\ast$ is an atom of $[I,G]$ such that $J\in [I_\ast,G]$ and  $I_\ast\atomprec{I} I_j$.

\medskip

\noindent\textbf{(1)} Suppose that $I_j\setminus I\subset V$ and $(I_j\setminus I)\cap\{1,2\}\neq\emptyset$.
Note that
$\newprec{I}: \  1,2,\ldots,{n},a_1,\ldots,a_{2m},b_1,\ldots,b_\ell.$
{Since we assumed $I_j$ is not the first atom of $[I,G]$,
$(I_j\setminus I)\cap\{1,2\}=\{2\}$.}
Suppose $1\in J\setminus I_j$. Then
\[  J\setminus I_j  =  \left\{\begin{array}{ll}
        1ac\text{ or }1b, &\text{ if }1\not\in A,\\
        1a\text{ or }1{v}b  & \text{ if }1\in A,
    \end{array}\right.  \] where $a\in B\cap A$, $c\in A$, $b\in B\setminus A$, and ${v}=\min(V\setminus (I_j\cup\{1\}))$.
Then $I_\ast$ is either
$I\cup 1$ or $I\cup 1v$ in each case, which proves the `if' part.

If $1\not\in J\setminus I_j$, then $J\setminus I_j$ cannot have a multiple edge and so it consists of vertices greater than $\max^{I}(I_j\setminus I)$ by~($\dagger$). Therefore,  {$I_j$ is the first atom of $[I,G]$, which} proves the `only if' part.

\medskip

\noindent\textbf{(2)} Suppose that  $I_j\setminus I=va$ for $v\in V$ and a small edge $a$ of $I$ in $B\cap A$. The existence of a small edge of $I$ implies $\{1,2\}\subset I$  and $v={\min(V\setminus  I)}$. Hence $\newprec{I}=\newprec{I_j}$ and {they are} either
\begin{eqnarray*}
&&\newprec{I}=\newprec{I_j}: \ 1,2,a_1,a_2,\ldots,a_{k},3, \ldots,n, a_{k+1},\ldots,a_{2m},b_1,\ldots,b_\ell, \quad \text{or}\\
&&\newprec{I}=\newprec{I_j}: \ 1,2,a_1,a_2,\ldots,a_{2m},b_1,\ldots,b_{k},3,\ldots,n, b_{k+1},\ldots,b_\ell.
\end{eqnarray*}
If $J\setminus I_j$ has an element $a'\newprec{I_j} \min^{I_j}\{v,b_1\}$, then $a'\in B\cap A$ and $I_\ast=I \cup aa'$, which proves the `if' part.

Suppose that $\min^{I_j}(J\setminus I_j) \succeq_{\mathrm{lex}}^{I_j} \min^{I_j}\{v,b_1\}$.
If $J\setminus I_j=b$ for some $b\in B\setminus A$, then $[I,J]$ has only two atoms $I_j$ and $I\cup b$, and $I_j$ is the first.
If $J\setminus I_j=cc'\subset A$ for some $c,c'\succeq_{\mathrm{lex}}^{I_j} v$, then $I_j\setminus I$ consists of the first two smallest elements of $J\setminus I$ and so $I_j$ is the first atom of $[I,J]$. This proves the `only if' part.

\medskip

\noindent\textbf{(3)} Suppose that  $I_j\setminus I$ consists of only big edges of $I$.
{Then $I\cap B\neq\emptyset$ and either $I_j\setminus I=aa'$ or  $I_j\setminus I=b$,} where $a,a'\in B\cap A$ and $b\in B\setminus A$.

\noindent\underline{(Case 1) $I_j\setminus I=aa'$.} The existence of a big edge of $I$ in $B\cap A$ implies that $I\cap (B\setminus A)=\emptyset$, and the lexicographic orders $\newprec{I}$ and $\newprec{I_j}$ are as follows:
\begin{eqnarray*}
\newprec{I}:&\!\!\!\!\!\!& 1,2,a_1,\ldots,a_{k},3,\ldots,n, a_{k+1},\ldots,a,\ldots,a'(=a_t),\ldots, a_{2m},b_{1},\ldots,b_\ell\\
\newprec{I_j}:&\!\!\!\!\!\!& 1,2,a_1,\ldots,a_{k},\ldots,a,\ldots,a'(=a_t),3,\ldots,n, a_{t+1},\ldots, a_{2m},b_{1},\ldots,b_\ell.
\end{eqnarray*}
Set $x:=\min^{I_j}(J\setminus I_j)$. Suppose that $x\newpreceq {I_j} n$. Then $x \newprec{I} a'$ and $I_j\lessdot J$ is of (E2). Hence we can set $J\setminus I_j=xx' \subset A$,  where $x\newprec{I_j}x'$.
{If $x\newpreceq{I_j} {\min(V\setminus I)}$, then  $I_\ast=I\cup ax$.
If $ {\min(V\setminus I)}\newprec{I_j} x$, then both $x$ and $x'$ are vertices} and $I_\ast=I\cup xx'$. This proves the `if' part.

Suppose that $n \newprec{I_j}  x$.
Then either $J\setminus I_j=xx'\subset B\cap A$ or $J\setminus I_j=x\in B\setminus A$. Since $a\newprec{I} a' \newprec{I} x$, $I_j$ is the first atom of $[I,J]$ in $\atomprec{I}$. This proves the `only if' part.

\noindent\underline{(Case 2)  $I_j\setminus I=b$.} In this case, the lexicographic order $\newprec{I_j}$ is given as follows:
$$\newprec{I_j}: 1,2,a_1,a_2,\ldots,a_{2m},b_1,\ldots,b(=b_k),3,4,\ldots,n,b_{k+1},\ldots,b_{\ell}.$$
Suppose that $J\setminus I_j$ contains an element $x$ with ${x}\newpreceq{I_j} n$.  Note that ${x}\newprec{I} b$.
If $\min^{I_j}(J\setminus I_j)$ is a vertex, then $J\setminus I_j$ consists of two vertices, and $I_\ast= I\cup (J\setminus I_j)$.
Now let $\min^{I_j}(J\setminus I_j)$ be a multiple edge $e$. 
If {$e\in B\setminus A$,} then $I_\ast=I_j\cup e$.
If {$e\in B\cap A$,} then $J\setminus I_j=ec$ for some $c\in A$, which implies that
$I_\ast=I\cup ec$. This proves the `if' part.

If  $n \newprec{I_j} \min^{I_j}(J\setminus I_j)$,  then $J\setminus I_j=\{b'\}$ for some $b'\in B\setminus A$ with $b\newprec{I_j} b'$, and hence $[I,J]$ has only two atoms $I_j$ and  $I\cup b'$, where $I_j$  is the first atom in $\atomprec{I}$. This proves the `only if' part.

\medskip

In order to show \textbf{(4)} and \textbf{(5)}, we need to show the following claim.
\begin{claim}\label{lem:falling-(4)}
Suppose that $I\cap (B\setminus A)$ is empty, $I_j\setminus I$ has both an element $c\newpreceq{I} n$ and a big edge of $I$,  and  $|(I_j\setminus I) \cap \{1,2\}|{\equiv 0\pmod{2}}$.
Then an atom $J$ of $[I_j,G]$ belongs to $[I_\ast,G]$ for some atom $I_\ast$ of $[I,G]$ with $I_\ast\atomprec{I} I_j$ if and only if {one of the following holds:}
\begin{itemize}
\item[(i)] $\min^{I_j}(J\setminus I_j)\newpreceq{I_j}  {\min(V\setminus I_j)}$ if $V\setminus I_j\neq \emptyset$
\item[(ii)]  $\min^{I_j}(J\setminus I_j)\newprec{I_j}\max^{I_j}(I_j\setminus I)$ if $V\setminus I_j= \emptyset$.
\end{itemize}
\end{claim}

\begin{proof}[Proof of Claim~\ref{lem:falling-(4)}]
From the hypotheses, {we need to consider the following four cases \textcircled{\footnotesize 1}$\sim$\textcircled{\footnotesize 4} in the table below,} where $a,a'\in B\cap A$ with {$a'\newprec{I} a$}, $b\in B\setminus A$, {and} $v=\min(V\setminus I)$:
\renewcommand{\arraystretch}{1.1}
\begin{center}{\small{
\begin{tabular}{l||c}
\toprule
\quad $I_j\setminus I$ \quad& The lexicographic orders \\ \hline \hline
\multirow{2}{*}{\textcircled{\footnotesize 1} $12a'a$} & \multicolumn{1}{l}{$\newprec{I}: 1,2,\ldots,{n},a_1,\ldots,a',\ldots,a(=a_k),\ldots,a_{2m},b_1,\ldots,b_\ell$} \\
        \cline{2-2}
    &\multicolumn{1}{l}{$\newprec{I_j}: \  1,2,a_1,\ldots,a',\ldots,a(=a_k),3,\ldots,{n},a_{k+1},\ldots,a_{2m},b_1,\ldots,b_\ell$}\\
        \hline
    \multirow{2}{*}{\textcircled{\footnotesize 2} $12b$} &    \multicolumn{1}{l}{$\newprec{I}: \  1,2,\ldots,{n},a_1,\ldots,a_{2m},b_1,\ldots,b(=b_k),\ldots,b_\ell$}\\
        \cline{2-2}
        &   \multicolumn{1}{l}{$\newprec{I_j}: \  1,2,a_1,\ldots,a_{2m},b_1,\ldots,b(=b_k),3,\ldots,{n},b_{k+1},\ldots,b_\ell$}\\
        \hline
         \multirow{2}{*}{\textcircled{\footnotesize 3}  $va$} &      \multicolumn{1}{l}{$\newprec{I}: \  1,2,a_1,\ldots,a_{j},3,\ldots,v,\ldots,{n},a_{j+1},\ldots,a(=a_{k}),\ldots,a_{2m},b_1,\ldots,b_\ell$}\\
        \cline{2-2}
        &    \multicolumn{1}{l}{$\newprec{I_j}: \  1,2,a_1,\ldots, a' (=a_{k}),3,\ldots,v,\ldots,{n},a_{k+1},\ldots,a_{2m},b_1,\ldots,b_\ell$}\\
        \hline
         \multirow{2}{*}{\textcircled{\footnotesize 4}  $a'a$}   & \multicolumn{1}{l}{$\newprec{I}: \  1,2,a_1,\ldots,a',\ldots,a_{j},3,\ldots,{n},a_{j+1},\ldots,a(=a_{k}),\ldots,a_{2m},b_1,\ldots,b_\ell$}\\
        \cline{2-2}
          & \multicolumn{1}{l}{$\newprec{I_j}: \  1,2,a_1,\ldots,a',\ldots,a(=a_{k}),3,\ldots,{n},a_{k+1},\ldots,a_{2m},b_1,\ldots,b_\ell$}\\
        \bottomrule
    \end{tabular}}}
\end{center}

Note that {cases \textcircled{\footnotesize 1} and \textcircled{\footnotesize 2}} can occur only when $\{1,2\}\subset A$. Let $v_\ast:={\min(V\setminus I_j)}$, provided $V\setminus I_j\neq\emptyset$.
In cases \textcircled{\footnotesize 1}$\sim$\textcircled{\footnotesize 4}, if $v_\ast\in J\setminus I_j$, then $I_\ast$'s are $I\cup 1v_\ast$, $I\cup 1v_\ast$, $I\cup vv_\ast$, and {$I\cup a'v_\ast$}, respectively. {Note that {when} $v_{\ast}\not\in J\setminus I_j$, {it holds that } $\min^{I_j}(J\setminus I_j)\newpreceq{I_j} v_\ast$ if and only if $\min^{I_j}(J\setminus I_j)\newprec{I_j}\max^{I_j}(I_j\setminus I)$. Now we assume that $v_{\ast}\not\in J\setminus I_j$ and $\min^{I_j}(J\setminus I_j)\newprec{I_j}\max^{I_j}(I_j\setminus I)$.} Set ${x}:=\min^{I_j}(J\setminus I_j)$.  Since $\max^{I_j}(I_j\setminus I)$ is a small edge of $I_j$, $x$ is also a multiple edge and hence $J\setminus I_j$ consists of multiple edges.  In case \textcircled{\footnotesize 2}, $I_\ast=I\cup 12 \cup(J\setminus I_j)$. For the other cases, $x\newprec{I}a$ and $x,a\in B\cap A$. Hence $I_\ast$ is obtained from $I_j$ by replacing $a$ with $x$. This proves the `if' part.

To prove the `only if' part, first suppose that $V\setminus I_j\neq\emptyset$ and $\min^{I_j}(J\setminus I_j)\succ_{\mathrm{lex}}^{I_j}v_\ast$.
Then $\min^{I_j}(J\setminus I_j)$ is either a vertex greater than $v_\ast$ or a big edge of $I_j$. Hence $J\setminus I_j$ consists of either two vertices greater than $v_\ast$ or only big edges of $I_j$. Note that a big edge of $I_j$ is also a big edge of $I$. Hence  $I_j$ is the first atom of $[I,J]$ in $\atomprec{I}$. If $V\setminus I_j=\emptyset$ and $\min^{I_j}(J\setminus I_j)\succ_{\mathrm{lex}}^{I_j}\max^{I_j}(I_j\setminus I)$, then $\min^{I_j}(J\setminus I_j)$ is a big edge of $I_j$ and hence $I_j\setminus I$ consists of only big edges of $I_j$, and so $I_j$ is the first atom of $[I,J]$ in~$\atomprec{I}$.
\end{proof}

By Claim~\ref{lem:falling-(4)}, \textbf{(4)} follows and \textbf{(5)} partially follows.
We  exclude the cases of \textbf{(1)}$\sim$\textbf{(4)} and the case shown  by Claim~\ref{lem:falling-(4)}.
We divide the remaining part into two cases according to the existence of a big edge of $I$ in $I_j\setminus I$.

\smallskip

\noindent\underline{(Case 1) $I_j\setminus I$ has no big edge of $I$.} By excluding \textbf{(1)} and \textbf{(2)}, we get one of the following:
\begin{itemize}
\item[\textcircled{\footnotesize 1}] $I_j\setminus I=b$ where $b\in B\setminus A$  and $b$ is a small edge of $I$;
\item[\textcircled{\footnotesize 2}] $I_j\setminus I=aa'$ where both $a,a'\in B\cap A$ are small edges of $I$; or
\item[\textcircled{\footnotesize 3}] $I_j\setminus I=vv'$ where $v,v'\in V\setminus\{1,2\}$.
    \end{itemize}
In each case, the `only if' part {easily} follows, that is, if
$\max^{I_j}(I_j\setminus I) \newprec{I_j} \min^{I_j}(J\setminus I_j)$, then $I_j\setminus I$ has the first $|I_j\setminus I|$ smallest elements of $J\setminus I$ (in $\newprec{I}$), and so $I_j$ is the first atom of $[I,J]$ in $\atomprec{I}$.  Let us prove the `if' part of each case.  We note that $\newprec{I}=\newprec{I_j}$.

\noindent\textcircled{\footnotesize 1} From  the existence of a small edge in $B\setminus A$, it follows that $I\cap(B\setminus A)\neq\emptyset$ and
$$\newprec{I}=\newprec{I_j}: 1,2,a_1,a_2,\ldots,a_{2m},b_1,\ldots,b,\ldots,b_k,3,\ldots,n,b_{k+1},\ldots,b_{\ell}.$$
If $\min^{I_j}(J\setminus I_j)\newprec{I_j} b$, then $I_\ast=I\cup (J\setminus I_j)$.

\noindent\textcircled{\footnotesize 2} Let $a \newprec{I_j} a'$, and hence $a'=\max^{I_j}(J\setminus I_j)$.
If $\min^{I_j}(J\setminus I_j)\newprec{I_j} a'$, then $J\setminus I$ contains a multiple edge $a''$ with $a''\newprec{I} a'$, and so
$I_\ast= I\cup aa''$.

\noindent\textcircled{\footnotesize 3} Let $v \newprec{I_j} v'$, and hence $v'=\max^{I_j}(I_j\setminus I)$.
Suppose that $\min^{I_j}(J\setminus I_j)\newprec{I_j} v'$.
If $J\setminus I_j$ has an element $b\in B\setminus A$, then
\begin{equation*}
    \begin{cases}
      J\setminus I_j=  b &\text{ if }\{1,2\}\subset I,\\
      J\setminus I_j=  12b&\text{ if }\{1,2\}\cap I=\emptyset,\text{ or}\\
      J\setminus I_j=  wb\text{ or }wv''b &\text{ if }|\{1,2\}\cap I|=1,
    \end{cases}
\end{equation*}  {where $w\in\{1,2\}\setminus I$ and $v''={\min(V\setminus (I_j\cup\{1,2\}))}$.} Then $I_\ast$'s are $I\cup b$, $I\cup 12b$, $I\cup wb$, and $I\cup wv_\ast b$ in the order, where $v_\ast=\min  \{v, v''\}$.

If $J\setminus I_j$ contains a multiple edge in $B\cap A$, then
$J\setminus I_j$ is either $aa'$, $12aa'$, $v''a$, or $wac$, where $a,a'\in B\cap A$, $w\in \{1,2\}$, $c\in A$ and $v''=\min(V\setminus I_j )$. If $J\setminus I_j$ is either $aa'$ or $12aa'$, then $I_\ast=I\cup(J\setminus I_j)$. If $J\setminus I_j=v''a$ and $v''\not\in\{1,2\}$, then $I_\ast=I\cup v_\ast a$ where $v_\ast=\min\{v,v''\}$.
If $J\setminus I_j=v''a$ and $v''\in\{1,2\}$, then $I_\ast=I\cup (J\setminus I_j)$.
If $J\setminus I_j=wac$, then $I_\ast=I\cup wac_\ast$, where $c_\ast=\min^{I_j}\{v,c\}$.

If $J\setminus I_j$ consists of only vertices, then {$J\setminus I$ consists of only vertices and}
$I_\ast=I\cup xy$, where $x$ and $y$ are the first two smallest elements of $J\setminus I$.  This completes the proof of the `if' part.

\smallskip

\noindent\underline{(Case 2) $I_j\setminus I$ has a big edge of $I$.}
By excluding \textbf{(3)} and \textbf{(4)}, we get the following five cases \textcircled{\footnotesize 1}$\sim$\textcircled{\footnotesize 5} in the table below, where $w\in \{1,2\}$, $a,a'\in B\cap A$ {with $a' \newprec{I_j} a$}, $b\in B\setminus A$, and $v=\min(V\setminus (I\cup\{1,2\}))$:
\begin{center}
{\small{\begin{tabular}{c|l||l | c}
\toprule
  & \quad $I_j\setminus I$ \quad& \qquad \qquad \qquad \qquad \qquad The lexicographic order $\newprec{I_j}$  & $\max^{I_j}(I_j\setminus I)$ \\ \hline
 \multirow{3}{*}{$w\not\in A$}   &  {\textcircled{\footnotesize 1} $wa'a$}
  & {$\newprec{I_j}: \  1,2,a_1,\ldots,a',\ldots,a(=a_k),3,\ldots,{n},a_{k+1},\ldots,a_{2m},b_1,\ldots,b_\ell$}&   {$a$} \\    \cline{2-4}
  &{\textcircled{\footnotesize 2} $wva$} &   {$\newprec{I_j}: \  1,2,a_1,\ldots,a(=a_k),3,\ldots,v,\ldots,{n},a_{k+1},\ldots,a_{2m},b_1,\ldots,b_\ell$}
&   {$v$}  \\      \cline{2-4}
  & {\textcircled{\footnotesize 3} $wb$}
&{$\newprec{I_j}: \  1,2,a_1,\ldots,a_{2m},b_1,\ldots,b(=b_k),3,\ldots,{n},b_{k+1},\ldots,b_\ell$}
&  {$b$} \\     \hline
\multirow{2}{*}{$w\in A$}&
 {\textcircled{\footnotesize 4} $wa$} & {$\newprec{I_j}: \  1,2,a_1,\ldots,a(=a_k),3,\ldots,{n},a_{k+1},\ldots,a_{2m},b_1,\ldots,b_\ell$}
 &  {$a$} \\  \cline{2-4}
& {\textcircled{\footnotesize 5} $wvb$}   &  {$\newprec{I_j}: \  1,2,a_1,\ldots,a_{2m},b_1,\ldots,b(=b_k),3,\ldots,v,\ldots,{n},b_{k+1},\ldots,b_\ell$} &{$v$}
         \\ \bottomrule
    \end{tabular}}}
\end{center}
Note that, in any case, the lexicographic ordering $\newprec{I}$ {on $V\cup B$ is given by} $$\newprec{I}: \ 1,2,3,\ldots,{n},a_1,\ldots,\ldots,a_{2m},b_1,\ldots,b_\ell,$$ and any atom of $[I,J]$ containing the element $w$ has a multiple edge.
If $\min^{I_j}(J\setminus I_j)\succ_{\mathrm{lex}}^{I_j}\max^{I_j}(I_j\setminus I)$, then $J\setminus I_j$ cannot have a multiple edge less than {$\max^I(B\cap(I_j\setminus I))$} in $\newprec{I}$, and hence  $I_j$ is the first atom of $[I,J]$ in $\atomprec{I}$.
This proves the `only if' part.

Suppose that $\min^{I_j}(J\setminus I_j)\newprec{I_j} \max^{I_j}(I_j\setminus I)$.
In cases~\textcircled{\footnotesize 1},  \textcircled{\footnotesize 2}, and \textcircled{\footnotesize 4}, $J\setminus I_j$ contains a multiple edge $a''$ with $a''\newprec{I_j} a$, which  implies that $I_\ast$ {can be} obtained from $I_j$ by replacing $a$ with $a''$. In cases~\textcircled{\footnotesize 3} and~\textcircled{\footnotesize 5}, $I_j\setminus I$ contains a multiple edge ${e}\newprec{I_j}b$.
If ${e}\in A$, then $J\setminus I_j$ is ${e}c$ for some $c\in A$, and hence $I_\ast$'s are $I\cup w{e}c$ and $I\cup w{e}$, repsectively.
If ${e}\not\in A$, then $I_\ast$'s are $I\cup w{e}$ and $I\cup wv{e}$, repsectively. This proves the `if' part.
\end{proof}

We remark that \textbf{(1)}$\sim$\textbf{(5)} of the proof  above  are useful to figure out  {which chains of $\pP{G,A}$ are falling}, which will be discussed in the next section.

\section{Falling chains  and the order complex of  a poset}\label{sec:falling}
Throughout the section, for a graph  in Figure~\ref{fig:list of possible graphs} and its admissible collection $A$,  the labeling of the vertices follows the  way  in Figure~\ref{fig:labeling of vertices-1x2x}, and so the labels of the endpoints of the bundle {are} changed  according to $A$.
We always let $V$ and $B$ be the  {vertex set and the bundle} of~$G$, respectively.

Recall that if a bounded poset $\mathcal{P}$ admits a recursive atom ordering,  then we can find the CL-labeling~$\rho$ as  in the sketch of the proof of Theorem~\ref{thm:CL-shellable-recursive-atom}. Furthermore  the $i$th reduced Betti number of the order complex $\Delta(\overline{\mathcal{P}})$ equals the number of falling chains of length $i+2$ from Theorem~\ref{thm:falling}.
For a graph $G\in\mathcal{G}^\ast$,
if $G$ is  simple, then the homotopy type of $\Delta(\overline{\mathcal{P}_{G}^{\mathrm{even}}})$ is already known  {in~\cite{CP}.}
If $G$ is a graph in Figure~\ref{fig:list of possible graphs}, then
as we seen in Section~\ref{sec:CL-shellable}, the order $\atomprec{I}$ in Definition~\ref{def:atomprec} gives a recursive atom ordering of $\pP{G,A}$ for every $A\in\mathcal{A}(G)$, and so we can determine the homotopy type of  $\Delta(\overline{\mathcal{P}_{G,A}^{\mathrm{even}}})$
by considering the CL-labeling $\rho$ obtained from  the recursive atom order on $\pP{G,A}$.

\begin{proposition}\label{cor:falling:chain:lem}
Let  {$(G,A)$ be a pair of a graph $G$ and its admissible collection $A$ illustrated in Figure~\ref{fig:labeling of vertices-1x2x}.}
Then  $\pP{G,A}$ has a falling chain  if and only if one of following holds:
(a) $2$ and $3$ are adjacent;
(b) $|V|$ is even.
\end{proposition}

\begin{proof}
We show the `only if' part first.
Let $\sigma:I_0\lessdot I_1\lessdot \cdots \lessdot I_{p}$ be a falling chain of $\pP{G,A}$, and $I_{k-1}\lessdot I_k$ be the cover such that $I_k\setminus I_{k-1}$ contains the vertex $1$.
Suppose that the vertices $2$ and $3$ are not adjacent and $|V|$ is odd.
Then $A\cap V\neq V$ and $1\not\in A$.
Note that there is no cover $I\lessdot J$ such that $J\setminus I$ contains $\{1,2\}$, and therefore, $2\not\in I_k\setminus I_{k-1}$.
Since  $1\not\in I_{k-1}$ and (a) fails, it follows that $2\not\in I_{k-1}$ and so $2\not\in I_k$.
Thus, $I_k=I_{k-1}\cup 1$ and so $I_{k-1}\lessdot I_k$ is the first atom of $[I_{k-1},G]$, which implies that $\sigma$ cannot be a falling chain, a contradiction.

To show the `if' part, recall that $A=\{a_1,\ldots,a_m\}$ and $B\setminus A=\{b_1,b_2,\ldots, b_{\ell}\}$ as long as $B\setminus A\neq\emptyset$.
We fix a falling chain $\sigma_E$ of an interval  {of $\pP{G,A}$} as follows.
If $B\setminus A\neq \emptyset$, then let $\sigma_E$ be a falling chain of $[V\cup b_{\ell}, G]$ defined by
\[ V\cup b_{\ell} \ \lessdot  \ V\cup b_{\ell-1}b_{\ell} \ \lessdot  \ \cdots \ \lessdot  \ V\cup b_1\cdots b_{\ell} \ \lessdot  \ \cdots \ \lessdot  \ V\cup a_{m-1}a_{m}b_1\cdots b_{\ell}  \ \lessdot  \ \cdots \ \lessdot  \
   V\cup a_1\cdots b_{\ell}=G. \]
If  $B\subset A$, then let $\sigma_E$ be a falling chain of $[V\cup a_{m-1}a_m, G]$ defined by
\[ V\cup a_{m-1}a_{m} \ \lessdot  \ V\cup a_{m-3}a_{m-2}a_{m-1}a_{m}\ \lessdot  \ \cdots  \ \cdots \  \lessdot  \ \cdots \ \lessdot  \
   V\cup a_1\cdots a_{m}=G. \]
Note that there is a falling chain $\sigma_\ast$ of $[\emptyset,I]$, where \[I=\begin{cases}
  34\cdots n & \text{if }n \text{ is even};\\
   4\cdots n & \text{if }n \text{ is odd and }n\ge 5; \\
  \emptyset & \text{if }n= 3. \\
\end{cases}\]
Suppose that (a) or (b) is true. We will show that $\sigma_\ast$ and $\sigma_E$ are extended to a falling chain of $\pP{G,A}$.
When $\{1,2\}\subset A$,  {the chain
$\sigma$ obtained by concatenating $\sigma_\ast$, $\sigma_0$ and $\sigma_E$ is a falling chain of $\pP{G,A}$,} where
\[ \sigma_0=\begin{cases}
 I\lessdot I\cup 12b_{\ell} &\text{ if }B\setminus A\neq\emptyset, \\
 I\lessdot I\cup 12a_{m-1}a_m &\text{ if }B\subset A.
\end{cases} \]
When $A\cap \{1,2\}=\emptyset$, {the chain
$\sigma$ obtained by concatenating $\sigma_\ast$, $ I\lessdot I\cup 2\lessdot I\cup 1b_{\ell}$ and $\sigma_E$ is a falling chain of $\pP{G,A}$.}
If $A\cap \{1,2\}=\{2\}$, then (a) is true, and hence  {the chain
$\sigma$ obtained by concatenating $\sigma_\ast$, $\sigma_0$ and $\sigma_E$ is a falling chain of $\pP{G,A}$,} where
\[ \sigma_0=\begin{cases}
 I\lessdot I\cup 23 \lessdot 1  b_{\ell}  &\text{ if }B\setminus A\neq\emptyset, \\
 I\lessdot I\cup 23 \lessdot 1 a_{m-1}a_m &\text{ if }B\subset A.
\end{cases}\]
\end{proof}

\begin{proposition}\label{prop:chain:falling}
Let  {$(G,A)$ be a pair of a graph $G$ and its admissible collection $A$ illustrated in Figure~\ref{fig:labeling of vertices-1x2x}.}
Let $\sigma:I_0\lessdot I_1\lessdot \cdots \lessdot I_{p+1}$ be a falling chain of $\pP{G,A}$, and $I_{k-1}\lessdot I_k$ be the cover such that $I_k\setminus I_{k-1}$ contains the vertex $1$.
Then the length $\ell(\sigma)$ of $\sigma$ and the set $I_k\setminus I_{k-1}$ are one of the forms in Table~\ref{table:lengths}.
\end{proposition}

\renewcommand{\arraystretch}{1.5}
\begin{table}[h]
\centering
   \footnotesize{ \begin{tabular}{c|c||c|c|c}
    \toprule
      \multicolumn{2}{c||}{\multirow{2}{*}{$\ell(\sigma)$,\quad $I_k\setminus I_{k-1}$ }}
   &\multicolumn{2}{c|}{$|V|$ is even} & \multirowcell{2}{\shortstack{$|V|$ is odd (2 and 3 are \\ adjacent by Proposition~\ref{cor:falling:chain:lem})}} \\ \cline{3-4}
        \multicolumn{1}{c}{}  &  \multicolumn{1}{c||}{} & $\tilde{P}_{n,m}$, $\tilde{S}_{n,m}$, $\tilde{T}_{n,m}$ & $\tilde{P}'_{n,m}$, $\tilde{S}'_{n,m}$, $\tilde{T}'_{n,m}$  &  \multicolumn{1}{c}{}  \\       \hline \hline
        \multirow{2}{*}{$A\cap V=V$}& \multicolumn{1}{c||}{$B\setminus A\neq\emptyset$} & \multicolumn{1}{c|}{$\frac{|A|}{2}+|B\setminus A|-1$,\quad $12b$ }
        & \multicolumn{1}{c|}{$\frac{|A|}{2}+|B\setminus A|-1$, \quad $12b$ or $1vb$ } &  \multirow{2}{*}{N/A} \\  \cline{2-4}
          \multicolumn{1}{c}{} &  \multicolumn{1}{|c||}{$B\subset A$} & \multicolumn{1}{c|}{$\frac{|A|}{2}-1$, \quad $12aa'$ }
        & \multicolumn{1}{c|}{$\frac{|A|}{2}-1$ or $\frac{|A|}{2}$, \quad $12aa'$ or $1a$ } &  \multicolumn{1}{c}{} \\ \hline
               \multirowcell{2}{\shortstack{$A\cap V\neq V$\\($1\not\in A$)}}& \multicolumn{1}{c||}{$B\setminus A\neq\emptyset$} & \multicolumn{2}{c|}{$\frac{|A|}{2}+|B\setminus A|+1$, \quad$1b$ }
        & \multicolumn{1}{c}{$\frac{|A|}{2}+|B\setminus A|$,\quad $1b$ } \\  \cline{2-5}   \multicolumn{1}{c}{}  & \multicolumn{1}{|c||}{$B\subset A$} & \multicolumn{2}{c|}{$\frac{|A|}{2}+1$,\quad $1aa'$ or $1av$ }
        & \multicolumn{1}{c}{$\frac{|A|}{2}$, \quad $1aa'$ or $1av$} \\
      \bottomrule
    \end{tabular}}
    \captionsetup{width=1.0\linewidth}\\[1ex]
    \caption{Each cell represents $\ell(\sigma)$ and $I_k\setminus I_{k-1}$, where $a,a'\in A\cap B$, $b\in B\setminus A$, $v\in V\setminus\{1,2\}$.}\label{table:lengths} \vspace{-0.5cm}
\end{table}

\begin{proof} First assume that $A\cap V\neq V$. Then $1\not\in A$ by the way of labeling.
If $I_k=I_{k-1}\cup 1$, then since $I_{k-1}\cup 1$ is the first atom of $[I_{k-1},G]$, {then the chain} $\sigma$ cannot be a falling chain. {Hence $I_k\neq I_{k-1}\cup 1$ and $2\in I_{k-1}$. Then the vertices $2$ and $3$ are adjacent.} 
By Proposition~\ref{prop:length:1x}, if $|V|$ is even and $A\cap V\neq V$, then $\pP{G,A}$ is pure and $\ell(\sigma)=\frac{|A|}{2}+|B\setminus A|+1$; if $|V|$ is odd, then {$\sigma$ is not a longest chain} and so $\ell(\sigma)=\frac{|A|}{2}+|B\setminus A|$.
If $B\setminus A\neq\emptyset$, then for the first {index} $q$ such that $I_{q}\setminus I_{q-1}$ containing an element $b$ in $B\setminus A$, for $I_{q-2}\lessdot I_{q-1}\lessdot I_q$ being falling, it follows that $k=q$, that is, $I_k\setminus I_{k-1}= 1b$.
If $B\subset A$, then $I_k\setminus I_{k-1}= 1aa'$ or $1av$ for some $a,a'\in B\cap A$ and $v\in V\setminus\{1,2\}$.

Now assume that $A\cap V=V$. Then  $|V|$ is even.
If $2\not\in I_k$ then $I_k=I_{k-1}\cup 1v$ for some $v\in V\setminus\{1,2\}$, and so $\sigma$ cannot be a falling chain, since $I_{k-1}\cup 1v$ is the first atom of $[I_{k-1},G]$.
Thus $2\in I_{k}$.

Suppose that $B\setminus A\neq\emptyset$. Then for the first index $q$ such that $I_{q}\setminus I_{q-1}$ containing an element $b$ in $B\setminus A$, for $I_{q-2}\lessdot I_{q-1}\lessdot I_q$ being falling of $[I_{q-2}, I_q]$, it follows that $k=q$, that is, $b\in I_k$.
Then \[I_k\setminus I_{k-1}=\begin{cases}
  1vb  & \text{if }2\in I_{k-1};\\
  12b & \text{if }2\in I_{k}\setminus I_{k-1}.
  \end{cases}\]
Thus, {$\sigma$ is not a longest chain} and so $\ell(\sigma)= \frac{|A|}{2}+|B\setminus A|-1$
by Proposition~\ref{prop:length:1x}.

Suppose that $B\subset A$.
Suppose that the vertices $2$ and $3$ are not adjacent.
Then $2\not\in I_{k-1}$ and so $2\in I_{k-1}$.
Thus, $I_k\setminus I_{k-1}=12aa'$ for some $a,a'\in B$, which implies that
{$\sigma$ is not a longest chain.}
Hence $\ell(\sigma)= \frac{|A|}{2}+|B\setminus A|-1$
by Proposition~\ref{prop:length:1x}.
If the vertices $2$ and $3$ are adjacent, then
by Proposition~\ref{prop:length:1x},  $\ell(\sigma)=\frac{|A|}{2}-1$ or $\frac{|A|}{2}$, and it holds that
 \[I_k\setminus I_{k-1}=\begin{cases}
  1a  & \text{if }2\in I_{k-1};\\
  12aa' & \text{if }2\in I_{k}\setminus I_{k-1}.
  \end{cases}\]
\end{proof}

By Theorem~\ref{thm:falling} and Propositions~\ref{cor:falling:chain:lem}~and~\ref{prop:chain:falling}, the following hold:
\begin{corollary}\label{cor:final:dim}
Let {$(G,A)$ be a pair of a graph $G$ and its admissible collection $A$ illustrated in Figure~\ref{fig:labeling of vertices-1x2x}.}
If neither (a) nor (b) of Proposition~\ref{cor:falling:chain:lem} holds, then
 $\Delta(\overline{\mathcal{P}_{G,A}^{\mathrm{even}}})$ is contractible.
 {If not,} the order complex $\Delta(\overline{\mathcal{P}_{G,A}^{\mathrm{even}}})$ has the homotopy type of a wedge of spheres of dimensions
\[\begin{cases}
 \frac{|A|}{2}+ |B\setminus A|-2, & \text{if }|V|\text{ is odd}\\
 \frac{|A|}{2}+ |B\setminus A|-1, & \text{if }|V|\text{ is even and }A\cap V\neq V;\\
 \frac{|A|}{2}-1 \text{ or } \frac{|A|}{2} & \text{if } A=V\cup B \text{ and the vertices 2 and 3 are adjacent};\\
 \frac{|A|}{2}+|B\setminus A|-3,& \text{otherwise}.\end{cases} \]
\end{corollary}

\begin{example}\label{ex:falling}
     See the posets $\mathcal{P}_{G,A}^{\mathrm{even}}$ in Figure~\ref{fig:examples of even posets}. The posets in (i) and (iii) are nonpure but none of {the} longest maximal chains of (i) and (iii) are falling chains. In (i), (ii), and (iii), there are four, three, and four falling chains, respectively:
{\small
$$\begin{array}{lll}
            \hspace{-.2cm}(i)&\emptyset<\boldsymbol{23}
            <\boldsymbol{1}{23}\boldsymbol{{b_1}}
            <{123}\boldsymbol{4a_1}{b_1}
            < {123}\boldsymbol{4}{5a_1}\boldsymbol{a_2}{b_1}
            &\emptyset<\boldsymbol{23}
            <\boldsymbol{1}{23}\boldsymbol{b_1}
            <{123}\boldsymbol{4a_2}{b_1}
            < {1234}\boldsymbol{5a_1}{a_2b_1}\\
            &\emptyset<\boldsymbol{34}
            <\boldsymbol{2}{34}\boldsymbol{5}<\boldsymbol{1}2345\boldsymbol{b_1}
            <12345\boldsymbol{a_1a_2}b_1
            &\emptyset<\boldsymbol{45}<\boldsymbol{23}45
            <\boldsymbol{1}2345\boldsymbol{b_1}
            <12345\boldsymbol{a_1a_2}b_1\\
            \hspace{-.2cm}(ii)&\emptyset<\boldsymbol{2}<\boldsymbol{1} 2 \boldsymbol{b_2}
            <12\boldsymbol{b_1}b_2
            <12\boldsymbol{3a_1}b_1b_2
            <123\boldsymbol{4}a_1\boldsymbol{a_2}b_1b_2
            &\emptyset<\boldsymbol{2}<\boldsymbol{1}2\boldsymbol{b_2}
            <12\boldsymbol{b_1}b_2
            <12\boldsymbol{3a_2}b_1b_2
            <123\boldsymbol{4a_1}a_2b_1b_2 \\
            &\emptyset<\boldsymbol{34}<\boldsymbol{2}34
            <\boldsymbol{1}234\boldsymbol{b_2}
            <1234\boldsymbol{b_1} b_2
            <1234\boldsymbol{a_1a_2}b_1b_2&\\
            \hspace{-.2cm}(iii)& \emptyset<\boldsymbol{12b_2}
            <12\boldsymbol{a_1a_2}b_2
            <12a_1a_2\boldsymbol{b_1}b_2
            <12\boldsymbol{34}a_1a_2b_1b_2
            &\emptyset<\boldsymbol{34}
            <\boldsymbol{12}34\boldsymbol{b_2}
            <1234\boldsymbol{b_1}b_2
            <1234\boldsymbol{a_1a_2}b_1b_2\\
            & \emptyset<\boldsymbol{12b_2}
            <12\boldsymbol{b_1}b_2
            <12\boldsymbol{3a_1}b_1b_2
            <123\boldsymbol{4}a_1\boldsymbol{a_2}b_1b_2
            & \emptyset<\boldsymbol{12b_2}
            <12\boldsymbol{b_1}b_2
            <12\boldsymbol{3a_2}b_1b_2
            <123\boldsymbol{4a_1}a_2b_1b_2
        \end{array}$$ }Hence the order complexes $\Delta(\overline{\pP{G,A}})$  for the posets $\pP{G,A}$ in Figure~\ref{fig:examples of even posets} are homotopy equivalent to {\tiny{$\displaystyle\bigvee_{4}$}}$S^2$, {\tiny{$\displaystyle\bigvee_{3}$}}$S^3$, and {\tiny{$\displaystyle\bigvee_{4}$}}$S^2$, respecively.
\end{example}

In the rest of  the section, we consider the graph $H=\tilde{P}_{n,m}$ in Figure~\ref{fig:list of possible graphs}. Let $V=\{1,2,\ldots,n\}$ be the set of vertices and $B=\{a_1,\ldots,a_m\}$ be the {unique bundle} of $H$.
Recall that we follow the labeling of the vertices shown in Figure~\ref{fig:labeling of vertices-1x2x}.  Note that $A=V\cup B$ or $(V\cup B)\setminus 1$ or $(V\cup B)\setminus \{1,2\}$.

Let $A\in \mathcal{A}(H)$ {such that $B\subset A$.} Let $\sigma:I_0\lessdot I_1\lessdot \cdots \lessdot I_{p+1}$ be a falling chain of $\pP{H,A}$, and $I_{k-1}\lessdot I_k$ be the cover such that $I_k\setminus I_{k-1}$ contains the vertex $1$.
Suppose that $A=V\cup B$.
Then $|V|$ is even and by the way of labeling, the vertices 2 and 3 are not adjacent.
By Proposition~\ref{prop:chain:falling}, $I_{k}\setminus I_{k-1}=12aa'$, where $a,a'\in B$, and $v\in V\setminus\{1,2\}$.
Thus the number of falling chains of $\pP{H,A}$  is equal to
\begin{equation}\label{eq:catalan_last}
\begin{split}
  &   \sum_{I\subset V\setminus\{1,2\}} \!\!\!\!\!\!\! (\#\text{ falling chains of } [I\cup 12aa',H] \text{ for some }a,a'\in B) \times(\#\text{ falling chains of }[\emptyset, I]) .\\
\end{split}
\end{equation}
Suppose that $A$  is  either $(V\cup B)\setminus 1$ or $(V\cup B)\setminus \{1,2\}$. Then $I_{k}\setminus I_{k-1}=1aa'$ or $1av$ by Proposition~\ref{prop:chain:falling}.
Thus the number of falling chains of $\pP{H,A}$ is equal to
\begin{equation}\label{eq:catalan_last last}
\begin{split}
  &   \sum_{I\subset V\setminus\{1\} \atop   2\in I} \!\!\!\!\!\!  (\#\text{ falling chains of } [I\cup 1ac,H]\text{ for some }a\in B, c\in A)\times (\#\text{ falling chains of }[\emptyset, I]) .  \end{split}\end{equation}
\begin{proposition}\label{prop:cataln2}
Let $H=\tilde{P}_{n,2}$ in Figure~\ref{fig:list of possible graphs}. For $A\in\mathcal{A}(H)$, the number of falling chains of $\pP{H,A}$ is
\[\begin{cases}
C_k&\text{if }n=2k\text{ for some }k\ge1\\
C_{k+1}-C_k&\text{if }n=2k+1\text{ for some }k\ge1\text{ and } {A\cap V}  \text{ itself induces a connected graph},\\
0 &\text{otherwise},
\end{cases}\]
where $C_k$ is the $k$th Catalan number.
\end{proposition}

\begin{proof} Let $V$ be the set of vertices and $B=\{a_1,a_2\}$ be the bundle of $H$.
First, suppose that  $A\cap V=V$. Then $A=V\cup B$ and $|V|=2k$ for some $k\ge 1$.
If $k=1$, then it is clear.
Suppose that $k\ge 2$.
Since we have only two multiple edges, from~\eqref{eq:catalan_last} the number of falling chains of $\pP{H,A}$ is
{\small
\begin{equation*}
\begin{split}
  &\sum_{q=1}^{2} (\#\text{ falling chains of } \mathcal{P}_{2q}\text{ starting with } 12a_1a_2) \times  \!\!\sum_{ \ |I|=2k-2q \atop \quad I\subset \{3,4,\ldots,2k\} } (\#\text{ falling chains of }[\emptyset, I])\\
  &\quad={C_{k-1} \!+ \!  (\#\text{ falling chains of } \mathcal{P}_{4}\text{ starting with } 12a_1a_2  ) \times \!\! \!\!  \sum_{ \ |I|=2k-4 \atop \quad I\subset V\setminus \{1,2\} }  \!\! \!\!  (\#\text{ falling chains of }[\emptyset, I])},
\end{split}
\end{equation*}
}where $\mathcal{P}_{2q}$ means the poset $\pP{H',H'}$  for  $H'=\tilde{P}_{2q,2}$, and the second summation is over the  vertices  $I$ of $\pP{H,A}$.
Since the number of
falling chains of $\mathcal{P}_{4}$ starting with $12a_1a_2$ is  only  one (see the second poset of Figure~\ref{fig:example_C_even_poset}), the number of falling chains  is $C_{k-1}+s$, where
\[{ s= \sum_{  \ |I|=2k-4 \atop \quad   I\subset V\setminus \{1,2\} } \!\!\!\! (\#\text{ falling chains of }[\emptyset, I])}.\]
Let $I\subset\{3,4,\ldots,2k\}$ be an element of $\pP{H,A}$ with $2k-4$ vertices.
Then $V\setminus I=\{1,2, v_1,v_2\}$ where $v_1<v_2$. Since each component of $I$ has an even number of vertices, $v_{1}$ is odd and $v_{2}$ is even, and so the number of falling chains of $[\emptyset, I]$ is
$C_{\frac{v_1-3}{2}}   C_{\frac{v_2-v_1-1}{2}}    C_{\frac{2n-v_{2}}{2}}$. By  a recursion of the Catalan numbers,
\begin{equation}\label{eg:catalan}
s=\sum_{v_1=3\atop v_1:\text{ odd}}^{2k-1}\! \sum_{v_2=v_1+1 \atop v_2:\text{ even}}^{2k}\!\! C_{\frac{v_1-3}{2}} C_{\frac{v_2-v_1-1}{2}}       C_{\frac{2k-v_{2}}{2}} = \sum_{v_1=3\atop v_1:\text{ odd}}^{2k-1} \!\! C_{\frac{v_1-3}{2}}  C_{\frac{2k-v_1+1}{2}} = C_k-C_{k-1}.
\end{equation}
Hence the number of falling chains  is $C_{k-1}+s=C_k$ when $n=2k$ ($k\ge 1$) and $A$ contains $V$.

Now we suppose that $A\cap V\neq V$.
Note that it follows from~\eqref{eq:catalan_last last} that there is no falling chain of $\pP{H,A}$ if $|V|$ is odd and ${A\cap V}$ does not induce a connected graph. Hence we need to consider the case where $|V|$ is even or ${A\cap V}$ induces a connected graph.
In~\eqref{eq:catalan_last last}, a falling chain of $[I\cup 1a_ic,H]$ for some $a_i\in B$ and $c\in A$ is either $I\lessdot I\cup 1a_2v \lessdot H$ ($v\in V$), or  $I\lessdot I\cup 1a_1a_2=H$. In each of the cases, it is uniquely determined.
Hence the number of falling chains is equal to $s_1+s_2$, where
\[
s_1= (\#\text{ falling chains of } [\emptyset, H\setminus (1\cup B) ]),
\qquad s_2= \sum_{|I|=|V(H)|-3\atop I\subset V\setminus\{1\}, \ 2\in I} \!\!\!\! (\#\text{ falling chains of } [\emptyset, I]).\]

{Let us compute $s_1+s_2$.}
First, suppose $|V(H)|=2k$ and $k\ge 1$.
Then $s_1$ is equal to $C_{k-1}$, the number of falling chains of $\pP{P_{2k-2}}$.
If $k=1$, then $s_2=0$ and so the number of falling chains is $C_1$ (since $C_0=C_1=1$).
Suppose that $k \ge2 $.
Let $I\subset\{2,3,\ldots,2k\}$ be an element of $\pP{H,A}$ with $2k-3$ vertices containing the vertex $2$. Then $V\setminus I=\{1,v_1,v_2\}$ where $2<v_1<v_2$. Since each component of $I$ has an even number of vertices and $2\not\in A$,
$v_{1}$ is odd and $v_{2}$ is even.
Since $s_2$ has the same equation in (\ref{eg:catalan}),  $s_1+ s_2=C_{k-1}+(C_{k}-C_{k-1})=C_k$.
Hence the number of falling chains is $C_k$ if $n=2k$.

Suppose $|V(H)|=2k+1$ and $k\ge 1$.
Then  $s_1$ is equal to $C_{k}$, the number of falling chains of $\pP{P_{2k}}$.
If $k=1$, then $s_2=1$ and so the number of falling chains is $C_2-C_1$ (note $C_2=2$ and $C_1=1$).
Suppose $k\ge 2$.
Let $I\subset\{2,3,\ldots,2k+1\}$ be an element of $\pP{H,A}$ with $2k-3$ vertices containing the vertex~$2$. Then $V\setminus I=\{1,v_1,v_2\}$, where $2<v_1<v_2$. Since each component of $I$ has an even number of vertices and $2\in A$,
$v_{1}$ is even and $v_{2}$ is odd. Thus $s_1+ s_2=C_k+(C_{k+1}-2C_k)=C_{k+1}-C_k$, since
\[\begin{split}
s_2& =\sum_{v_1=4\atop v_1:\text{ even}}^{2k} \sum_{v_2=v_1+1 \atop v_2:\text{ odd}}^{2k+1} C_{\frac{v_1-2}{2}} C_{\frac{v_2-v_1-1}{2}}  C_{\frac{2k+1-v_{2}}{2}}=  \sum_{v_1=4\atop v_1:\text{ even}}^{2k}
C_{\frac{v_1-2}{2}} C_{\frac{2k-v_1+2}{2}}= C_{k+1}-2C_{k}.
\end{split}\]
Thus the number of falling chains is $C_{k+1}-C_{k}$. It completes the proof.
\end{proof}

{From Corollary~\ref{cor:final:dim}  and Proposition~\ref{prop:cataln2}, we can compute the homotopy types of   $\Delta(\overline{\pP{H,A}})$  when $H=\tilde{P}_{n,2}$ and $A$ is an admissible collection of $H$, as in Table~\ref{table:path}.}  One may formulate the number of falling chains of $\pP{G,A}$, when $G=\tilde{P}_{n,m}$, in terms of the Catalan numbers (or the secant numbers), and it would be interesting to explain the formula by using other combinatorial objects.

\renewcommand{\arraystretch}{1.1}
\begin{table}[h]
\footnotesize{ \begin{tabular}{c|c|c|c}
\toprule
\multirow{2}{*}{$H$}&  \multirow{2}{*}{$A\in \mathcal{A}(H)$} & \multicolumn{1}{r}{$\Delta(\overline{\pP{H,A}})$} & \\ \cline{3-4}
&  & Dimension  & Homotopy Type   \\ \hline
\multirow{2}{*}{ $H=\tilde{P}_{2k,2}$}&  $V(H)\setminus A=\emptyset$    &  $\frac{|A|}{2}-3=k-2$
& {\tiny{$\displaystyle\bigvee_{C_k}^{}$}}{$S^{k-2}$ } \\ \cline{2-4}
 & $V(H)\setminus A\neq \emptyset$  & $\frac{|A|}{2}-1=k-1$ &{\tiny{$\displaystyle\bigvee_{C_{k}}^{}$}}$S^{k-1} $  \\ \hline
{$H=\tilde{P}_{2k+1,2}$}&
 $V(H)\setminus A\neq\emptyset$ & $\frac{|A|}{2}-2=k-1$ &  {\tiny{$\displaystyle\bigvee_{C_{k+1}-C_{k}}^{}$}}$S^{k-1}$   \\
\bottomrule
    \end{tabular}  }
    \captionsetup{width=1.0\linewidth}\\[1ex]
    \caption{The homotopy types of $\overline{\pP{H,A}}$ for  $A\in\mathcal{A}^\ast(H)$ and $H=\tilde{P}_{n,2}$.
    The last row of the table is true only when  ${A\cap V}$ induces a connected graph.}\vspace{-0.5cm}
    \label{table:path}
\end{table}

\section{Topology of real toric manifolds arising from graphs}\label{sec:application}

As it was noticed, the posets {$\pP{G,A}$} are appeared in \cite{CPP2015} to compute the rational cohomology of real toric manifolds arising from pseudograph associahedron. First, we present the results in~\cite{CC2017,Suciu-Trevisan,Trevisan} and summarize the results in~\cite{CP,CPP2015}, and then discuss how to compute the integral cohomology groups of
the {real toric manifold} associated with the graph~$\tilde{P}_{n,2}$.

It is well-known in toric geometry that there is a one-to-one correspondence between projective smooth toric varieties and Delzant polytopes\footnote{An $n$-dimensional convex polytope is said to be a \emph{Delzant polytope} if the (outward) normal vectors to the facets (codimension-$1$ faces) meeting at each vertex (dimension-$0$ face) form an integral basis of $\Z^n$.},  where
a \emph{toric variety} of complex dimension~$n$ is a normal algebraic variety over $\C$ with an effective action of $(\C^\ast)^n$ having an open dense orbit.
A \emph{real toric manifold} is the real locus of a compact smooth toric variety. Whereas the integral cohomology ring of a compact smooth toric variety was studied by Danilov~\cite{Danilov} and Jurkiewicz~\cite{Jurkiewicz} in the late 1970s, only little is known about the cohomology of real toric manifolds.

In 1985, the cohomology ring $H^\ast(X^\R;\Z_2)$ was computed by Jurkiewicz~\cite{Jurkiewicz2} and it is similar to the integral cohomology ring $H^\ast(X;\Z)$. Note that for an $n$-dimensional real toric manifold $X^\R$, the dimension of $H^i(X^\R;\Z_2)$ as a vector space over~$\Z_2$ is equal to $h_i$, where $(h_0,h_1,\ldots,h_n)$ is the $h$-vector of the simplicial complex $K_X$.

Recently, there were several effort to compute the integral cohomology of a real toric manifold.
Let $P$ be a Delzant polytope of dimension~$n$ and let $\mathcal{F}(P)=\{F_1,\ldots,F_m\}$ be the set of facets of $P$. Then the primitive outward normal vectors of $P$ can be understood as a function~$\phi$ from $\mathcal{F}(P)$ to $\Z^n$, and the composition map $\lambda \colon \mathcal{F}(P) \stackrel{\phi}{\rightarrow} \Z^n \stackrel{\text{mod $2$}}{\longrightarrow} \Z_2^n$ is called the (mod $2$) \emph{characteristic function} over $P$.  Note that $\lambda$ can be represented by a $\Z_2$-matrix $\Lambda_P$ of size $n \times m$ as
    $$
    \Lambda_P = \begin{pmatrix}
      \lambda(F_1) & \cdots & \lambda(F_m)
    \end{pmatrix},
    $$ where the $i$th column of $\Lambda_P$ is $\lambda(F_i) \in \Z_2^n$.
    For $\omega \in \Z_2^m$, we define $P_\omega$ to be the union of facets $F_j$ such that the $j$th entry of $\omega$ is nonzero.
    Then the following holds:
  \begin{theorem}[\cite{Suciu-Trevisan,Trevisan,CP-torsion}]\label{formula}
    Let  $P$ be a Delzant polytope of dimension $n$ and let $X^\R(P)$ be the real toric manifold associated with $P$. Then the rational cohomology group of $X^\R(P)$ is given as follows
    $$
    \dim H^i(X^\R(P);\Q) = \sum_{S\subseteq [n]} \dim \widetilde H^{i-1}(P_{\omega_S};\Q),
    $$
    where $\omega_S$ is the sum of the $k$th rows of $\Lambda_P$ for all $k\in S$.
    \end{theorem}
For $S\subset [n]$, let $K_S$ be the dual simplicial complex of $P_{\omega_S}$. Note that $K_S$ and $P_{\omega_S}$ have the same homotopy type. In general, the integral cohomology of $K_S$ has arbitrary amount of torsion and hence its computation is not easy.
Furthermore, the torsion of $H^\ast(K_S;\Z)$ influences the torsion of $H^\ast(X^\R(P);\Z)$.

\begin{theorem}[\cite{CC2017}]\label{thm:cai-choi}
    Let $X^\R(P)$ be a real toric manifold of dimension~$n$. The integral cohomology of $X^\R(P)$ is completely determined by the integral cohomology of $K_S$ and the $h$-vector of $P$, and the following statements are equivalent:
    \begin{enumerate}
        \item $H^\ast(X^\R(P);\Z)$ is either torsion-free or has only two-torsion elements.
        \item $H^\ast(K_S;\Z)$ is torsion-free for every $S\subseteq [n]$.
    \end{enumerate}
\end{theorem}

The cohomology of  real toric manifolds associated with  pseudograph associahedra have been studied in~\cite{CPP2015}. Given a graph $G$, the pseudograph associahedron $P_G$ is firstly introduced in \cite{CDF2011}, as a generalization of a graph associahedron in \cite{CD2006}. Let $G$ be a connected graph with vertex set $V$ and the bundles $B_1,\ldots,B_k$. The polytope $P_G$ is constructed from $\Delta^{|V|-1}\times\Delta^{|B_1|-1}\times\cdots\times \Delta^{|B_k|-1}$ by truncating the faces corresponding to the proper connected semi-induced subgraphs of $G$.  See Figure~\ref{fig:example_graph asso}. 
 Then the following hold.
\begin{enumerate}
        \item There is a one-to-one correspondence between the facets $F_I$ of $P_G$ and the proper connected semi-induced subgraphs $I$ of~$G$.
        \item Two facets $F_H$ and $F_{H'}$ of $P_G$ intersect if and only if $H$ and $H'$ are disjoint and cannot be connected by an edge of $G$, or one contains the other.
\end{enumerate}
If $G_1,\ldots,G_\ell$ are connected components of $G$, then $P_G=P_{G_1}\times\cdots\times P_{G_\ell}$. Note that the dimension  of $P_G$ is equal to $|V(G)|-1+\sum_{i=1}^\ell(|B_i|-1)$, where $B_i$'s are all the bundles of $G$. See \cite[Sections 2 and 3]{CPP2015}, where the reader may find examples, definitions, and a much more detailed account of results {for} pseudograph associahedra.

 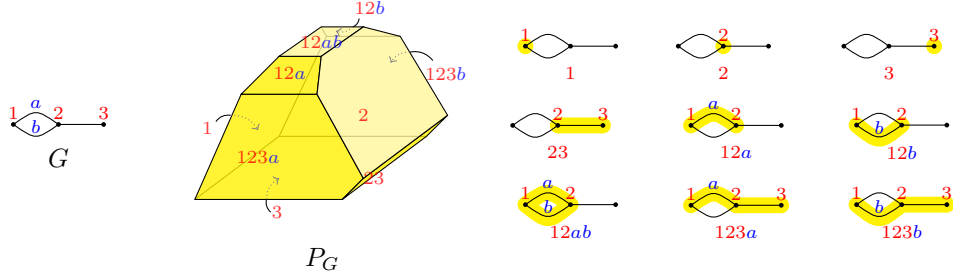
\begin{figure}[t]
 \begin{minipage}{0.2\textwidth}
     \begin{subfigure}[b]{.1\textwidth}
    \begin{tikzpicture}[scale=.3]
    	\draw (2,-1.5)--(4,-1.5);
    	\draw (0,-1.5)..controls (1,-2.2)..(2,-1.5);
    	\draw (0,-1.5)..controls (1,-0.8)..(2,-1.5);
    	\fill (0,-1.5) circle(3pt);
    	\fill (2,-1.5) circle(3pt);
    	\fill (4,-1.5) circle(3pt);
    	\draw (0,-1) node{\tiny{\red{$1$}}};
    	\draw (2,-1) node{\tiny{\red{$2$}}};
    	\draw (4,-1) node{\tiny{\red{$3$}}};
    	\draw (1,-1.6) node{\tiny{\blue{$b$}}};
    	\draw (1,-0.6) node{\tiny{\blue{$a$}}};
    	\draw (2,-3) node{\small{$G$}};
    \end{tikzpicture}
    \end{subfigure}
 \end{minipage}\hspace{-1cm}
\begin{minipage}{0.2\textwidth}
  \begin{subfigure}[b]{1\textwidth}
    	\begin{tikzpicture}[scale=.28]
    		\pgfsetfillopacity{0.8}
    		\draw (-1,0)--(3,3);
    		\draw (3,3)--(10,3);
    		\draw (3,3)--(5.25,7.75);
    		\draw (5,8.2)--(5.25,7.75)--(8.75,7.75);
    		\filldraw[fill=yellow] (-1,0)--(6,0)--(7,1)--(4.8,5)--(1.2,5)--cycle;
    		\filldraw[fill=yellow] (1.2,5)--(4.8,5)--(5,6.8)--(3,6.8)--cycle;
    		\filldraw[fill=yellow] (7,1)--(6,0)--(10,3)--(11,4)--cycle;
    		\filldraw[fill=yellow!50] (3,6.8)--(5,8.2)--(7,8.2)--(5,6.8)--cycle;
    		\filldraw[fill=yellow!50] (7,8.2)--(8.75,7.75)--(11,4)--(7,1)--(4.8,5)--(5,6.8)--cycle;
       	\draw (2,2) node{\tiny{\red{$123$\blue{$a$}}}};
       	\draw (3.5,6) node{\tiny{\red{$12$\blue{$a$}}}};
       	\draw (5,7.5) node{\tiny{\red{$12$\blue{$ab$}}}};
       	\draw (7,4) node{\tiny{\red{$2$}}};
       	\draw (7.5,1) node{\tiny{\red{$23$}}};
       \path(0.7,4) edge[draw=gray,bend left,densely dotted,->] (2,3.3);
       \path(0,3.7) edge[ bend left ] (0.7,4);
       \draw (-0.4,3.4) node{\tiny{\red{$1$}}};
       \path(2.4,0) edge[draw=gray,bend left,densely dotted,->] (3,1);
       \path(2.6,-0.6) edge[ bend left ] (2.4,0);
       \draw (2.9,-0.6) node{\tiny{\red{$3$}}};
       \path(9.2,7) edge[draw=gray,bend right,densely dotted,->] (8.3,6.5);
       \path(10.3,6.6) edge[ bend right ] (9.2,7);
       \draw (10.9,6) node{\tiny{\red{$123$\blue{$b$}}}};
              \path(6.5,8.2) edge[draw=gray,bend left,densely  dotted,->] (5.7,8);
       \path(6.8,8.6) edge[ bend left ] (6.5,8.2);
       \draw (7.3,9.1) node{\tiny{\red{$12$\blue{$b$}}}};
    	\end{tikzpicture}
    	\caption*{$P_G$}
\end{subfigure}
\end{minipage}\qquad
\begin{minipage}{0.4\textwidth}
    \begin{subfigure}[b]{.3\textwidth}
    \begin{tikzpicture}[scale=.3]
    	\fill[yellow] (0,-1.5) circle(10pt);
    	\draw (2,-1.5)--(4,-1.5);
    	\draw (0,-1.5)..controls (1,-2.2)..(2,-1.5);
    	\draw (0,-1.5)..controls (1,-0.8)..(2,-1.5);
    	\fill (0,-1.5) circle(3pt);
    	\fill (2,-1.5) circle(3pt);
    	\fill (4,-1.5) circle(3pt);
    	\draw (0,-1) node{\tiny{\red{$1$}}};
        \draw (2,-2.7) node{\tiny{\red{$1$}}};
    \end{tikzpicture}
    \end{subfigure}
    \begin{subfigure}[b]{.3\textwidth}
    \begin{tikzpicture}[scale=.3]
    	\fill[yellow] (2,-1.5) circle(10pt);
    	\draw (2,-1.5)--(4,-1.5);
    	\draw (0,-1.5)..controls (1,-2.2)..(2,-1.5);
    	\draw (0,-1.5)..controls (1,-0.8)..(2,-1.5);
    	\fill (0,-1.5) circle(3pt);
    	\fill (2,-1.5) circle(3pt);
    	\fill (4,-1.5) circle(3pt);
    	\draw (2,-1) node{\tiny{\red{$2$}}};
        \draw (2,-2.7) node{\tiny{\red{$2$}}};
    \end{tikzpicture}
    \end{subfigure}
    \begin{subfigure}[b]{.3\textwidth}
    \begin{tikzpicture}[scale=.3]
    	\fill[yellow] (4,-1.5) circle(10pt);
    	\draw (2,-1.5)--(4,-1.5);
    	\draw (0,-1.5)..controls (1,-2.2)..(2,-1.5);
    	\draw (0,-1.5)..controls (1,-0.8)..(2,-1.5);
    	\fill (0,-1.5) circle(3pt);
    	\fill (2,-1.5) circle(3pt);
    	\fill (4,-1.5) circle(3pt);
    	\draw (4,-1) node{\tiny{\red{$3$}}};
        \draw (2,-2.7) node{\tiny{$\red{3}$}};
    \end{tikzpicture}
    \end{subfigure}

        \begin{subfigure}[b]{.3\textwidth}
    \begin{tikzpicture}[scale=.3]
    	\fill[yellow] (4,-1.5) circle(10pt);
    	\fill[yellow] (2,-1.5) circle(10pt);
    	\draw[line width=6pt, color=yellow] (2,-1.5)--(4,-1.5);
    	\draw (2,-1.5)--(4,-1.5);
    	\draw (0,-1.5)..controls (1,-2.2)..(2,-1.5);
    	\draw (0,-1.5)..controls (1,-0.8)..(2,-1.5);
    	\fill (0,-1.5) circle(3pt);
    	\fill (2,-1.5) circle(3pt);
    	\fill (4,-1.5) circle(3pt);
    	\draw (2,-1) node{\tiny{\red{$2$}}};
    	\draw (4,-1) node{\tiny{\red{$3$}}};
        \draw (2,-2.7) node{\tiny{$\red{23}$}};
    \end{tikzpicture}
    \end{subfigure}
    \begin{subfigure}[b]{.3\textwidth}
    \begin{tikzpicture}[scale=.3]
    	\fill[yellow] (0,-1.5) circle(10pt);
    	\fill[yellow] (2,-1.5) circle(10pt);
    	\draw[line width=6pt, color=yellow] (0,-1.5)..controls (1,-0.8)..(2,-1.5);
    	\draw (2,-1.5)--(4,-1.5);
    	\draw (0,-1.5)..controls (1,-2.2)..(2,-1.5);
    	\draw (0,-1.5)..controls (1,-0.8)..(2,-1.5);
    	\fill (0,-1.5) circle(3pt);
    	\fill (2,-1.5) circle(3pt);
    	\fill (4,-1.5) circle(3pt);
    	\draw (0,-1) node{\tiny{\red{$1$}}};
    	\draw (2,-1) node{\tiny{\red{$2$}}};
    	\draw (1,-0.6) node{\tiny{\blue{$a$}}};
        \draw (2,-2.7) node{\tiny{$\red{12}\blue{a}$}};
    \end{tikzpicture}
    \end{subfigure}
    \begin{subfigure}[b]{.3\textwidth}
    \begin{tikzpicture}[scale=.3]
    	\fill[yellow] (0,-1.5) circle(10pt);
    	\fill[yellow] (2,-1.5) circle(10pt);
    	\draw[line width=6pt, color=yellow] (0,-1.5)..controls (1,-2.3)..(2,-1.5);
    	\draw (2,-1.5)--(4,-1.5);
    	\draw (0,-1.5)..controls (1,-2.2)..(2,-1.5);
    	\draw (0,-1.5)..controls (1,-0.8)..(2,-1.5);
    	\fill (0,-1.5) circle(3pt);
    	\fill (2,-1.5) circle(3pt);
    	\fill (4,-1.5) circle(3pt);
    	\draw (0,-1) node{\tiny{\red{$1$}}};
    	\draw (2,-1) node{\tiny{\red{$2$}}};
    	\draw (1,-1.6) node{\tiny{\blue{$b$}}};
        \draw (2,-2.7) node{\tiny{$\red{12}\blue{b}$}};
    \end{tikzpicture}
    \end{subfigure}

    \begin{subfigure}[b]{.3\textwidth}
    \begin{tikzpicture}[scale=.3]
    	\fill[yellow] (0,-1.5) circle(10pt);
    	\fill[yellow] (2,-1.5) circle(10pt);
    \draw[line width=6pt, color=yellow] (0,-1.5)..controls (1,-0.8)..(2,-1.5);
    	\draw[line width=6pt, color=yellow] (0,-1.5)..controls (1,-2.3)..(2,-1.5);
    	\draw (2,-1.5)--(4,-1.5);
    	\draw (0,-1.5)..controls (1,-2.2)..(2,-1.5);
    	\draw (0,-1.5)..controls (1,-0.8)..(2,-1.5);
    	\fill (0,-1.5) circle(3pt);
    	\fill (2,-1.5) circle(3pt);
    	\fill (4,-1.5) circle(3pt);
    	\draw (0,-1) node{\tiny{\red{$1$}}};
    	\draw (2,-1) node{\tiny{\red{$2$}}};
    	\draw (1,-1.6) node{\tiny{\blue{$b$}}};
    	\draw (1,-0.6) node{\tiny{\blue{$a$}}};
        \draw (2,-2.7) node{\tiny{$\red{12}\blue{ab}$}};
    \end{tikzpicture}
    \end{subfigure}
    \begin{subfigure}[b]{.3\textwidth}
    \begin{tikzpicture}[scale=.3]
    	\fill[yellow] (4,-1.5) circle(10pt);
    	\fill[yellow] (2,-1.5) circle(10pt);
        \fill[yellow] (0,-1.5) circle(10pt);
        \draw[line width=6pt, color=yellow] (0,-1.5)..controls (1,-0.8)..(2,-1.5);
    	\draw[line width=6pt, color=yellow] (2,-1.5)--(4,-1.5);
    	\draw (2,-1.5)--(4,-1.5);
    	\draw (0,-1.5)..controls (1,-2.2)..(2,-1.5);
    	\draw (0,-1.5)..controls (1,-0.8)..(2,-1.5);
    	\fill (0,-1.5) circle(3pt);
    	\fill (2,-1.5) circle(3pt);
    	\fill (4,-1.5) circle(3pt);
    	\draw (0,-1) node{\tiny{\red{$1$}}};
    	\draw (2,-1) node{\tiny{\red{$2$}}};
    	\draw (4,-1) node{\tiny{\red{$3$}}};
    	\draw (1,-0.6) node{\tiny{\blue{$a$}}};
        \draw (2,-2.7) node{\tiny{$\red{123}\blue{a}$}};
    \end{tikzpicture}
    \end{subfigure}
    \begin{subfigure}[b]{.3\textwidth}
    \begin{tikzpicture}[scale=.3]
    	\fill[yellow] (4,-1.5) circle(10pt);
    	\fill[yellow] (2,-1.5) circle(10pt);
        \fill[yellow] (0,-1.5) circle(10pt);
        \draw[line width=6pt, color=yellow] (0,-1.5)..controls (1,-2.3)..(2,-1.5);
    	\draw[line width=6pt, color=yellow] (2,-1.5)--(4,-1.5);
    	\draw (2,-1.5)--(4,-1.5);
    	\draw (0,-1.5)..controls (1,-2.2)..(2,-1.5);
    	\draw (0,-1.5)..controls (1,-0.8)..(2,-1.5);
    	\fill (0,-1.5) circle(3pt);
    	\fill (2,-1.5) circle(3pt);
    	\fill (4,-1.5) circle(3pt);
    	\draw (0,-1) node{\tiny{\red{$1$}}};
    	\draw (2,-1) node{\tiny{\red{$2$}}};
   	    \draw (4,-1) node{\tiny{\red{$3$}}};
    	\draw (1,-1.6) node{\tiny{\blue{$b$}}};
        \draw (2,-2.7) node{\tiny{$\red{123}\blue{b}$}};
    \end{tikzpicture}
    \end{subfigure}
    \end{minipage}
    	\caption{The facets of $P_G$ and the proper semi-induced connected subgraphs of $G$}\label{fig:example_graph asso}
    \end{figure}

It should be noted that $P_G$ can be realized as a Delzant polytope canonically. Hence under the canonical Delzant realization, there is the projective smooth toric variety $X_G$ associated with a graph~$G$. Then the real toric manifold $X^\R_G$ is defined as the subset consisting of points with real coordinates of the projective smooth toric variety associated with $G$. For example, it is known that if $G$ is the simple path graph $P_3$, then the polytope $P_G$ is a pentagon
and  $X^\R_G$ is $\#3\mathbb{R}P^2$, the connected sum of three copies of the real projective plane $\mathbb{R}P^2$.

{Using Theorem~\ref{formula}, the Betti numbers of real toric manifolds arising from pseudograph associahedra $P_G$ were formulated
in~\cite{CPP2015}, in terms of posets $\pP{H,A}$ (or $\mathcal{P}_{H,A}^{\mathrm{odd}}$), as stated in Proposition~\ref{thm:cpp}.\footnote{{
In \cite{CPP2015}, they explain their main result, Proposition~\ref{thm:cpp}, in terms of $K_{H,A}^{\mathrm{odd}}$, where $K_{H,A}^{\mathrm{odd}}$ is the dual simplicial complex of the set
$P_{H,A}^{\mathrm{odd}}$ of facets of $P_H$ corresponding semi-induced connected subgraphs are $A$-odd.}}}
For ${A}\subset\mathcal{C}_G$, let $\mathcal{P}_{H,A}^{\mathrm{odd}}$ be the poset of all $A$-odd semi-induced subgraphs, where a semi-induced subgraph $I$ is $A$-odd if for any component $I'$ of $I$, $|I'\cap A|$ is odd.
\begin{proposition}[\cite{CPP2015}]\label{thm:cpp}
    For a connected graph $G$, the $i$th Betti number of $X^\R_G$ is 
    $$
    \dim H^{i}(X^\R_G;\Q) =
        \sum_{H:\text{ PI-graph}\atop\text{ of }G}\sum_{A\in\mathcal{A}(H)} \dim \widetilde{H}^{i-1}(\Delta(\overline{\mathcal{P}_{H,A}^{\mathrm{odd}}});\Q).
    $$
\end{proposition}

For a connected graph $H$, it was also noted in \cite{CPP2015} that  $\Delta(\overline{\pP{H,A}})$ (respectively, $\Delta(\overline{\mathcal{P}_{H,A}^{\mathrm{odd}}})$) is a geometric subdivision of the simplicial complex dual to the union of the facets $F_I$ of the polytope $P_H$ such that $|I\cap A|$ is even (respectively, odd). Hence, from the Alexander duality, we have
\begin{equation}\label{eq:duality}
\tilde{H}^{i}(\Delta(\overline{\mathcal{P}_{H,A}^{\mathrm{odd}}});\Z) \cong\tilde{H}_{\dim(P_H)-i-2}(\Delta(\overline{\pP{H,A}});\Z).
\end{equation}
By Theorem~\ref{main}, for every $G\in\mathcal{G}^\ast$, each poset $\pP{H,A}$ is shellable and hence $\Delta(\overline{\pP{H,A}})$ is homotopy equivalent to a wedge of spheres. Thus $H^\ast(\Delta(\overline{\pP{H,A}});\Z)$ is torsion-free and we get the following from~\eqref{eq:duality} and Theorems~\ref{thm:cai-choi} and~\ref{thm:cpp}.
\begin{corollary}
    For a graph $G\in\mathcal{G}^\ast$, the $i$th Betti number of $X^\R_G$ is
    $$
    \beta^{i}(X^\R_G) =
        \sum_{H:\text{ PI-graph}\atop\text{ of }G}\sum_{A\in\mathcal{A}(H)} \tilde{\beta}^{i-1}(\Delta(\overline{\mathcal{P}_{H,A}^{\mathrm{odd}}})),
    $$and $H^i(X^\R_G;\Z)=\Z^{\beta^i}\oplus\Z_2^{h_i-\beta^i}$, where $\beta^i=\beta^i(X^\R_G)$ and $h_i=h_i(P_G)$.
\end{corollary}

We finish the section by explaining how to compute the Betti numbers of $X^\R_G$ when $G=\tilde{P}_{n,2}$ in Figure~\ref{fig:list of possible graphs}.
It was shown in~\cite[Theorem 2.5]{CP} that, for the simple path $P_n$ with $n$ vertices, $\Delta(\overline{\pP{P_n}})$ is homotopy equivalent to {\tiny{$\displaystyle\bigvee_{C_k}$}}$S^{k-1}$ for $n=2k$ and it is contractible for an odd integer $n$.
In addition, for any integer $n\ge 2$, we have
\begin{eqnarray}\label{eq:simle_cp}
   \beta^{i}(X^\R_{P_n}) =
   \begin{cases}
     {n \choose i}-{n\choose i-1} & \text{if } 0\le i\le \lfloor \frac{n}{2}\rfloor\\
0& \text{otherwise.}
   \end{cases}
\end{eqnarray}
For a non-simple connected graph $H=\tilde{P}_{k,2}$ ($k> 0$) in Table~\ref{table:path}, since $\dim(P_H)=|V(H)|=k$, it follows from~\eqref{eq:duality} that
\begin{equation}\label{eq:betti_P_odd}
 b^i_k:=\sum_{A\in\mathcal{A}(H)}\tilde{\beta}^{i}(\Delta(\overline{\mathcal{P}_{H,A}^{\mathrm{odd}}}))=
 \begin{cases}   C_{\frac{k}{2}}&\text{if }i=\frac{k}{2} \text{ or }\frac{k}{2}-1\text{ for even }k\\
  C_{\frac{k+1}{2}}-C_{\frac{k-1}{2}} &\text{if }i=\frac{k-1}{2}\text{ for odd }k\\
        0&\text{otherwise}.
 \end{cases}
\end{equation}
Note that for a connected graph $G$ and $(H,A)\in \mathcal{A}^*(G)$, if $H_1$ is a component of $H$ and $A_1=A\cap \cC_{H_1}$ for $A\in\mathcal{A}(H)$, then $\mathcal{P}_{H,A}^{\mathrm{odd}}$ is isomorphic to the join $\mathcal{P}_{H_1,A_1}^{\mathrm{odd}}\ast \mathcal{P}_{H_2,A_2}^{\mathrm{odd}} $, where $H_2=H\setminus H_1$ and $A_2=A\setminus A_1$, see \cite[Lemma 4.5]{CPP2015}, 
and therefore the following holds:
\begin{eqnarray}\label{eq:betti_disconnected}
    &&\tilde{\beta}^{i-1}(\Delta(\overline{\mathcal{P}_{H,C}^{\mathrm{odd}}}))= \sum_{\ell} \tilde{\beta}^{\ell}(\Delta(\overline{\mathcal{P}_{H_1,C_1}^{\mathrm{odd}}})) \times \tilde{\beta}^{i-\ell-2}(\Delta(\overline{\mathcal{P}_{H_2,C_2}^{\mathrm{odd}}})).
\end{eqnarray}

Now we are ready to explain how to compute $\beta^i(X^\R_G)$ from~\eqref{eq:betti_disconnected}, \eqref{eq:simle_cp} and \eqref{eq:betti_P_odd}, when $G=\tilde{P}_{n,2}$. Assume that $i>0$.
Let $\mathcal{H}_1$ be the set of all simple PI-graphs of $G$ and $\mathcal{H}_2$ the set of all non-simple PI-graphs of $G$. By Proposition~\ref{thm:cpp},  $\beta^i(X^\R_G)=s^i_1+s^i_2$ where
\begin{equation*}
  s^i_1 = \sum_{H\in\mathcal{H}_1} \sum_{A\in\mathcal{A}(H)} \tilde{\beta}^{i-1}(\Delta(\overline{\mathcal{P}_{H,A}^{\mathrm{odd}}})) \qquad  \text{and}\qquad s^i_2 = \sum_{H\in\mathcal{H}_2} \sum_{A\in\mathcal{A}(H)}\tilde{\beta}^{i-1}(\Delta(\overline{\mathcal{P}_{H,A}^{\mathrm{odd}}})).
\end{equation*} As $\mathcal{H}_1$ is the set of PI-graphs of the simple graph $P_n$, $s^i_1=\beta^{i}(X^\R_{P_n})$.
By Proposition~\ref{thm:cpp} and \eqref{eq:betti_disconnected},
\begin{eqnarray*} s^i_2
&=&
\sum_{m=2}^{n-2}   \sum_{A\in\mathcal{A}(\tilde{P}_{m,2})} \sum_{\ell=0}^{\lfloor\frac{m}{2}\rfloor} \tilde{\beta}^{\ell}(\Delta(\overline{\mathcal{P}_{\tilde{P}_{m,2},A}^{\mathrm{odd}}}))   \times \beta^{i-\ell-1}(M_{P_{n-m-1}}) +\sum_{m=n-1}^{n}\sum_{A\in\mathcal{A}(\tilde{P}_{m,2})} \tilde{\beta}^{i-1}(\Delta(\overline{\mathcal{P}_{\tilde{P}_{m,2},A}^{\mathrm{odd}}}))\\
&=&\sum_{\ell=0}^{{i-1}} \sum_{m=2}^{n-2}  b^{\ell}_m  \beta^{i-\ell-1}(X^\R_{P_{n-m-1}})+b^{i-1}_{n-1}+b^{i-1}_n,
\end{eqnarray*}
and we note the second summation above is valid when $n\ge 4$. Hence
\begin{eqnarray}\label{eq:final}
  \beta^i(X^\R_G) &=&
 \beta^{i}(X^\R_{P_n}) + \sum_{\ell=0}^{{i-1}} \sum_{m=2}^{n-2}  b^{\ell}_m  \beta^{i-\ell-1}(X^\R_{P_{n-m-1}})+b^{i-1}_{n-1}+b^{i-1}_n.
\end{eqnarray}
{Combining
\eqref{eq:final} with \eqref{eq:simle_cp} and \eqref{eq:betti_P_odd},
one} can completely compute  $\beta^i(X^\R_G)$ when $G=\tilde{P}_{n,2}$. {Table~\ref{table:betti  for path} shows the Betti numbers of $X^\R_{\tilde{P}_{n,2}}$} for some small  integers  $n$.
We observe a more simple formula for $\beta^i(X^\R_G)$ for some $i$.
For example, $\beta^1(X^\R_G)=n$ and  $\beta^2(X^\R_G)={n\choose 2}$. We also see that $\beta^{k}(X^\R_{\tilde{P}_{2k,2}})=\beta^{k+1}(X^\R_{\tilde{P}_{2k+1,2}})=\frac{6k}{k+2}C_k$, which is {known as} the total number of nonempty subtrees over all binary trees having $k+1$ internal vertices, see\cite[A071721]{OEIS}. It would be interesting if one  finds the exact formula of $\beta^i(X^\R_G)$ and figures out that $\beta^i(X^\R_G)$ counts other combinatorial {objects} for every~$i$.
\begin{table}[h]
\renewcommand{\arraystretch}{1.05}
\newcolumntype{C}{>{\centering\arraybackslash}p{1.5em}}
\small{ \begin{tabular}{c|C|C|C|C|C|C|C|C|C|C|C|c|c|c}
 \toprule
 \diaghead{\theadfont ABCDE}{$i$}{$n$} & 2 & 3 & 4 & 5 & 6 & 7 & 8 & 9 & 10 & 11 & 12 & 13 & 14 & 15  \\  \hline
0& 1 & 1 & 1 & 1
& 1 & 1 & 1 & 1 & 1
& 1 & 1 & 1 & 1 & 1
\\
1 & 2
& 3
& 4
& 5
& 6
& 7
& 8
& 9
& 10
& 11
& 12
& 13
& 14
& 15
\\ 2
& 1
& 2
& 6
& 10
& 15
& 21
& 28
& 36
& 45
& 55
& 66
& 78
& 91
& 105
\\ 3
& 0
& 0
& 0
& 6
& 18
& 33
& 54
& 82
& 118
& 163
& 218
& 284
& 362
& 453
\\ 4
& 0
& 0
& 0
& 0
& 0
& 18
& 56
& 110
& 192
& 310
& 473
& 691
& 975
& 1337
\\ 5
& 0
& 0
& 0
& 0
& 0
& 0
& 0
& 56
& 180
& 372
& 682
& 1155
& 1846
& 2821
\\ 6
& 0
& 0
& 0
& 0
& 0
& 0
& 0
& 0
& 0
& 180
& 594
& 1276
& 2431
& 4277
\\ 7
& 0
& 0
& 0
& 0
& 0
& 0
& 0
& 0
& 0
& 0
& 0
& 594
& 2002
& 4433
\\ 8
& 0
& 0
& 0
& 0
& 0
& 0
& 0
& 0
& 0
& 0
& 0
& 0
& 0
& 2002
\\ \bottomrule
 \end{tabular}}\\[1ex]
 \caption{The Betti numbers $\beta^i(X^\R_{\tilde{P}_{n,2}})$ for  small $n$}\label{table:betti for path}\vspace{-0.5cm}
 \end{table}

\section{Remarks}\label{sec:last}

In this paper, we characterize the family $\mathcal{G}^\ast$, that is, we find all graphs $G$ such that $\pP{H,A}$ is shellable for every $(H,A)\in\mathcal{A}^{\ast}(G)$. As the problem was motivated by the topology of  a real toric manifold associated with a graph, we could compute the Betti numbers of the one associated with  $\tilde{P}_{n,2}$.

 \begin{figure}[h]
         \begin{subfigure}[h]{.25\textwidth}\centering
    	\begin{tikzpicture}[scale=.6]
                \draw (-2,5) node{\tiny$G$};
        	\fill (-1,5) circle(3pt);
        	\fill (0,5) circle(3pt);
        	\fill (1,5) circle(3pt);
        	\fill (2,5) circle(3pt);
        	\fill (3,5) circle(3pt);
        	\draw (-1,5)--(3,5);
        	\draw (0,5)..controls (0.5,5.5)..(1,5);
        	\draw (-1,4.7) node{\tiny$3$};
        	\draw (0,4.7) node{\tiny$1$};
        	\draw (1,4.7) node{\tiny$2$};
        	\draw (2,4.7) node{\tiny$4$};
            \draw (3,4.7) node{\tiny$5$};
           	\draw (0.5,4.7) node{\tiny$a$};
        	\draw (0.5,5.6) node{\tiny$b$};
                \node [draw] (0) at (1,0) {\!\tiny$\emptyset$\!};
        	\node [draw] (1) at (-0.5,0.5) {\!\tiny$1$\!};
        	\node [draw] (2) at (2,0.5) {\!\tiny$2$\!};
        	\node [draw] (123a) at (-2,1.8) {\!\tiny$123a$\!};
        	\node [draw] (123b) at (-0.5,1.8) {\!\tiny$123b$\!};
        	\node [draw] (234c) at (1,1.8) {\!\tiny$234a$\!};
        	\node [draw] (234d) at (2.5,1.8) {\!\tiny$234b$\!};
        	\node [draw] (12ab) at (4,1.8) {\!\tiny$12ab$\!};
        	\node [draw] (G) at (1,3) {\!\tiny$G'=1234ab$\!};
            \path (0)  edge (1)  edge (2) ;
            \path (1) edge (123a) edge (123b)  edge (234c) edge (234d) edge (12ab);
                    \path (2) edge (123a) edge (123b)  edge (234c) edge (234d) edge (12ab);
            \path (G) edge (123a) edge (123b) edge (234c) edge (234d) edge (12ab) ;
         	\draw (1,-1) node{\small$\pP{G',34ab}$};
    \end{tikzpicture}
    \end{subfigure}
        \begin{subfigure}[h]{.37\textwidth}\centering
    	\begin{tikzpicture}[scale=.6]
            \node [draw] (0) at (0,-0.7) {\!\tiny$\emptyset$\!};
        	\node [draw] (2) at (-3,0.3) {\!\tiny$2$\!};
        	\node [draw] (13) at (1,0.3) {\!\tiny$13$\!};
        	\node [draw] (45) at (3,0.3) {\!\tiny$45$\!};
        	\node [draw] (12a) at (-4,1.3) {\!\tiny$12a$\!};
        	\node [draw] (12b) at (-2,1.3) {\!\tiny$12b$\!};
        	\node [draw] (1345) at (4,1.8) {\!\tiny$1345$\!};
        	\node [draw] (245) at (1.7,1.3) {\!\tiny$245$\!};
        	\node [draw] (123ab) at (-3.3,3.3) {\!\tiny$123ab$\!};
            \node [draw] (124ab) at (-5,3.3) {\!\tiny$124ab$\!};
            \node [draw] (1234a) at (-1.5,3.3) {\!\tiny$1234a$\!};
            \node [draw] (1234b) at (0.1,3.3) {\!\tiny$1234b$\!};
            \node [draw] (1245a) at (2,3.3) {\!\tiny$1245a$\!};
            \node [draw] (1245b) at (3.7,3.3) {\!\tiny$1245b$\!};
            \node [draw] (H) at (0, 4.8) {\!\tiny$G$\!};
            \path (0)  edge (2) edge (45);
            \draw[thick] (1345)..controls (5.2,4)..(H);
              \path (H) edge (123ab) edge (124ab)edge (1234a) edge (1234b) edge (1245b) edge (1245a);	
        \path (13) edge (0) edge (1345) edge (1234a) edge (123ab) edge (1234b);
        	\path (2) edge (12a) edge (12b) edge (245);
        	\path (45) edge (1345) edge (245);
        	\path (12a) edge (124ab) edge (123ab) edge (1234a) edge (1245a);
        	\path (12b) edge (124ab) edge (123ab) edge (1234b) edge (1245b);
            \path (245) edge  (1245a) edge (1245b);
         	\draw (0,-1.7) node{\small$\pP{G,1345ab}$};
    	\end{tikzpicture}
    \end{subfigure}
             \begin{subfigure}[h]{.36\textwidth}\centering
    	\begin{tikzpicture}[scale=.6]
            \node [draw] (0) at (0,-0.7) {\!\tiny$\emptyset$\!};
        	\node [draw] (1) at (-3,0.3) {\!\tiny$1$\!};
        	\node [draw] (24) at (1,0.3) {\!\tiny$24$\!};
        	\node [draw] (45) at (3,0.3) {\!\tiny$45$\!};
        	\node [draw] (12a) at (-4,1.7) {\!\tiny$12a$\!};
        	\node [draw] (12b) at (-2,1.7) {\!\tiny$12b$\!};
        	\node [draw] (145) at (1.7,1.7) {\!\tiny$145$\!};
        	\node [draw] (123ab) at (-4.8,3.3) {\!\tiny$123ab$\!};
            \node [draw] (124ab) at (-3.1,3.3) {\!\tiny$124ab$\!};
            \node [draw] (1234a) at (-1.5,3.3) {\!\tiny$1234a$\!};
            \node [draw] (1234b) at (0.1,3.3) {\!\tiny$1234b$\!};
            \node [draw] (1245a) at (2,3.3) {\!\tiny$1245a$\!};
            \node [draw] (1245b) at (3.8,3.3) {\!\tiny$1245b$\!};
            \node [draw] (H) at (0, 4.8) {\!\tiny$G$\!};
            \path (0)  edge (24) edge (2) edge (45);
          \path (H) edge (123ab) edge (124ab)edge (1234a) edge (1234b) edge (1245b) edge (1245a);	
        \path (24)  edge (124ab) edge (1234a) edge (1234b);
        	\path (2) edge (12a) edge (12b) edge (145);
        	\path (45) edge (145);
        	\path (12a) edge (124ab) edge (123ab) edge (1234a) edge (1245a);
        	\path (12b) edge (124ab) edge (123ab) edge (1234b) edge (1245b);
            \path (145) edge  (1245a) edge (1245b);
       \draw (0,-1.7) node{\small$\pP{G,2345ab}$};
    	\end{tikzpicture}
    \end{subfigure}
\captionsetup{width=1.0\linewidth}
       \caption{A graph $G\in\mathcal{G}^\ast_1$,  and three shellable posets  $\pP{G,1345ab}$, $\pP{G,2345ab}$, and $\pP{1234ab,34ab}$}
   \label{fig:problem}
    \end{figure}
As a further research, it would be also interesting to see the family $\mathcal{G}^\ast_1$ of graphs $G$ such that $\pP{G,A}$ is shellable for every $A\in\mathcal{A}(G)$. Since $G$  is  a PI-graph of itself, it is clear that $\mathcal{G}^\ast \subset \mathcal{G}^{\ast}_1$.
Here is an example to show that $\mathcal{G}^\ast $ is a proper subset of $\mathcal{G}^{\ast}_1$.
Consider a graph $G$ with five vertices and one bundle $B=\{a,b\}$ of size two in Figure~\ref{fig:problem}.
Then $\mathcal{A}(G)=\{1345ab, 2345ab\}$, and both posets $\pP{G,1345ab}$ and  $\pP{G,2345ab}$ are shellable, see Figure~\ref{fig:problem}. However, $H\not\in\mathcal{G}^{\ast}$ by Theorem~\ref{main} (For a specific reason, refer the proof of Claim~\ref{claim:neighbor} in Section~\ref{sec:List_Graphs}.).
One may find  infinitely many such graphs.

Going one step further,
we  ask to completely characterize all pairs $(G,A)$  supporting a shellable poset $\pP{G,A}$. It would be the first step to find such pairs $(G,A)$ when $G$ has exactly one bundle. For example, for a subgraph $G'=1234ab$ of the graph $G$ in Figure~\ref{fig:problem},  $\mathcal{A}(G')=\{34ab,1234ab\}$, $\pP{G',34ab }$ is shellable as in Figure~\ref{fig:problem} and $\pP{G',1234ab}$ is not by the proof of Claim~\ref{claim:neighbor}.
{If we find the answers of those questions, then it helps us to compute the Betti number of a real toric manifold $M_G$ for any graph $G$, other than the graphs in Figure~\ref{fig:list of possible graphs}.}

{In addition,
note that in \cite{CP,SS}, a simple formula for $\mu(\overline{\mathcal{P}_{G,A}^{\mathrm{even}}})$ or $\beta^i(M_G)$ is founded for other interesting simple graphs $G$ by using combinatorial arguments or generating functions.
As an issue of enumerative combinatorics, it would be interesting to find a simple formula for $\beta^i(M_G)$ or  a way of counting the falling chains of $\mathcal{P}_{G,A}^{\mathrm{even}}$ for any graph $G$ in Figure~\ref{fig:list of possible graphs}.}

\appendix\section{Proof of Lemma~\ref{lem:types}}

We present the proof of Lemma~\ref{lem:types}.

\begin{proof}[Proof of Lemma~\ref{lem:types}]
Given a cover $I\lessdot J$, the proof is based on the size of $J\setminus I$ and the intersection with $B\setminus A$. Recall the observations before Lemma~\ref{lem:types} including ($\dagger$).

If $|J\setminus I|=1$, then the element $x\in J\setminus I$ does not belong to $A$, and hence $x$ is a vertex in $\{1,2\}\setminus A$ or a multiple edges in $B\setminus A$. Thus $I\lessdot J$ is one of types (E1) and (E1${}^{\prime}$).

If $|J\setminus I|=2$, then $J\setminus I\subset A$ or $(J\setminus I)\cap A=\emptyset$. If $J\setminus I\subset A$, then it easily follows from the observations that $I\lessdot J$ is of type (E2). Assume that $J\setminus I=\{x,y\}$ and $\{x,y\}\cap A=\emptyset$. Since both $I$ and $J$ are semi-induced subgraphs of $G$, $\{x,y\}\neq\{1,2\}$ and one of $x$ and $y$ is an element of $B\setminus A$, say $x$. If $y$ is also an element of $B\setminus A$, then $I\cap B\neq\emptyset$ and $K:=J\setminus \{x\}$ satisfies $I\subsetneq K\subsetneq J$ and $(K\setminus I)\cap A=\emptyset$, a contradiction to $I\lessdot J$. Hence $y$ is a vertex in $\{1,2\}\setminus A$, and $I\lessdot J$ is of type (E2${}^{\prime}$).

Now assume that $|J\setminus I|\geq 3$.

\begin{claim}\label{claim:table1-00}
A cover $I\lessdot J$ satisfies that $1\leq|(J\setminus I)\cap V|\leq 2$, $|(J\setminus I)\cap (V\setminus\{1,2\})|\leq 1$, and $1\leq|(J\setminus I)\cap B|\leq 2$.
\end{claim}

\begin{proof}
Suppose that $J\setminus I$ consists of only elements in $B$. Then $J\setminus I$ has at least three elements $x,y,z$ in $B$. Since $|\{x,y,z\}\cap A|$ is even, one of $x,y,z$ is not in $A$, say $x$. Then $K:=J\setminus\{x\}$ satisfies that $I<K<J$ and $|K\cap A|\equiv |J\cap A|$, a contradiction to $I\lessdot J$. Hence $(J\setminus I)\cap V\neq\emptyset$.

Suppose that $J\setminus I$ consists of only vertices. Then $J\setminus I$ consists of at least three vertices. Since both $I$ and $J$ are semi-induced subgraphs of $G$, if $I\subset V$, then $\{1,2\}\not\subset J$. If $I\subseteq V$, then $J\setminus I$ is contained in $V\setminus\{1,2\}$. Hence
for the two largest vertices $v$ and $v'$ in $J\setminus I$,
$I\cup vv'$ is a semi-induced subgraph and $|\{v,v'\}\cap A|=2$,
which implies that $I\cup vv'$ is an element of $\pP{G,A}$,
a contradiction to the fact that $I\lessdot J$. Thus $(J\setminus I)\cap B\neq\emptyset$.

If $J\setminus I$ contains two vertices $v,v'$ in $V\setminus \{1,2\}$, then
$I\cup vv'$ is a semi-induced subgraph and $|vv'\cap A|=2$, a contradiction to the fact that $I\lessdot J$. Thus $J\setminus I$ has at most one vertex in $V\setminus \{1,2\}$.

Suppose that $|(J\setminus I)\cap V|\geq 3$. Then $(J\setminus I)\cap V=12v$ for some $v\in V\setminus \{1,2\}$.
If $1\not\in A$, then $I\cup 1$ is a semi-induced subgraph and  $|1 \cap A|=0$, a contradiction to the fact that $I\lessdot J$.
Thus $1\in A$ and so $2\in A$ from the way of labeling\footnote{See the first paragraph of Subsection~\ref{subsec:def:ordering}.}. Then $I\cap 1v$ is a semi-induced subgraph and  $|1v \cap A|=2$, a contradiction to the fact that $I\lessdot J$. Thus $|(J\setminus I)\cap V|\leq 2$.

If  $|(J\setminus I)\cap B|\ge 3$, then
by taking two multiple edges $a$ and $a'$ in $(J\setminus I)\cap B = (J\setminus I)\cap B\cap A$, we get a semi-induced subgraph $J\setminus \{a,a'\}$, a contradiction to the fact that $I\lessdot J$. Thus, $|(J\setminus I)\cap B|\le 2$.
\end{proof}

By Claim~\ref{claim:table1-00}, $(J\setminus I)\cap B\neq\emptyset$. Note that there is no cover $I\lessdot J$ in $\mathcal{P}_{G,A}^{\mathrm{even}}$ such that $(J\setminus I)\cap (B\cap A)\neq\emptyset$ and $(J\setminus I)\cap (B\setminus A)\neq\emptyset$. Hence exactly one of the two cases $(J\setminus I)\cap (B\cap A)\neq\emptyset$ and $(J\setminus I)\cap (B\setminus A)\neq\emptyset$ holds.

\begin{claim}\label{claim:table1-0}
Suppose that $(J\setminus I)\cap (B\cap A)\neq \emptyset$. Then $(J\setminus I)\cap \{1,2\}\neq\emptyset$. Moreover, if $|(J\setminus I)\cap \{1,2\}|=1$ then $I\lessdot J$ is of (E$3$), and if $|(J\setminus I)\cap \{1,2\}|=2$ then $I\lessdot J$ is of (E$4$).
\end{claim}

\begin{proof}
From the assumption, we have
\begin{eqnarray}
\label{eq:1}&&1\le |(J\setminus I)\cap B|=|(J\setminus I)\cap (B\cap A)|\le 2.
\end{eqnarray}

Suppose that $J\setminus I$ does not contain any element in $\{1,2\}$. Then since $J$ contains a multiple edge, $I$ contains $\{1,2\}$.
Since $|J\setminus I|\ge 3$ and $|(J\setminus I)\cap (V\setminus\{1,2\})|\le 1$ by Claim~\ref{claim:table1-00},  $J\setminus I$ has two multiple edges $a$ and $a'$.
Then $I\cup \{a,a'\}$ is a semi-induced subgraph and $|\{a,a'\}\cap A|=2$, a contradiction to the fact that $I\lessdot J$.
Hence, $J\setminus I$ contains at least one element in $\{1,2\}$.

Assume that $|(J\setminus I)\cap \{1,2\}|=1$, and let $(J\setminus I)\cap \{1,2\}=\{w\}$. By \eqref{eq:1}, there is an element $a\in (J\setminus I)\cap (B\cap A)$ and $\{1,2\}\setminus\{w\}\subset I$.  If $w\in A$ then $I\cup wa$ is a semi-induced subgraph and $|wa\cap A|=2$, a contradiction to the fact that $I\lessdot J$. Thus $w\not\in A$.
In addition, by Claim~\ref{claim:table1-00}, $|(J\setminus I)\cap (V\setminus \{1,2\})| \le 1$.
If $|(J\setminus I)\cap (V\setminus \{1,2\})| = 0$, then $(J\setminus I)\cap V=w$. Since $|J\setminus I|\ge 3$ from the assumption, $J\setminus I=waa'$ for some $a'\in B\cap A$, which implies that $I\lessdot  J$ is of  (E3).
 Suppose that $|(J\setminus I)\cap (V\setminus \{1,2\})| = 1$, say $(J\setminus I)\cap (V\setminus \{1,2\})=v$.
By ($\dagger$) noted before Lemma~\ref{lem:types} and the structure of $G$, we can see
\begin{equation*}
  \begin{cases}
    v\in\{n-1,n\}&\text{ if }V\setminus I=\{w,n-1,n\}\text{ and }G\text{ is neither }\tilde{P}_{n,m}\text{ nor }\tilde{P}'_{n,m};\\
    v=\min(V\setminus (I\cup\{1,2\}))&\text{ otherwise.}
  \end{cases}
\end{equation*}
Hence $J\setminus I=wva$ and $I\lessdot J$ is of type (E$3$).

Suppose that $|(J\setminus I)\cap \{1,2\}|=2$. It follows from Claim~\ref{claim:table1-00} that $(J\setminus I)\cap V=\{1,2\}$.
It follows from~\eqref{eq:1} that $J\setminus I=12a$ or $12aa'$ for some $a,a'\in B\cap A$.
Since $|(J\setminus I)\cap A|$ is even, $J\setminus I=12aa'$ and so $I\lessdot  J$ is of (E4).
\end{proof}

Now it remains to check for $(J\setminus I)\cap (B\setminus A)\neq\emptyset$.

\begin{claim}\label{claim:table1-1}
Suppose that $(J\setminus I)\cap (B\setminus A)\neq\emptyset$. Then $|(J\setminus I)\cap V|=2$ and $I\lessdot J$ is of (E${3}^{\prime}$).
\end{claim}

\begin{proof} From the hypothesis,
  $J\setminus I$ contains an element $b \in B\setminus A$.
Note that $|((J\setminus I)\setminus\{b\}) \cap A|=|(J\setminus I)\cap A|$ and so $|((J\setminus I)\setminus\{b\}) \cap A|$ is even. By the fact that $I\lessdot J$,
$J\setminus\{b\}$ is not a semi-induced subgraph. Thus
$J$ contains both 1 and 2  and  \begin{eqnarray}
&&\label{eq:2}(J\setminus I)\cap B=J\cap B=\{b\}.
\end{eqnarray}
Then $I$  contains at most one from $\{1,2\}$ and so $(J\setminus I)\cap V$ contains at least one vertex, say $w$, from $\{1,2\}$.

If $w\not\in A$, then $I\cap\{1,2\}\neq\emptyset$ because $I\cup\{w\}$ is not an atom of $[I,G]$. Furthermore, $w\not\in A$ if and only if $(J\setminus I)\cap V=\{w\}$. Then $K=I\cup wb$ satisfies that $I\subsetneq K\subsetneq J$ and $(K\setminus I)\cap A=\emptyset$, a contradiction to $I\lessdot J$. Hence $w\in A$.

Since $(J\setminus I)\cap A=(J\setminus I)\cap V\cap A$ and $w\in A$, $|(J\setminus I)\cap V| \ge 2$. By Claim~\ref{claim:table1-00}, we have $|(J\setminus I)\cap V| =2$. Then $(J\setminus I)\cap V=12$ or $wv$ for some $v\in V\setminus\{1,2\}$.
Together with \eqref{eq:2},
if $(J\setminus I)\cap V=12$ then $J=I\cup 12b$ and $I\lessdot J$ is of type (E$3^\prime$-2).
If $(J\setminus I)\cap V=wv$, then $I\cap\{1,2\}\neq\emptyset$ and $J=I\cup wvb$.

It remains to show that $v=\min(V\setminus (I\cup\{1,2\}))$ when $(J\setminus I)\cap V=wv$  for some $w\in \{1,2\}$. Let $v_\ast=\min(V\setminus (I\cup\{1,2\}))$.
From the structure of $G$ and by ($\dagger$),
it is easy to see $v=v_\ast$ as long as either $G=P_{n,m}$, $\tilde{P}_{n,m}$ or $v_\ast\neq n-1$.
Suppose that $n$ is odd and $G$ is one of  $\tilde{S}_{n,m}$, $\tilde{S}'_{n,m}$, $\tilde{T}_{n,m}$, and $\tilde{T}'_{n,m}$, and assume that $v_\ast=n-1$.
Hence, either $I\cup wb=123\cdots (n-2)b$ or $I\cup wb=123\cdots (n-2)nb$.
If $I\cup wb=123\cdots (n-2)b$, then $I=V\setminus \{w,n-1,n\}$.
Since $n$ is odd and $w\in A\cap\{1,2\}$, $1\not\in A$ and $w=2$ by the way of labeling. Hence $I=134\cdots(n-2)$ and $|I\cap A| =|\{3,4,\ldots,n-2\}|=n-4$, which is a contradiction to $I\in \pP{G,A}$. Therefore, $I\cup wb=123\cdots (n-2)nb$. In this case, $I\cup wb$ contains $(n-1)$ vertices and so it is clear that $v=v_\ast$.
\end{proof}

It completes the proof.
\end{proof}

\end{document}